\documentclass[10pt,reqno]{amsart}
\usepackage{amsthm, amsmath,amsfonts,amssymb,euscript,hyperref,graphics,color,slashed}
\usepackage{graphicx}
\usepackage{mathrsfs}
\usepackage{comment}
\usepackage{import}
\usepackage{tikz}
\usepackage{latexsym}
\usepackage[makeroom]{cancel}

\def\inte#1{
\displaystyle\mathop{#1\kern0pt}^\circ }

\let\al=\alpha
\let\b=\beta

\let\r=\rho

\let\f=\frac

\let\p=\psi

\let\wt=\widetilde

\def\virgp{\raise 2pt\hbox{,}}
\def\cdotpv{\raise 2pt\hbox{;}}

\def\eqdefa{\buildrel\hbox{\footnotesize def}\over =}

\def\C{\mathop{\mathbb C\kern 0pt}\nolimits}
\def\DD{\mathop{\mathbb D\kern 0pt}\nolimits}
\def\EE{\mathop{{\mathbb E \kern 0pt}}\nolimits}
\def\K{\mathop{\mathbb K\kern 0pt}\nolimits}
\def\N{\mathop{\mathbb N\kern 0pt}\nolimits}
\def\Q{\mathop{\mathbb Q\kern 0pt}\nolimits}
\def\R{\mathop{\mathbb R\kern 0pt}\nolimits}
\def\SS{\mathop{\mathbb S\kern 0pt}\nolimits}
\def\ZZ{\mathop{\mathbb Z\kern 0pt}\nolimits}
\def\TT{\mathop{\mathbb T\kern 0pt}\nolimits}
\def\P{\mathop{\mathbb P\kern 0pt}\nolimits}

\newcommand{\Z}{{\ZZ}}

\def\dv{\mbox{div}}

\def\curl{\mathop{\rm curl}\nolimits}

\def\na{\nabla}
\def\p{\partial}

\newcommand{\beq}{\begin{equation}}
\newcommand{\eeq}{\end{equation}}
\newcommand{\ben}{\begin{eqnarray}}
\newcommand{\een}{\end{eqnarray}}
\newcommand{\beno}{\begin{eqnarray*}}
\newcommand{\eeno}{\end{eqnarray*}}

\newtheorem{thm}{Theorem}[section]

\newcommand{\vv}[1]{\boldsymbol{#1}}

\def\v{v}
\def\b{b}
\def\div{\text{div}\,}
\def\curl{\text{curl}\,}

\def\Zp{{Z}_+}
\def\Zm{{Z}_-}

\def\zp{{z}_+}
\def\zm{{z}_-}
\def\zpm{{z}_{\pm}}

\newtheorem*{Main Theorem}{Main Theorem}
\newtheorem{theorem}{Theorem}[section]
\newtheorem{lemma}[theorem]{Lemma}
\newtheorem{proposition}[theorem]{Proposition}
\newtheorem{corollary}[theorem]{Corollary}
\newtheorem{definition}[theorem]{Definition}
\newtheorem{remark}[theorem]{Remark}

\numberwithin{equation}{section}

\begin{document}

\title[MHD in thin domain]{On the ideal magnetohydrodynamics in three-dimensional thin domains: well-posedness and asymptotics}

\author[Li XU]{Li Xu}
\address{LSEC, Institute of Computational Mathematics, Academy of Mathematics and Systems Science, Chinese Academy of Science\\ Beijing, China}
\email{xuliice@lsec.cc.ac.cn}


\begin{abstract}
We consider  the ideal magnetohydrodynamics (MHD) subjected to a strong magnetic field along $x_1$ direction in three-dimensional thin domains $\Omega_\delta=\R^2\times(-\delta,\delta)$  with slip boundary conditions.  It is well-known that in this situation the system will generate Alfv\'en waves. Our results are summarized as follows:

(i).\, We construct the global   solutions (Alfv\'en waves) to MHD in the thin domain $\Omega_\delta$ with $\delta>0$. In addition, the uniform energy estimates are  obtained with respected to the parameter $\delta$.

(ii). We justify the asymptotics of the MHD equations from the thin domain $\Omega_\delta$ to the plane $\R^2$. More precisely, we prove that the 3D Alfv\'en waves in $\Omega_\delta$ will converge to the Alfv\'en waves in $\R^2$ in the limit  that $\delta$ goes to zero.  This shows that Alfv\'en waves propagating along the horizontal direction of the (3D) strip are stable and can be approximated by the (2D) Alfv\'en waves when $\delta$ is sufficiently small.  Moreover, the control of the (2D) Alfv\'en waves  can be obtained from the control of (3D) Alfv\'en waves in the thin domain $\Omega_\delta$ with aid of the uniform bounds.

The proofs of main results rely on the design of the proper energy functional and the null structures of the nonlinear terms. Here the null structures means two aspects: separation of the Alfv\'en waves ($z_+$ and  $z_-$) and no bad quadratic terms $Q(\p_3z_-^h,\p_3 z_+^h)$ where $z_\pm=(z_\pm^h, z^3_\pm)$ and $Q(\p_3 z_-^h,\p_3 z_+^h)$ is the linear combination of terms $\p^{\al}\p_3z_-^h\p^{\beta}\p_3z_+^h$ with $\al,\beta\in(\Z_{\geq0})^3$.
\end{abstract}

\maketitle

\section{Introduction}
The purpose of this article is to study the global well-posdness and  asymptotics of  incompressible ideal magnetohydrodynamics(MHD) in three-dimensional(3D) thin domains $\Omega_\delta\eqdefa\R^2\times(-\delta,\delta)$ with strong background magnetic fields.  We remark that  parameter $\delta>0$ is sufficiently small. Thin domains are widely considered in the study of many problems in science, such as in solid mechanics (thin rod, plates and shells), in fluid dynamics (lubrication, meteorology problems, ocean dynamics), and in magnetohydrodynamics (wave heating in the solar and stellar atmosphere, solar tachocline, shallow-water MHD).  Most of the above problems are described by  partial differential equations(PDE) in thin domains.  We  refer to  \cite{Davidson, Mu-Ro-Ha, Ra} for more details on physics background.

In the present article, we consider the incompressible ideal MHD equations in the thin plate $\Omega_\delta$. The ideal MHD equations in thin plate (or strip) read as
\begin{equation}\label{ideal MHD}
\begin{split}
\partial_t  \vv v+ \vv v\cdot \nabla \vv v &= -\nabla p + (\na\times\vv b)\times\vv b, \ \ \text{in}\ \ \Omega_\delta\times\R^+\\
\partial_t \vv b + \vv v\cdot \nabla \vv b &=  \vv b \cdot \nabla \vv v,\\
\div \vv v &=0,\\
\div \vv b &=0,
\end{split}
\end{equation}
where $\vv b=(b^1,b^2,b^3)^T, \vv v=(v^1,v^2,v^3)^T$ and $ p$ are the magnetic field, the velocity and scalar pressure of the fluid respectively. 
The slip boundary conditions are imposed on $\vv v$ and $\vv b$ as
\beq\label{original boundary condition}
v^3|_{x_3=\pm\delta}=0,\quad b^3|_{x_3=\pm\delta}=0.
\eeq

We can write the Lorentz force term $(\na\times\vv b)\times\vv b$ in the momentum equation in a more convenient form. Indeed, we have
$$(\nabla\times\vv b)\times\vv b=\na(-\f12|\vv b|^2)+\vv b \cdot \nabla \vv b.$$
The first term $\na(-\f12|\vv b|^2)$ is called the magnetic pressure force since it is in the gradient form just as the fluid pressure does. The second term $\vv b\cdot\na\vv b=\na\cdot(\vv b\otimes\vv b)$ is the magnetic tension force, which is the only restoring source to generate Alfv\'{e}n waves. Therefore, we can use $p$ again in the place of $p+\f12|\b|^2$. The momentum equation then reads
$$
\partial_t  \vv v+ \vv v\cdot \nabla \vv v = -\nabla p + \vv b \cdot \nabla \vv b .
$$

We study the most interesting situation when a strong background magnetic field $\vv B_0$ presents (to generate Alfv\'en waves). Let $\vv B_0=|\vv B_0| \,\vv e_1$ be a uniform constant (non-vanishing) background magnetic field. The vector $\vv e_1$ is the unit vector parallel to $x_1$-axis. In this case, the small initial perturbation will generate a stable Alfv\'en waves which propagate along the background magnetic field $\vv B_0$. We shall study the global existence of the Alfv\'en waves in the thin plate and also the asymptotics of the system as the width of the strip goes to zero.

\medskip

We introduce the Els\"{a}sser variables:
$
\Zp = \v +\b$ and $\Zm = \v-\b,
$
to rewrite the MHD equations \eqref{ideal MHD} as
\begin{equation}\label{MHD in Elsasser}
\begin{split}
\partial_t  \Zp +\Zm \cdot \nabla \Zp &= -\nabla p, \ \ \text{in}\ \ \Omega_\delta\times\R^+\\
\partial_t  \Zm +\Zp \cdot \nabla \Zm &= -\nabla p,\\
\div \Zp &=0,\\
\div \Zm &=0.
\end{split}
\end{equation}
Let $\vv B_0 =|\vv B_0|(1,0,0)^T$  be a uniform background magnetic field and
$
\zp = \Zp-B_0, \ \ \zm = \Zm + B_0.
$
Then the MHD equations and boundary condition \eqref{boundary condition} can  be reformulated as
\begin{equation}\label{MHD}
\begin{split}
\partial_t  \zp +\Zm \cdot \nabla \zp &= -\nabla p,  \ \ \text{in}\ \ \Omega_\delta\times\R^+\\
\partial_t  \zm +\Zp \cdot \nabla \zm &= -\nabla p,\\
\div \zp &=0,\\
 \div \zm &=0,
\end{split}
\end{equation}
with the boundary conditions
\beq\label{boundary condition}
\zp^3|_{x_3=\pm\delta}=0,\quad\zm^3|_{x_3=\pm\delta}=0.
\eeq

Suppose $j_+=\curl z_+$ and $j_-=\curl z_-$. It is easy to see that $(j_+,j_-)$ satisfies
\begin{equation}\label{Voricity equation}
\begin{split}
\partial_t  j_+ +\Zm \cdot \nabla j_+ &= -\na z_-\wedge\na z_+,  \ \ \text{in}\ \ \Omega_\delta\times\R^+\\
\partial_t  j_- +\Zp \cdot \nabla j_- &= -\na z_+\wedge\na z_-.\\
\end{split}
\end{equation}
We remark that $j_+$ and $j_-$ are divergence free. The explicit expressions of the nonlinear terms on the righthand side (r.h.s) are
\beq\label{nonlinearity}
\na z_-\wedge\na z_+=\na z_-^k\times\p_k z_+,\quad \na z_+\wedge\na z_-=\na z_+^k\times\p_k z_-.
\eeq
Here we use the Einstein's convection: if an index appears once up and once down, it is understood to be summing over $\{1,2,3\}$.

\subsection{Short review of the problem}
We first give a  short review on the PDEs in thin domains as well as the incompressible MHD system with strong background magnetic fields.

 For the incompressible Navier-Stokes(NS) system, we refer to \cite{Ra-Sell}  on the global strong solutions and attractors of NS system for large initial data and force term in  $\Omega_\varepsilon=Q_2\times(0,\varepsilon)$ with periodic boundary conditions in the horizontal direction, where $Q_2$ is the rectangle in $\R^2$.    In \cite{If-Ra-Sell}, the authors considered the NS system in a non-flat thin domain $\Omega_\varepsilon=\{x\in\R^3: x_1,x_2\in(0,1),\ x_3\in(0,\varepsilon g(x_1,x_2))\}$ with periodic boundary condition in the horizontal direction and Navier boundary conditions on the top and bottom. They proved the existence of global solutions to NS system for the initial data whose $H^1(\Omega_\varepsilon)$ norm  is smaller than $C\varepsilon^{-\f12}$. They also verified that solutions to 3D NS converge to the 2D Navier-Stokes-like system(for $g=1$, it is exactly 2D NS), provided that the initial data converge to the 2D vector fields as $\varepsilon$ goes to zero.
For the Euler equations in   $Q_\varepsilon=\Omega\times(0,\varepsilon)$ where $\Omega$ is a  rectangle in $\R^2$, the authors in \cite{Ma-Ra-Ra}  considered the periodic condition in the horizontal direction and the slip boundary condition on the vertical direction. By assuming that the initial data is uniformly bounded in $\varepsilon$ in space $W^{2,q}(Q_\varepsilon)$ $(q>3)$, they obtained the classic solution  in the interval $(0,T(\varepsilon))$, where $T(\epsilon)\rightarrow\infty$ as $\varepsilon\rightarrow 0$. We remark that the mean value operator in the thin direction plays an important role.

For incompressible MHD with strong background magnetic fields, Bardos, Sulem and Sulem [2] first obtained the global
solutions for the ideal MHD in the H\"older space (not in the energy space). The work in \cite{Bardos} treated MHD system as 1D waves system and it relied on the convolution of the fundamental solutions.  For MHD with only strong fluid viscosity, in \cite{Lin-Xu-Zhang, Xu-Zhang} the authors proved the global existence of the small (w.r.t viscosity) solution with some admissible condition for the initial data. The work in \cite{Xu-Zhang} (also \cite{Lin-Xu-Zhang}) regarded MHD system as an anisotropic damped wave equation in the Lagrangian coordinates, that is, $\p_t^2Y-\mu\Delta\p_t Y-\p_3^2Y\sim 0$. When $\mu=0$, it holds $\p_t^2Y-\p_3^2Y\sim0$ which shows the ideal MHD system is 1D waves system. Actually, in 1942, Alfv\'en \cite{Alfven} first discovered such 1D waves (the so-called Alfv\'en waves) by linear analysis, in the incompressible ideal MHD with strong background magnetic fields. Recently, the authors in \cite{He-Xu-Yu} provided a rigorous mathematical proof for the existence, propagation and stability of the (ideal/viscous) Alfv\'en waves in the nonlinear setting.  The approach   is  inspired by   the stability of Minkowski spacetime \cite{Ch-K}. We also refer to \cite{c-l} and \cite{w-z} for alternative proofs on  the same subject.

Finally we give some comments on this short review:

(i). In \cite{If-Ra-Sell,Ma-Ra-Ra,Ra-Sell}, the estimate of the pressure $p$ can be neglected due to the Leray projection and the divergence free condition. While for the incompressible ideal MHD, to catch the propagation of the  Alfv\'en waves, the spacetime weight is introduced and thus  the estimates for the pressure is  compulsory.

(ii). In \cite{If-Ra-Sell,Ma-Ra-Ra,Ra-Sell}, the authors addressed the problem in the bounded domains and regarded the asymptotic  of the equations from 3D to 2D as a perturbation of the equations in 2D.

 \medskip

Different from the previous work, in this paper, we will consider  the effect of the shape of the thin domain $\Omega_\delta$ with $\delta>0$ on the Alfv\'en waves. Our main goals can be  concluded as follows:

(i).  We want to show that  \eqref{MHD} with boundary conditions \eqref{boundary condition}  is global well-posed in any thin domain $\Omega_\delta$ with $\delta>0$. In addition, some kind of the uniform energy estimates with respect to $\delta$  can be obtained.

(ii). We want to investigate the asymptotics of the MHD equations from the thin domain $\Omega_\delta$ to the plane $\R^2$. More precisely, we want to prove that the 3D Alfv\'en waves in $\Omega_\delta$ will converge to the Alfv\'en waves in $\R^2$ in the limit  that $\delta$ goes to zero.  It shows that when the width of the strip is very small  Alfv\'en waves propagating along the horizontal direction of the (3D) strip are stable and  can be approximated by the (2D) Alfv\'en waves.  Moreover, the control of the (2D) Alfv\'en waves  can be obtained from the control of (3D) Alfv\'en waves in the thin domain $\Omega_\delta$ via the uniform bounds.

\subsection{Difficulties, key observations and the strategies}  The main difficulty of the problem results from the shape of the thin domain $\Omega_\delta$. For instance,
 in the thin domain $\Omega_\delta$, the constant in the Sobolev imbedding inequality is related to the parameter $\delta$(see Lemma \ref{Sobolev}).  Thus to make clear the dependence of the energy functional on the parameter $\delta$, it is natural to  rewrite the MHD system  in a fixed domain $\Omega_1$ by taking the proper scaling. To do that, we
introduce the new functions defined by
\beq\label{scaling}\begin{aligned}
&z_{\pm(\delta)}^h(t,x_h,x_3)\eqdefa z_{\pm}^h(t,x_h,\delta x_3),\quad z_{\pm(\delta)}^3(t,x_h,x_3)\eqdefa \delta^{-1}z_{\pm}^3(t,x_h,\delta x_3),\\
&p_{(\delta)}(t,x_h,x_3)\eqdefa p(t,x_h,\delta x_3),\quad\text{for any}\quad x\in\Omega_1.
\end{aligned}\eeq
It is easy to check that $z_{\pm(\delta)}=(z_{\pm(\delta)}^h,z_{\pm(\delta)}^3)$ still verifies the divergence free condition.
Similarly we have
\beno
z_{\pm(\delta),0}^h(x_h,x_3)\eqdefa z_{\pm,0}^h(x_h,\delta x_3),\quad z_{\pm(\delta),0}^3(x_h,x_3)\eqdefa \delta^{-1}z_{\pm,0}^3(x_h,\delta x_3),\quad\text{for any}\quad x\in\Omega_1.
\eeno
 Then due to \eqref{MHD} and \eqref{scaling}, $(z_{+(\delta)},z_{-(\delta)},p_{(\delta)})$ satisfies the following system
\begin{equation}\label{scaling MHD}
\begin{split}
\partial_t  z_{+(\delta)} +(-\vv e_1+z_{-(\delta)})\cdot \nabla z_{+(\delta)} &= -\nabla_\delta p_{(\delta)},  \ \ \text{in}\ \ \Omega_1\times\R^+\\
\partial_t  z_{-(\delta)} +(\vv e_1+z_{+(\delta)}) \cdot \nabla z_{-(\delta)} &= -\nabla_\delta p_{(\delta)},\\
\na\cdot z_{+(\delta)} =0,\quad \na\cdot z_{-(\delta)} &=0,\\
z_{+(\delta)}|_{t=0}=z_{+(\delta),0}(x),\quad z_{-(\delta)}|_{t=0}&=z_{+(\delta),0}(x),
\end{split}
\end{equation}
where $\na_\delta=(\p_1,\p_2,\delta^{-2}\p_3)^T$. Now it is clear that the system \eqref{scaling MHD} is anisotropic and moreover the pressure $p_{(\delta)}$ satisfies the singular Laplace equation \ben\label{slaeq}-(\p_1^2+\p_2^2+\delta^{-2}\p_3^2)p_{(\delta)}=\na\cdot(z_{-(\delta)}\cdot\na z_{+(\delta)}).\een    We remark that these two properties induce the difficulties of the solvability  of the well-poedness for the original system  \eqref{MHD}.

\medskip

Now let us talk about the key observations which are crucial  to overcome the difficulties.
Thanks to the anisotropic property of the equations, by calculation, we see  that
\beno&\| \p_h^{\al_h}\p_3^lz_{\pm(\delta)}^h\|_{L^2(\Omega_1)}=\delta^{l-\f12}\| \p_h^{\al_h}\p_3^lz_{\pm}^h\|_{L^2(\Omega_\delta)},\quad
\| \p_h^{\al_h}\p_3^lz_{\pm(\delta)}^3\|_{L^2(\Omega_1)}=\delta^{l-\f32}\|\ \p_h^{\al_h}\p_3^lz_{\pm}^3\|_{L^2(\Omega_\delta)},
 \eeno
which give the hints on the construction of the energy functional. In fact, we will introduce  \ben\label{basicenergy}\sum_{+,-} [\delta^{2(l-\f12)}\|\langle x_1\mp t\rangle^{1+\sigma}\na_h^{k}\p_3^lz_{\pm}(t)\|_{L^2(\Omega_\delta)}^2+ \delta^{-3}\|\langle x_1\mp t\rangle^{1+\sigma}\na_h^{k}z^3_{\pm}(t)\|_{L^2(\Omega_\delta)}^2]\een
in the total energy(see Theorem \ref{global existence in thin domain}) which is  compatible with the anisotropic property of the system.

To prove the propagation of Alfv\'en waves, we will use two kinds of the null structures inside the system. Similar to \cite{He-Xu-Yu}, in a fixed strip, we still have the separation property of Alfv\'en waves, since the background magnetic filed $\vv B_0=|\vv B_0|\vv e_1$ parallels to the strip. To implement the idea, the estimate of the pressure is compulsory since the spacetime weight $x_1\mp t$ is involved in the energy functional. Recalling that in \eqref{scaling MHD}, the pressure $p_{(\delta)}$ verifies \eqref{slaeq}.  Thus the first challenge is to give an explicit expression for the pressure in the thin domain $\Omega_\delta$. By construction of the Green function,
    we successfully obtain the explicit formula for the pressure and moreover get the upper bounds for the Green function(see Lemma \ref{equation for pressure} and Corollary \ref{derivative of Green cor}).  It is not surprise that the upper bounds for the Green function contains the singular factor $\delta^{-1}$ because of \eqref{slaeq}. But it  will bring the trouble to close the energy estimates if the total energy functional only contains the norms in \eqref{basicenergy}. To overcome the difficulty, we   introduce another energy $$\sum_{+,-}\delta^{2(l-\f{1}{2})}\|\langle x_1\mp t\rangle^{1+\sigma}\na_h^{k}\p_3^l\p_3z_{\pm,0}\|_{L^2(\Omega_\delta)}^2$$ to absorb the additional singular factor $\delta^{-1}$ coming from the pressure.
Then all the difficulties are reduced to prove the propagation of this new energy.
Our key observation lies in the second type of the null structure of the system that  there is no   linear combination of terms $\p^{\al}\p_3z_-^h\p^{\beta}\p_3z_+^h$ with $\al,\beta\in(\Z_{\geq0})^3$ in the system. Thus the propagation of the energy can be proved and then we complete the energy estimates.

\medskip

Based on  the above observations, our strategy can be concluded as:
\begin{enumerate}
\item  We first prove the global existence of the solutions (Alfv\'en waves) to the MHD system in 3D thin plates and derive some kind of the uniform energy estimates with respective to the width parameter $\delta$.
\item  With these uniform energy estimates in hand, we  consider the asymptotics of the equations from $\Omega_\delta$ to $\R^2$. We split the proof  into two steps. In the first step,
we show that the horizontal component of the Alfv\'en waves in $\Omega_\delta$ can be approximated by their mean average in height. Then in the second step, we prove that
 the mean average of the Alfv\'en waves in horizontal converges to the 2D Alfv\'en waves as $\delta$ goes to zero if the initial Alfv\'en waves converge to the 2D Alfv\'en waves.
\end{enumerate}

\subsection{Statement of the main results} In this subsection, we will state three   results on the MHD system in the thin domain $\Omega_\delta$.
\subsubsection{Global well-posedness}
Before stating the existence result, we introduce the linear characteristic hypersurfaces
\beno\begin{aligned}
&C_{u_+}^+\eqdefa\{(t,x)\in\R^+\times\Omega_\delta\,|\, u_+=x_1-t=constant\},\\
&C_{u_+,h}^+\eqdefa\{(t,x_h)\in\R^+\times\R^2\,|\, u_+=x_1-t=constant\},
\end{aligned}\eeno
where $x=(x_h,x_3)$ and $x_h=(x_1,x_2)$.
 $C_{u_-}^-$ and $C_{u_-,h}^-$ can be defined in a similar way. Here $u_\pm=u_\pm(t,x_1)=x_1\mp t$. We remark that  the hypersurfaces
$C_{u_\pm}^\pm$ and $C_{u_\pm,h}^\pm$  are regarded as the level sets of functions $u_\pm(t,x_1)$ in $(0,t^*)\times\Omega_\delta$ and  $(0,t^*)\times\R^2$
respectively.

For given multi-index $\al_h=(\al_1,\al_2)\in(\Z_{\geq0})^2$ and $l\in\Z_{\geq0}$, we set $\p_h^{\al_h}\p_3^lf=\p_1^{\al_1}\p_2^{\al_2}\p_3^lf$ and $|\na_h^k\p_3^lf|^2=\sum_{|\al_h|=k}|\p_h^{\al_h}\p_3^lf|^2$. Then
  the energy $E_{\pm}^{(\al_h,l)}(\zpm(t))$ and flux $F_{\pm}^{(\al_h,l)}(\zpm)$ are defined as follows:
\beno\begin{aligned}
&E_{\pm}^{(\al_h,l)}(\zpm(t))=\|\langle u_\mp\rangle^{1+\sigma}\p_h^{\al_h}\p_3^l\zpm(t)\|_{L^2(\Omega_\delta)}^2,\\
&F_{\pm}^{(\al_h,l)}(\zpm)=\int_0^{t^*}\int_{\Omega_\delta}\f{\langle u_\mp\rangle^{2(1+\sigma)}}{\langle u_\pm\rangle^{1+\sigma}}|\p_h^{\al_h}\p_3^l\zpm|^2dxdt,
\end{aligned}\eeno where $\langle u\rangle=(1+|u|^2)^{\f12}.
$
We  can also give the definitions to $E_{\pm}^{(\al_h,l)}(\zpm^h(t))$, $E_{\pm}^{(\al_h,l)}(\zpm^3(t))$, $E_{\pm}^{(\al_h,l)}(j_\pm(t))$, $F_{\pm}^{(\al_h,l)}(\zpm^h)$, $F_{\pm}^{(\al_h,l)}(\zpm^3)$ and $F_{\pm}^{(\al_h,l)}(j_\pm)$, etc., in a similar way.

For simplicity, we introduce the total energy  and flux  as follows:
\beno\begin{aligned}
&E_{\pm}^{(k,l)}(\zpm)=\sup_{0\leq t\leq t^*}E_{\pm}^{(k,l)}(\zpm(t))=\sup_{0\leq t\leq t^*}\sum_{|\al_h|=k}E_{\pm}^{(\al_h,l)}(\zpm(t)),\\
&F_{\pm}^{(k,l)}(\zpm)=\sum_{|\al_h|=k}F_{\pm}^{(\al_h,l)}(\zpm).
\end{aligned}\eeno
We remark that $E_{\pm}^{(k,l)}(\zpm^h)$, $E_{\pm}^{(k,l)}(\zpm^3)$, $E_{\pm}^{(\al_h,l)}(j_\pm)$, $F_{\pm}^{(k,l)}(\zpm^h)$, $F_{\pm}^{(k,l)}(\zpm^3)$ and $F_{\pm}^{(\al_h,l)}(j_\pm)$, etc., can be defined in a similar way.

Now we are in a position to state our main results. The first result is on the global well-posedness and uniform energy estimates (with respect to $\delta$) of \eqref{MHD} in the domain $\Omega_\delta$ for any $\delta\in(0,1]$.

\begin{thm}\label{global existence in thin domain}
Let $\vv B_0 =(1,0,0)^T$ be a given background magnetic field and $N_*=2N$, $N\in \mathbb{Z}_{\geq 5}$, $\delta\in(0,1]$ and $\sigma\in(0,\f13)$. Suppose that the vector fields $(z_{+,0}(x),z_{-,0}(x))$   satisfies $\div\,z_{\pm,0}=0$ and $z_{+,0}^3|_{x_3=\pm\delta}=0$, $z_{-,0}^3|_{x_3=\pm\delta}=0$. Then there exists a constant $\varepsilon_0>0$ such that  if
\beq\label{initial on delta}\begin{aligned}
\mathcal{E}(0)&\eqdefa\sum_{+,-}\bigl[\sum_{k+l\leq N_*}\delta^{2(l-\f12)}\|\langle x_1\rangle^{1+\sigma}\na_h^{k}\p_3^lz_{\pm,0}\|_{L^2(\Omega_\delta)}^2+\sum_{k\leq N_*-1}\delta^{-3}\|\langle x_1\rangle^{1+\sigma}\na_h^{k}z^3_{\pm,0}\|_{L^2(\Omega_\delta)}^2\\
&\quad+\sum_{k+l\leq N+2}\delta^{2(l-\f{1}{2})}\|\langle x_1\rangle^{1+\sigma}\na_h^{k}\p_3^l\p_3z_{\pm,0}\|_{L^2(\Omega_\delta)}^2\bigr]\\
&\leq \varepsilon_0^2,
\end{aligned}\eeq
 the MHD system  \eqref{MHD} with boundary conditions \eqref{boundary condition} admits a unique and global smooth solution. Moreover, there holds
\beq\label{total energy estimates}\begin{aligned}
&\sum_{k+l\leq N_*}\delta^{2(l-\f12)}\bigl(E_\pm^{(k,l)}(\zpm)+F_\pm^{(k,l)}(\zpm)\bigr)+\sum_{k\leq N_*-1}\delta^{-3}\bigl(E_\pm^{(k,0)}(\zpm^3)+F_\pm^{(k,0)}(\zpm^3)\bigr)\\
&\quad+\sum_{k+l\leq N+2}\delta^{2(l-\f12)}\bigl(E_{\pm}^{(k,l)}(\p_3\zpm)+F_{\pm}^{(k,l)}(\p_3\zpm)\bigr)\\
&\leq C \mathcal{E}(0).
\end{aligned}\eeq
 In particular, the constant $\varepsilon_0$  and $C$ are independent of the parameter $\delta$.
\end{thm}

\begin{remark}
By   \eqref{scaling}, it is easy to check that
$
\sum_{k\leq N+2}\|\langle x_1\rangle^{1+\sigma}\na^{k}\p_3z_{\pm(\delta),0}^h\|_{L^2(\Omega_1)}^2\leq\delta^2\varepsilon_0^2,
$
which means that as $\delta\rightarrow 0$, the limit of $z_{\pm(\delta),0}^h$ (if there exists) is independent of the vertical variable $x_3$.
By similar argument, Theorem \ref{global existence in thin domain} implies that for any time $t$, the limit of solution $z_{\pm(\delta)}^h$ (if there exists) will be also independent of the vertical variable $x_3$. Thanks to $\div\,z_{\pm(\delta)}=0$,  the limit of $z_{\pm(\delta)}^3$ will be also independent of $x_3$.
It looks promising that  the solutions to the 3D MHD in thin domains  will converge to the solutions to the 2D MHD as  $\delta$ goes to $0$,.
\end{remark}

\subsubsection{Approximation theory}
The second  result is on the approximation of the solutions to MHD in thin domains. We first  introduce the projection $M_\delta$  from $L^2(\Omega_\delta)$ to $L^2(\R^2)$
as follows:
\beq\label{projection}
M_\delta f(x_h)\eqdefa\f{1}{2\delta}\int_{-\delta}^\delta f(x_h,x_3)dx_3.
\eeq
Then there hold
\beq\label{properties for projection}\begin{aligned}
&\na_h M_\delta f=M_\delta(\na_h f),\quad M_\delta(\p_3f)=\p_3M_\delta f=0\quad\text{for any}\quad f|_{x_3=\pm\delta}=0,\\
&M_\delta^2f=M_\delta f,\quad M_\delta(I-M_\delta)f=(I-M_\delta)M_\delta f=0.
\end{aligned}\eeq

Assume that $(\zp(t,x),\zm(t,x))$ is a  smooth solution to \eqref{MHD} with  boundary conditions \eqref{boundary condition} and initial data $\zpm(0,x)=z_{\pm,0}(x)$. We set
\beq\label{approximate solution}
\bar{z}_+^h=(M_\delta z_+^h)(t,x_h),\quad \bar{z}_-^h=(M_\delta z_-^h)(t,x_h),\quad \bar{z}_\pm=(\bar{z}_\pm^h,0),
\eeq
where  $f^h\eqdefa(f^1,f^2)^T$. We also set
\beq\label{error}
w_\pm^h=\zpm^h-\bar{z}_\pm^h,\quad w_\pm^3=z_\pm^3.
\eeq
Then by \eqref{properties for projection}, \eqref{approximate solution} and \eqref{error},  we have
\beq\label{app 1}
M_\delta w_\pm^h=0.
\eeq

Now we want to derive the equations for $\bar{z}_\pm^h$ and $w_\pm^h$. To do that, we observe that
\beno\begin{aligned}
z_\mp\cdot\na\zpm^h&=(\bar{z}_\mp+w_\mp)\cdot\na(\bar{z}_\pm^h+w_\pm^h)\\
&=\bar{z}_\mp^h\cdot\na_h\bar{z}_\pm^h+\bar{z}_\mp^h\cdot\na_h w_\pm^h+w_\mp\cdot\na\zpm^h.
\end{aligned}\eeno
Using \eqref{app 1} and $\div\,z_\pm=0$, we have
\beq\label{app 2}\begin{aligned}
&M_\delta(z_\mp\cdot\na\zpm^h)=\bar{z}_\mp^h\cdot\na_h\bar{z}_\pm^h+M_\delta(w_\mp\cdot\na\zpm^h),\\
&(I-M_\delta)(z_\mp\cdot\na\zpm^h)=\bar{z}_\mp^h\cdot\na_h w_\pm^h+(I-M_\delta)(w_\mp\cdot\na\zpm^h).
\end{aligned}\eeq
Due to \eqref{properties for projection}, $\div\,\zpm=0$ and $z_+^3|_{x_3=\pm\delta}=0$, $z_-^3|_{x_3=\pm\delta}=0$ , we get
\beq\label{app 3}
\na_h\cdot\bar{z}_\pm^h=M_\delta(\na\cdot\zpm)=0,\quad\na\cdot w_\pm=0,\quad w_+^3|_{x_3=\pm\delta}=0,\quad w_-^3|_{x_3=\pm\delta}=0.
\eeq
Thanks to \eqref{app 2} and \eqref{app 3}, we deduce from \eqref{MHD} that
\beq\label{equation for app solution}\begin{aligned}
&\p_t\bar{z}_\pm^h\mp\p_1\bar{z}_\pm^h+\bar{z}_\mp^h\cdot\na_h\bar{z}_\pm^h=-\na_h(M_\delta p)-M_\delta(w_\mp\cdot\na\zpm^h),\quad\text{in}\quad \R^2\times\R^+\\
&\na_h\cdot\bar{z}_\pm^h=0,\\
&\bar{z}_\pm^h|_{t=0}=M_\delta z_{\pm,0}^h\eqdefa\bar{z}_{\pm,0}^h.
\end{aligned}\eeq
and
\beq\label{equation for remainder}\begin{aligned}
&\p_tw_\pm^h\mp\p_1w_\pm^h+\bar{z}_\mp^h\cdot\na_hw_\pm^h=-(I-M_\delta)(\na_h p)-(I-M_\delta)(w_\mp\cdot\na\zpm^h),\quad\text{in}\quad \Omega_\delta\times\R^+\\
&w_\pm^h|_{t=0}=(I-M_\delta)z_{\pm,0}^h\eqdefa w_{\pm,0}^h.
\end{aligned}\eeq

To investigate the difference between the original system and the mean average system, we introduce the energy $E_{\pm,h}^{(\al_h)}(\bar{z}_\pm^h(t))$ and the flux $F_{\pm,h}^{(\al_h)}(\bar{z}_\pm^h)$ involving the quantities on $\R^2$ as follows:
\beno\begin{aligned}
&E_{\pm,h}^{(\al_h)}(\bar{z}_\pm^h(t))=\|\langle u_\mp\rangle^{1+\sigma}\p_h^{\al_h}\bar{z}_\pm^h(t)\|_{L^2(\R^2)}^2,\\
&F_{\pm,h}^{(\al_h)}(\bar{z}_\pm^h)=\int_0^{t^*}\int_{\R^2}\f{\langle u_\mp\rangle^{2(1+\sigma)}}{\langle u_\pm\rangle^{1+\sigma}}|\p_h^{\al_h}\bar{z}_\pm^h(t)|^2dxdt.
\end{aligned}\eeno
Then the total energy and flux  are defined by
\beno\begin{aligned}
&E_{\pm,h}^{(k)}(\bar{z}_\pm^h)=\sup_{0\leq t\leq t^*}E_{\pm,h}^{(k)}(\bar{z}_\pm^h(t))=\sup_{0\leq t\leq t^*}\sum_{|\al_h|=k}E_{\pm,h}^{(\al_h)}(\bar{z}_\pm^h(t)),\\
&F_{\pm,h}^{(k)}(\bar{z}_\pm^h)=\sum_{|\al_h|=k}F_{\pm,h}^{(\al_h)}(\bar{z}_\pm^h).
\end{aligned}\eeno
We remark that $E_{\pm,h}^{(\al_h)}(w_\pm(\cdot,x_3))$, $F_{\pm,h}^{(\al_h)}(w_\pm(\cdot,x_3))$, $E_{\pm,h}^{(k)}(w_\pm(\cdot,x_3))$, $F_{\pm,h}^{(k)}(w_\pm(\cdot,x_3))$, etc., can also be defined in a  similar way.

\begin{thm}\label{stability}
Let  $(z_+,z_-)$ be a solution obtained in Theorem \ref{global existence in thin domain} with $N_*=2N$, $N\in \Z_{\geq5}$. Suppose that  $w_\pm$ is defined by  \eqref{error}. Then there exists a constant $\varepsilon_1(\leq\varepsilon_0)$ such that if $
\mathcal{E}(0)\leq\varepsilon_1^2
$(see the definition of $\mathcal{E}(0)$ in \eqref{initial on delta}),
it holds
\beq\label{stability energy estimate}\begin{aligned}
&\sum_{+,-}\sup_{x_3\in(-\delta,\delta)}\Bigl(\sum_{k\leq N}\bigl(E_{\pm,h}^{(k)}(w_\pm^h(\cdot,x_3))+F_{\pm,h}^{(k)}(w_\pm^h(\cdot,x_3))\bigr)+\delta^{-2}\sum_{k\leq N-1}\bigl(E_{\pm,h}^{(k)}(w_\pm^3(\cdot,x_3))+F_{\pm,h}^{(k)}(w_\pm^3(\cdot,x_3))\bigr)\Bigr)\\
&\leq C\sum_{+,-}\sum_{k\leq N}\sup_{x_3\in(-\delta,\delta)}E_{\pm,h}^{(k)}(w_{\pm,0}^h(\cdot,x_3))+C\delta^2\varepsilon_1^4.
\end{aligned}\eeq
\end{thm}
\begin{remark} Roughly speaking,
the theorem implies that if $z_\pm^h$ is close to the mean average  $M_\delta z_\pm^h$ in height at the initial time, then   $z_\pm^h$ will keep close to the mean average $M_\delta z_\pm^h$  in height for all time. In other words, the horizontal part of the solutions to MHD in 3D thin domains can be approximated by the mean average in height while the vertical part of the solution is close to zero.
\end{remark}

\subsubsection{Asymptotics  from 3D MHD to 2D MHD}
Based on the first two results, we are in a position to investigate the asymptotics from the 3D MHD system in thin domains $\Omega_\delta$ to the 2D MHD in $\R^2$ as $\delta$ goes to zero. We recall that the system in thin domain $\Omega_\delta$
can be rewritten by the following system in $\Omega_1$:
\begin{equation*}
\begin{split}
\partial_t  z_{+(\delta)} +(-\vv e_1+z_{-(\delta)})\cdot \nabla z_{+(\delta)} &= -\nabla_\delta p_{(\delta)},  \ \ \text{in}\ \ \Omega_1\times\R^+\\
\partial_t  z_{-(\delta)} +(\vv e_1+z_{+(\delta)}) \cdot \nabla z_{-(\delta)} &= -\nabla_\delta p_{(\delta)},\\
\na\cdot z_{+(\delta)} =0,\quad \na\cdot z_{-(\delta)} &=0,\\
z_{+(\delta)}|_{t=0}=z_{+(\delta),0}(x),\quad z_{-(\delta)}|_{t=0}&=z_{+(\delta),0}(x),
\end{split}
\end{equation*}
where $\na_\delta=(\p_1,\p_2,\delta^{-2}\p_3)^T$.
By setting
\beno
\bar{z}_{\pm(\delta)}^h(t,x_h)=M_1z_{\pm(\delta)}^h=M_\delta z_{\pm}^h,\quad w_{\pm(\delta)}^h=z_{\pm(\delta)}^h-\bar{z}_{\pm(\delta)}^h,\quad w_{\pm(\delta)}^3=z_{\pm(\delta)}^3.
\eeno
and  with the help of \eqref{equation for app solution} and \eqref{scaling}, we see that $(\bar{z}_{+(\delta)}^h,\bar{z}_{-(\delta)}^h,p_{(\delta)})$ satisfies
\begin{equation}\label{scaling for app system}
\begin{split}
\p_t\bar{z}_{+(\delta)}^h-\p_1\bar{z}_{+(\delta)}^h+\bar{z}_{-(\delta)}^h\cdot \nabla_h\bar{z}_{+(\delta)}^h& = -\nabla_h M_1p_{(\delta)}-M_1\bigl(w_{-(\delta)}\cdot\na z_{+(\delta)}^h\bigr),  \ \ \text{in}\ \ \R^2\times\R^+\\
\p_t\bar{z}_{-(\delta)}^h+\p_1\bar{z}_{-(\delta)}^h+\bar{z}_{+(\delta)}^h\cdot \nabla_h\bar{z}_{-(\delta)}^h& = -\nabla_h M_1p_{(\delta)}-M_1\bigl(w_{+(\delta)}\cdot\na z_{-(\delta)}^h\bigr), \\
\na_h\cdot \bar{z}_{+(\delta)}^h =0,\quad \na_h\cdot\bar{z}_{-(\delta)}^h &=0,\\
\bar{z}_{+(\delta)}^h|_{t=0}=M_1z_{+(\delta),0}^h\eqdefa\bar{z}_{+(\delta),0}^h,\quad &\bar{z}_{-(\delta)}^h|_{t=0}=M_1z_{-(\delta),0}^h\eqdefa\bar{z}_{-(\delta),0}^h.
\end{split}
\end{equation}

Before stating the result, we introduce the energy   $E_\pm^{(k,l)}(z_{\pm(\delta)})$ and flux   $F_\pm^{(k,l)}(z_{\pm(\delta)})$  on domain $(0,t^*)\times\Omega_1$ which are defined in a similar way as those on domain $(0,t^*)\times\Omega_\delta$.
  For simplicity, we also use the  notations:
\beq\label{notation}\begin{aligned}
	E_\pm^{(k)}(z_{\pm(\delta)})\eqdefa\sum_{k'+l'=k}E_\pm^{(k',l')}(z_{\pm(\delta)}),\quad
	F_\pm^{(k)}(z_{\pm(\delta)})\eqdefa\sum_{k'+l'=k}F_\pm^{(k',l')}(z_{\pm(\delta)}).
\end{aligned}\eeq
We remark that $F_\pm^{(k)}(w_{\pm(\delta)})$, $E_{\pm,h}^{(k)}(\bar{z}_{\pm(\delta)}^h)$, $F_{\pm,h}^{(k)}(\bar{z}_{\pm(\delta)}^h)$, e.t.c. can be defined in a similar way.

Before stating the main result, let us give a remark on the  energies defined for the original system \eqref{MHD} and those for the rescaled system \eqref{scaling MHD}.
 \begin{remark}
	By  \eqref{scaling},  for any $\al_h\in(\Z_{\geq0})^2,\, l\geq 0$, we have
	\beno\begin{aligned}
		&\|\langle u_\mp \rangle^{1+\sigma}\p_h^{\al_h}\p_3^lz_{\pm(\delta)}^h\|_{L^2(\Omega_1)}=\delta^{l-\f12}\|\langle u_\mp \rangle^{1+\sigma}\p_h^{\al_h}\p_3^lz_{\pm}^h\|_{L^2(\Omega_\delta)},\\
		&\|\langle u_\mp \rangle^{1+\sigma}\p_h^{\al_h}z_{\pm(\delta)}^3\|_{L^2(\Omega_1)}=\delta^{-\f32}\|\langle u_\mp \rangle^{1+\sigma}\p_h^{\al_h}z_{\pm}^3\|_{L^2(\Omega_\delta)},\\
		&\|\langle u_\mp \rangle^{1+\sigma}\p_h^{\al_h}\p_3^lz_{\pm(\delta)}^3\|_{L^2(\Omega_1)}=\delta^{l-\f32}\|\langle u_\mp \rangle^{1+\sigma}\p_h^{\al_h}\p_3^lz_{\pm}^3\|_{L^2(\Omega_\delta)}\quad\text{for}\quad l\geq1,
	\end{aligned}\eeno
	i.e.,
	\beq\label{transform 1}\begin{aligned}
		&E_\pm^{(\al_h,l)}(z_{\pm(\delta)}^h)=\delta^{2(l-\f12)}E_\pm^{(\al_h,l)}(z_{\pm}^h),\quad E_\pm^{(\al_h,0)}(z_{\pm(\delta)}^3)=\delta^{-3}E_\pm^{(\al_h,0)}(z_{\pm}^3),\\
		&E_\pm^{(\al_h,l)}(z_{\pm(\delta)}^3)=\delta^{2(l-\f32)}E_\pm^{(\al_h,l)}(z_{\pm}^3)=\delta^{2(l-\f32)}E_\pm^{(\al_h,l-1)}(\na_h\cdot z_{\pm}^h)\quad\text{for}\quad l\geq1.
	\end{aligned}\eeq
	Similarly, we also have
	\beq\label{transform 2}\begin{aligned}
		&F_\pm^{(\al_h,l)}(z_{\pm(\delta)}^h)=\delta^{2(l-\f12)}F_\pm^{(\al_h,l)}(z_{\pm}^h),\quad F_\pm^{(\al_h,0)}(z_{\pm(\delta)}^3)=\delta^{-3}F_\pm^{(\al_h,0)}(z_{\pm}^3),\\
		&F_\pm^{(\al_h,l)}(z_{\pm(\delta)}^3)=\delta^{2(l-\f32)}F_\pm^{(\al_h,l)}(z_{\pm}^3)=\delta^{2(l-\f32)}F_\pm^{(\al_h,l-1)}(\na_h\cdot z_{\pm}^h)\quad\text{for}\quad l\geq1.
	\end{aligned}\eeq
	Then thanks to \eqref{transform 1} and \eqref{transform 2}, using $\div z_\pm=0$ and \eqref{notation}, we obtain
	\beq\label{transform for energy functional}\begin{aligned}
		\mathcal{E}_{(\delta)}(t)&\eqdefa\sum_{k\leq N_*}\bigl(E_\pm^{(k)}(z_{\pm(\delta)}^h)+F_\pm^{(k)}(z_{\pm(\delta)}^h)\bigr)+\sum_{k\leq N_*-1}\bigl(E_\pm^{(k)}(z_{\pm(\delta)}^3)+F_\pm^{(k)}(z_{\pm(\delta)}^3)\bigr)\\
		&\quad+\delta^2\bigl(E_\pm^{(N_*)}(z_{\pm(\delta)}^3)+F_\pm^{(N_*)}(z_{\pm(\delta)}^3)\bigr)+\delta^{-2}\sum_{k\leq N+2}\bigl(E_{\pm}^{(k)}(\p_3z_{\pm(\delta)}^h)+F_{\pm}^{(k,l)}(\p_3z_{\pm(\delta)}^h)\bigr)\\
		\sim&\sum_{k+l\leq N_*}\delta^{2(l-\f12)}\bigl(E_\pm^{(k,l)}(\zpm)+F_\pm^{(k,l)}(\zpm)\bigr)+\sum_{k\leq N_*-1}\delta^{-3}\bigl(E_\pm^{(k,0)}(\zpm^3)+F_\pm^{(k,0)}(\zpm^3)\bigr)\\
		&\quad+\sum_{k+l\leq N+2}\delta^{2(l-\f12)}\bigl(E_{\pm}^{(k,l)}(\p_3\zpm)+F_{\pm}^{(k,l)}(\p_3\zpm)\bigr).
	\end{aligned}\eeq
\end{remark}

\begin{definition}\label{def1} We call that the sequence
 $\{z_{\pm(\delta)}^h(x)\}_{0<\delta\leq1}$  converges to $z_{\pm(0)}^h(x)$ in $H^N(\Omega_1)$ if it holds
\beq\label{convergence}
\lim_{\delta\rightarrow0}\sum_{k\leq N}\|\langle u_\mp\rangle^{1+\sigma}\na^{k}\bigl(z_{\pm(\delta)}^h-z_{\pm(0)}^h\bigr)\|_{L^2(\Omega_1)}=0.
\eeq
Similarly,  the sequence $\{z_{\pm(\delta)}^h(x_h)\}_{0<\delta\leq1}$  is said to  converge to $z_{\pm(0)}^h(x_h)$ in $H^N(\R^2)$ if it holds
\beq\label{convergence for 2D}
\lim_{\delta\rightarrow0}\sum_{k\leq N}\|\langle u_\mp\rangle^{1+\sigma}\na_h^{k}\bigl(z_{\pm(\delta)}^h-z_{\pm(0)}^h\bigr)\|_{L^2(\R^2)}=0.\eeq
The convergence of the sequence  $\{z_{\pm(\delta)}^3(\cdot)\}_{0<\delta<1}$ to $z_{\pm(0)}^3(\cdot)$ in $H^N(\Omega_1)$  and the sequence $\{z_{\pm(\delta)}^3(\cdot,x_3)\}_{0<\delta<1}$   to $z_{\pm(0)}^3(\cdot,x_3)$ (for fixed $x_3\in(-1,1)$) in $H^N(\R^2)$ can be defined in a similar way.

\end{definition}

\begin{thm}\label{limitation}
Let $N_*=2N$, $N\in \mathbb{Z}_{\geq 5}$, $\delta\in(0,1]$, $\sigma\in(0,\f13)$ and $\varepsilon_1$ be the constant in Theorem \ref{stability}.  Suppose that the vector fields sequence $\{(z_{\pm(\delta),0}^h,z_{\pm(\delta),0}^3)\}_{0<\delta\leq1}$ satisfies $\div\, z_{\pm(\delta),0}=0$, $z_{+(\delta),0}^3|_{x_3=\pm 1}=0$, $z_{-(\delta),0}^3|_{x_3=\pm 1}=0$ and for all $\delta\in(0,1]$
\beq\label{initial}\begin{aligned}
&\mathcal{E}_\delta(0)\eqdefa\sum_{+,-}\bigl(\sum_{k\leq N_*}E_\pm^{(k)}(z_{\pm(\delta),0}^h)+\sum_{k\leq N_*-1}E_\pm^{(k)}(z_{\pm(\delta),0}^3)\\
&\qquad
+\delta^2E_\pm^{(N_*)}(z_{\pm(\delta),0}^3)+\delta^{-2}\sum_{k\leq N+2}E_\pm^{(k)}(\p_3z_{\pm(\delta),0}^h)\bigr)\leq\varepsilon_1^2.
\end{aligned}\eeq
Assume that
\beq\label{lim ansatz}
\bigl(z_{\pm(\delta),0}^h(x_h,x_3),z_{\pm(\delta),0}^3(x_h,x_3)\bigr)\quad\mbox{converges to}\quad \bigl(z_{\pm(0),0}^h(x_h),0\bigr)\quad\text{in}\quad H^{N+1}(\Omega_1),
\eeq
with $\na_h\cdot z_{\pm(0),0}^h=0$.
Then if   $(z_{+(\delta)}(t,x),z_{-(\delta)}(t,x))$ is a solution to MHD \eqref{scaling MHD} with the initial data $(z_{+(\delta),0}(x),z_{-(\delta),0}(x))$,
  there exist functions $z_{\pm(0)}^h(t,x_h)$ such that for any $x_3\in(-1,1)$, $t>0$
\beq\label{limitation 1}\begin{aligned}
&\lim_{\delta\rightarrow0}z_{\pm(\delta)}^h(t,x_h,x_3)=z_{\pm(0)}^h(t,x_h),\quad\text{in}\quad H^{N}(\R^2),\\
&\lim_{\delta\rightarrow0}z_{\pm(\delta)}^3(t,x_h,x_3)=0,\quad\text{in}\quad H^{N-1}(\R^2).
\end{aligned}\eeq
In particular,  $(z_{+(0)}^h(t,x_h),z_{-(0)}^h(t,x_h))$ solves 2D version of the system \eqref{MHD}  with the initial data $(z_{+(0),0}^h(x_h), z_{-(0),0}^h(x_h))$ and verifies
\beq\label{energy estimate for 2D MHD}
\sum_{+,-}\sum_{k\leq N}\bigl(E_{\pm,h}^{(k)}(z_{\pm(0)}^h)+F_{\pm,h}^{(k)}(z_{\pm(0)}^h)\bigr)\leq C\sum_{+,-}\sum_{k\leq N}E_{\pm,h}^{(k)}(z_{\pm(0),0}^h).
\eeq
Moreover  the error estimate for the asymptotics can be concluded as follows:
\beq\label{error estimate}\begin{aligned}
&\sum_{+,-}\sup_{x_3\in(-1,1)}\Bigl(\sum_{k\leq N}\bigl(E_{\pm,h}^{(k)}(z_{\pm(\delta)}^h(\cdot,x_3)-z_{\pm(0)}^h(\cdot))+F_{\pm,h}^{(k)}(z_{\pm(\delta)}^h(\cdot,x_3)-z_{\pm(0)}^h(\cdot))\bigr)\\
&\quad+\sum_{k\leq N-1}\bigl(E_{\pm,h}^{(k)}(z_{\pm(\delta)}^3(\cdot,x_3)+F_{\pm,h}^{(k)}(z_{\pm(\delta)}^3(\cdot,x_3))\bigr)\Bigr)\\
&\leq C\sum_{+,-}\sum_{k\leq N}\bigl(\sup_{x_3\in(-1,1)}E_{\pm,h}^{(k)}(z_{\pm(\delta),0}^h(\cdot,x_3)-z_{\pm(0),0}^h(\cdot))
+C\delta^2\varepsilon_1^4.
\end{aligned}\eeq
\end{thm}

\begin{remark}
1. Thanks to \eqref{transform for energy functional}, the initial condition \eqref{initial} implies \eqref{initial on delta}. Then by Theorem 1.1, we have the global existence for \eqref{scaling MHD} and get  the uniform control of $\mathcal{E}_\delta(t)$ for all time.

2.
The assumption \eqref{initial} shows that
$
\sum_{k\leq N+2}\|\langle x_1\rangle^{1+\sigma}\na^{k}\p_3z_{\pm(\delta),0}\|_{L^2(\Omega_1)}^2\leq\delta^2\varepsilon_1^2,
$
which implies
$
\p_3z_{\pm(0),0}(x_h,x_3)=0$.
 Thus we have $z_{\pm(0),0}(x_h,x_3)=z_{\pm(0),0}(x_h)$. We remark that this property will be kept  for all time thanks to the uniform bound for $\mathcal{E}_\delta(t)$.
\end{remark}

\noindent  {\bf Notations.} For any $k\in\Z$, $H^k(\Omega)$ means the standard Sobolev spaces on domain $\Omega$. For any $k\in\Z$ and $p\in[1,\infty]$, $\Omega_\delta=\R^2\times(-\delta,\delta)\in\R^3$ with $\delta\in(0,1]$,  without confusion of the domain, the norms
$\|\cdot\|_{H^k}=\|\cdot\|_{H^k_x}$, $\|\cdot\|_{H_h^k}=\|\cdot\|_{H^K_{x_h}}$, $\|\cdot\|_{H_v^k}=\|\cdot\|_{H^k_{x_3}}$, $\|\cdot\|_{L^p_t}$, $\|\cdot\|_{L^p_x}$,
$\|\cdot\|_{L^p_h}=\|\cdot\|_{L^p_{x_h}}$ and $\|\cdot\|_{L^p_v}=\|\cdot\|_{L^p_{x_3}}$ mean $\|\cdot\|_{H^k(\Omega_\delta)}$, $\|\cdot\|_{H^k(\R^2)}$, $\|\cdot\|_{H^k(-\delta,\delta)}$, $\|\cdot\|_{L^p(0,t)}$, $\|\cdot\|_{L^p(\Omega)}$,
$\|\cdot\|_{L^p(\R^2)}$ and $\|\cdot\|_{L^p_{(-\delta,\delta)}}$ respectively. Similarly, we shall always use the notations $\|\cdot\|_{L^p_t(H^k)}$,
$\|\cdot\|_{L^p_v(H^k_h)}$,  $\|\cdot\|_{L^p_tL^q_x}$, $\|\cdot\|_{L^p_tL^q_h}$, $\|\cdot\|_{L^p_tL^q_v}$,
$\|\cdot\|_{L^p_tL^q_hL^r_v}$ etc. for any $p,q,r\in[1,\infty]$, $k\in\Z$. The notation $A\lesssim B$ means $A\leq cB$ while $A\sim B$ means $c^{-1} B\leq A\leq cB$ for some universal constant $c$ independent of $\delta$.

\section{The estimate for the pressure, technical lemmas and the characteristic geometry}

In this section, we will give the proof to some    lemmas which will be used throughout the paper.
\subsection{Derivation of the pressure $p$ in the domain $\Omega_\delta$}
Taking divergence to the first or the second equation of \eqref{MHD}, and using conditions $\dv\,\zpm=0$, $z_+^3|_{x_3=\pm\delta}=0$ and $z_-^3|_{x_3=\pm\delta}=0$, we have
\beq\label{equation for pressure}\begin{aligned}
&\Delta p=-\na\cdot(\zp\cdot\na \zm)\quad\text{in}\quad\Omega_\delta,\\
&\p_3p|_{x_3=\pm\delta}=0.
\end{aligned}\eeq
Due to $\dv\,\zpm=0$, the source term $\na\cdot(\zp\cdot\na \zm)$ can be also written in the  forms
\beno
\na\cdot(\zp\cdot\na \zm)=\p_i\zp^j\p_j\zm^i=\p_i\p_j(\zp^j\zm^i).
\eeno
In the next context, we shall solve $p$  from \eqref{equation for pressure}.

\begin{lemma}\label{derivation of pressure}
Given smooth vectors $(\zp,\zm)$, we have
\beq\label{expression of pressure}
\na p(t,x)=\int_{\Omega_\delta}\na_xG_\delta(x,y)(\p_i\zp^j\p_j\zm^i)(t,y)dy ,
\eeq
where  
\beq\label{Green function}
\na_xG_\delta(x,y)=\f{1}{4\pi}\bigl[\na_x\f{1}{|x-y|}+\sum_{k=1}^\infty\bigl(\na_x\f{1}{|x_{+,k}-y|}+\na_x\f{1}{|x_{-,k}-y|}\bigr)\bigr],
\eeq
with
\beq\label{reflected points}
x_{+,k}=(x_h,(-1)^k(x_3-2k\delta)),\quad x_{-,k}=(x_h,(-1)^k(x_3+2k\delta)).
\eeq
\end{lemma}
 
\begin{remark}
It is easy to check that
\beq\label{Green transformation}
\p_{x_h}^{\al_h}G_\delta(x,y)=(-1)^{|\al_h|}\p_{y_h}^{\al_h}G_\delta(x,y),
\eeq
for any $\al_h=(\al_1,\al_2)\in(\Z_{\geq0})^2$ and $|\al_h|\geq1$.
\end{remark}

\begin{proof}
Firstly, we set
\beno
f(t,x)\eqdefa(\p_i\zp^j\p_j\zm^i)(t,x).
\eeno

For continuous function $f\in C(\overline\Omega_\delta)$, we extend $f$ from $\Omega_\delta$ to the whole space $\R^3$ in the following way
\beq\label{extension}
\widetilde{f}(t,x_h,x_3)=f(t,x_h,(-1)^k(x_3-2k\delta)),\quad\text{if}\quad x_3\in((2k-1)\delta,(2k+1)\delta], \quad k\in\Z.
\eeq
It is easy to check that  $\wt{f}\in C(\R^3)$. Then we solve $\wt{p}$ (up to a constant) from the Laplacian equation $\Delta\wt{p}=-\wt{f}$ as follows:
\beq\label{extension pressure}
\wt{p}(t,x)=\f{1}{4\pi}\int_{\R^3}\f{1}{|x-y|}\wt{f}(t,y)dy,\quad\text{for}\quad x\in\R^3.
\eeq

Taking  $p=\wt{p}|_{\Omega_\delta}$, we shall verify that $p$ solves \eqref{equation for pressure}. By virtue of \eqref{extension} and \eqref{extension pressure}, we have
\beq\label{Laplace}
\Delta p=-f,\quad\text{for}\quad x\in\Omega_\delta.
\eeq
We only need to check that $\p_3p|_{x_3=\pm\delta}=0$. To do so, we transform the integration over $\R^3$ in \eqref{extension pressure} to that over $\Omega_\delta$.  Due to \eqref{extension} and \eqref{extension pressure}, we have
\beno\begin{aligned}
\na p(t,x)&=\f{1}{4\pi}\int_{\R^3}\na_x\f{1}{|x-y|}\wt{f}(t,y)dy\\
&=\f{1}{4\pi}\sum_{k=-\infty}^\infty\underbrace{\int_{(2k-1)\delta}^{(2k+1)\delta}\int_{\R^2}\na_x\f{1}{(|x_h-y_h|^2+|x_3-y_3|^2)^{\f12}}f(t,y_h,(-1)^k(y_3-2k\delta))dy_hdy_3}_{I_k(t,x)}.
\end{aligned}\eeno
For $I_k(t,x)$, setting $\wt{y}_3=(-1)^k(y_3-2k\delta)$, we have
\beno
y_3=(-1)^k\wt{y}_3+2k\delta,\quad x_3-y_3=x_3-2k\delta-(-1)^k\wt{y}_3=(-1)^k\bigl[(-1)^k\bigl(x_3-2k\delta)-\wt{y_3}\bigr],
\eeno
and
\beno\begin{aligned}
I_k&=(-1)^k\int_{(-1)^{k+1}\delta}^{(-1)^k\delta}\int_{\R^2}\na_x\f{1}{(|x_h-y_h|^2+|(-1)^k\bigl(x_3-2k\delta)-\wt{y_3}|^2)^{\f12}}
f(t,y_h,\wt{y}_3))dy_hd\wt{y}_3\\
&=\int_{-\delta}^{\delta}\int_{\R^2}\na_x\f{1}{(|x_h-y_h|^2+|(-1)^k\bigl(x_3-2k\delta)-y_3|^2)^{\f12}}
f(t,y_h,y_3)dy_hdy_3.
\end{aligned}\eeno
Then we have
\beq\label{pressure}
\na p(t,x)=\int_{\Omega_\delta}\na_x G_\delta(x,y)f(t,y)dy,
\eeq
where
\beq\label{Green}
\na_x G_\delta(x,y)\eqdefa\sum_{k=-\infty}^\infty\na_x\f{1}{(|x_h-y_h|^2+|(-1)^k\bigl(x_3-2k\delta)-y_3|^2)^{\f12}}.
\eeq
We remark that the r.h.s of \eqref{Green} is summable (see the following Lemma \ref{derivative of Green cor}).

Setting
\beno
x_{+,k}=(x_h,(-1)^k(x_3-2k\delta)),\quad x_{-,k}=(x_h,(-1)^k(x_3+2k\delta)),\quad\text{for}\quad k\in\N,
\eeno
we have
\beq\label{Green1}
\na_xG_\delta(x,y)=\f{1}{4\pi}\bigl[\na_x\f{1}{|x-y|}+\sum_{k=1}^\infty\bigl(\na_x\f{1}{|x_{+,k}-y|}+\na_x\f{1}{|x_{-,k}-y|}\bigr)\bigr].
\eeq
We remark that $x_{+,k}$ is the reflection point of the point $x_{+,k-1}$ in the plane $\Gamma_{(-)^{k-1}}$ and $x_{-,k}$ is the reflection point of the point $x_{-,k-1}$ in the plane $\Gamma_{(-)^{k}}$ (here $(-)^{m}$ equals $"+"$ if $m$ is even while equals $"-"$ if $m$ is odd, and $x_{+,0}=x_{-,0}\eqdefa x$).
Denoting by
\beno
\phi_k(x,y)=\f{1}{|x_{+,k}-y|}+\f{1}{|x_{-,k}-y|},\quad\text{for}\quad k\in\N,
\eeno
we have
\beno
\p_{x_3}\phi_k(x,y)=(-1)^{k+1}\Bigl(\f{x_{+,k}^3-y_3}{|x_{+,k}-y|^3}+\f{x_{-,k}^3-y_3}{|x_{-,k}-y|^3}\Bigr),
\eeno
which implies
\beno\begin{aligned}
&\p_{x_3}\phi_k(x,y)|_{x_3=\delta}=(-1)^{k+1}\Bigl(\f{(-1)^{k+1}(2k-1)\delta-y_3}{|(x_h,(-1)^{k+1}(2k-1)\delta)-y|^3}
+\f{(-1)^k(2k+1)\delta-y_3}{|(x_h,(-1)^k(2k+1)\delta)-y|^3}\Bigr),\\
&\p_{x_3}\phi_k(x,y)|_{x_3=-\delta}=(-1)^{k+1}\Bigl(\f{(-1)^{k+1}(2k+1)\delta-y_3}{|(x_h,(-1)^{k+1}(2k+1)\delta)-y|^3}
+\f{(-1)^k(2k-1)\delta-y_3}{|(x_h,(-1)^k(2k-1)\delta)-y|^3}\Bigr).
\end{aligned}\eeno
Then we have
\beno\begin{aligned}
&\sum_{k=1}^\infty\p_{x_3}\phi_k(x,y)|_{x_3=\delta}=\f{\delta-y_3}{|(x_h,\delta)-y|^3}
+\lim_{k\rightarrow\infty}\f{(-1)^{k+1}[(-1)^k(2k+1)\delta-y_3]}{|(x_h,(-1)^k(2k+1)\delta)-y|^3},\\
&\sum_{k=1}^\infty\p_{x_3}\phi_k(x,y)|_{x_3=-\delta}=\f{-\delta-y_3}{|(x_h,-\delta)-y|^3}
+\lim_{k\rightarrow\infty}\f{(-1)^{k+1}[(-1)^{k+1}(2k+1)\delta-y_3]}{|(x_h,(-1)^{k+1}(2k+1)\delta)-y|^3}.
\end{aligned}\eeno
Since $y_3\in(-\delta,\delta)$, we have
\beno
\lim_{k\rightarrow\infty}\f{(-1)^{k+1}[(-1)^k(2k+1)\delta-y_3]}{|(x_h,(-1)^k(2k+1)\delta)-y|^3}
=\lim_{k\rightarrow\infty}\f{(-1)^{k+1}[(-1)^{k+1}(2k+1)\delta-y_3]}{|(x_h,(-1)^{k+1}(2k+1)\delta)-y|^3}=0,
\eeno
which implies that
\beno
\sum_{k=1}^\infty\p_{x_3}\phi_k(x,y)|_{x_3=\pm\delta}=\f{\pm\delta-y_3}{|(x_h,\pm\delta)-y|^3}=-\p_{x_3}\Bigl(\f{1}{|x-y|}\Bigr)\Big|_{x_3=\pm\delta}.
\eeno

Due to \eqref{Green1}, we obtain
\beno
\p_{x_3}G_\delta(x,y)|_{x_3=\pm\delta}=0,
\eeno
which along with \eqref{pressure} implies
\beno
\p_3p|_{x_3=\pm\delta}=0.
\eeno

Thus, $p(t,x)$ obtained in the proof satisfies  \eqref{equation for pressure}. Moreover, there holds \eqref{expression of pressure}. The lemma is proved.
\end{proof}

\begin{corollary}\label{derivative of Green cor}
Assume that $\na_x G_\delta(x,y)$ is the function on $\Omega_\delta\times\Omega_\delta$ defined in \eqref{Green function}. Then we have
\beq\label{derivative of Green function}
|\na_x^k G_\delta(x,y)|\lesssim\f{1}{\delta}\f{1}{|x_h-y_h|^k},\quad k=1,2,3,\cdots.
\eeq
\end{corollary}
\begin{proof}
By the definition \eqref{Green function}, we have
\beno\begin{aligned}
&|\na_x G_\delta(x,y)|\leq\f{1}{4\pi}\bigl[\f{1}{|x-y|^2}+\sum_{k=1}^\infty\bigl(\f{1}{|x_{+,k}-y|^2}+\f{1}{|x_{-,k}-y|^2}\bigr)\bigr]\\
&\leq\f{1}{4\pi}\bigl[\f{1}{|x_h-y_h|^2+|x_3-y_3|^2}+\sum_{k=1}^\infty\bigl(\f{1}{|x_h-y_h|^2+|(-1)^kx_3-y_3-2(-1)^kk\delta|^2}\\
&\qquad
+\f{1}{|x_h-y_h|^2+|(-1)^kx_3-y_3+2(-1)^kk\delta|^2}\bigr)\bigr].
\end{aligned}\eeno
Since $x_3,y_3\in(-\delta,\delta)$, we have
\beno\begin{aligned}
&|\na_x G_\delta(x,y)|\leq\f{1}{4\pi}\bigl[\f{1}{|x_h-y_h|^2}+\sum_{k=1}^\infty\f{2}{|x_h-y_h|^2+|2(k-1)\delta|^2}\bigr]\\
&\leq\f{3}{4\pi}\sum_{k=0}^\infty\f{1}{|x_h-y_h|^2+|2k\delta|^2}\leq\f{3}{4\pi}\int_0^\infty\f{1}{|x_h-y_h|^2+4\tau^2\delta^2}d\tau\\
&\leq\f{3}{4\pi}\f{1}{2\delta|x_h-y_h|}\int_0^\infty\f{1}{1+\tau^2}d\tau=\f{3}{16}\cdot\f{1}{\delta}\cdot\f{1}{|x_h-y_h|}.
\end{aligned}\eeno
Then \eqref{derivative of Green function} holds for $k=1$. Similarly, we could prove that \eqref{derivative of Green function} holds for all $k\in\N$.
\end{proof}

%

\subsection{Technical lemmas}
We first state the following div-curl lemma.
\begin{lemma}[{\bf{div-curl lemma}}]\label{div-curl lemma}
Let $\lambda(x)$ be a smooth positive function on $\Omega_\delta$. Then for all smooth vector field $\vv v(x)=(v^1(x),
v^2(x),v^3(x))\in H^1(\Omega_\delta)$ with the following properties
\beno
\dv\,\vv v=0,\quad\sqrt\lambda\na\vv v\in L^2(\Omega_\delta),\quad\f{|\na\lambda|}{\sqrt\lambda}\vv v\in L^2(\Omega_\delta),
\eeno
it holds
\beq\label{div-curl}
\|\sqrt\lambda\na\vv v\|_{L^2}^2\leq C\bigl(\|\sqrt\lambda\curl\vv v\|_{L^2}^2+\|\f{|\na\lambda|}{\sqrt\lambda}\vv v\|_{L^2}^2+\bigl|\int_{\p\Omega_\delta}\lambda\vv v\cdot\na v^3dx_1dx_2\bigr|\bigr),
\eeq
where $C$ is a universal constant independent of $\delta$.
\end{lemma}
\begin{proof}
Since $\dv\,\vv v=0$, we have
\beno
-\Delta \vv v=\curl\curl\vv v.
\eeno
Multiplying the above identity by $\lambda\vv v$, and integrating over $\Omega_\delta$, we have
\beno
\int_{\Omega_\delta}-\Delta \vv v\cdot\lambda\vv v dx=\int_{\Omega_\delta}\curl\curl\vv v\cdot \lambda\vv v dx.
\eeno
By integration by parts, we have
\beno\begin{aligned}
&\int_{\Omega_\delta}-\Delta \vv v\cdot\lambda\vv v dx=\int_{\Omega_\delta}\lambda|\na\vv v|^2dx+\int_{\Omega_\delta}(\na\lambda\cdot\na)\vv v\cdot\vv v dx-\int_{\p\Omega_\delta}\f{\p\vv v}{\p\vv n}\cdot\lambda\vv v dS,\\
&\int_{\Omega_\delta}\curl\curl\vv v\cdot \lambda\vv v dx=\int_{\Omega_\delta}\lambda|\curl\vv v|^2dx+\int_{\Omega_\delta}\curl\vv v\cdot(\na\lambda\times\vv v)dx-\int_{\p\Omega_\delta}\curl\vv v\times\vv n\cdot\lambda\vv vdS,
\end{aligned}\eeno
where $\vv n$ is the unit outward normal to $\p\Omega_\delta$ and $dS$ is the surface measure of $\p\Omega_\delta$. Then we obtain
\beq\label{dc1}\begin{aligned}
&\int_{\Omega_\delta}\lambda|\na\vv v|^2dx=\int_{\Omega_\delta}\lambda|\curl\vv v|^2dx-\int_{\Omega_\delta}(\na\lambda\cdot\na)\vv v\cdot\vv v dx\\
&\qquad+\int_{\Omega_\delta}\curl\vv v\cdot(\na\lambda\times\vv v)dx+\int_{\p\Omega_\delta}\bigl(\f{\p\vv v}{\p\vv n}-\curl\vv v\times\vv n\bigr)\cdot\lambda\vv v dS.
\end{aligned}\eeq
The direct calculation yields
\beq\label{dc2}
\int_{\p\Omega_\delta}\bigl(\f{\p\vv v}{\p\vv n}-\curl\vv v\times\vv n\bigr)\cdot\lambda\vv v dS=\int_{\p\Omega_\delta}\lambda n_i\p_jv^iv^jdS,
\eeq
where $\vv n=(n_1,n_2,n_3)^T$. Using H\"older inequality, we deduce from \eqref{dc1} and \eqref{dc2} that
\beq\label{dc3}
\|\sqrt\lambda\na\vv v\|_{L^2}^2\leq C\bigl(\|\sqrt\lambda\curl\vv v\|_{L^2}^2+\|\f{|\na\lambda|}{\sqrt\lambda}\vv v\|_{L^2}^2+\bigl|\int_{\p\Omega_\delta}\lambda n_i\p_jv^iv^jdS\bigr|\bigr).
\eeq
Noticing that $\vv n|_{x_3=\pm\delta}=(0,0,\pm1)^T$ and $dS=dx_1dx_2$, we obtain the desired inequality \eqref{div-curl}. The Lemma is proved.
\end{proof}

As a consequence of Lemma \ref{div-curl lemma}, we have the following corollary.
\begin{corollary}\label{div-curl cor}
Let $\lambda(x)$ be a smooth positive function on $\Omega_\delta$ with $|\na\lambda|\leq C|\lambda|$. For all smooth vector field $\vv v(x)\in H^k(\Omega_\delta)$  $(k\in\N)$ with the following properties
\beno
\dv\,\vv v=0,\quad v^3|_{x_3=\pm\delta}=0,\quad\sqrt\lambda\na^l\vv v\in L^2(\Omega_\delta),\quad l=0,1,\cdots, k,
\eeno
we have
\beq\label{div-curl-general}
\|\sqrt\lambda\na^k\vv v\|_{L^2}^2\leq C\bigl(\sum_{l=0}^{k-1}\|\sqrt\lambda\na^l\curl\vv v\|_{L^2}^2+\|\sqrt\lambda\vv v\|_{L^2}^2\bigr).
\eeq
\end{corollary}
\begin{proof}
We prove \eqref{div-curl-general} by the induction method. For $k=1$, using \eqref{div-curl}, we have
\beno
\|\sqrt\lambda\na\vv v\|_{L^2}^2\leq C\bigl(\|\sqrt\lambda\curl\vv v\|_{L^2}^2+\|\f{|\na\lambda|}{\sqrt\lambda}\vv v\|_{L^2}^2+\bigl|\int_{\p\Omega_\delta}\lambda\bigl(\vv v^h\cdot\na_h v^3+v^3\p_3v^3\bigr)dx_1dx_2\bigr|\bigr).
\eeno
Due to the boundary condition $v^3|_{x_3=\pm\delta}=0$ and $|\na\lambda|\leq C\lambda$, we get
\beq\label{dc4}
\|\sqrt\lambda\na\vv v\|_{L^2}^2\leq C\bigl(\|\sqrt\lambda\curl\vv v\|_{L^2}^2+\|\sqrt\lambda\vv v\|_{L^2}^2\bigr).
\eeq

Assume that for any $1\leq m\leq k-1$, it holds \eqref{div-curl-general}, i.e.,
\beq\label{dc5}
\|\sqrt\lambda\na^m\vv v\|_{L^2}^2\leq C\bigl(\sum_{l=0}^{m-1}\|\sqrt\lambda\na^l\curl\vv v\|_{L^2}^2+\|\sqrt\lambda\vv v\|_{L^2}^2\bigr).
\eeq
Then for $k$, we have
\beno
\|\sqrt\lambda\na^k\vv v\|_{L^2}^2\leq\|\sqrt\lambda\na^{k-1}(\p_h\vv v)\|_{L^2}^2+\|\sqrt\lambda\na^{k-1}(\p_3\vv v)\|_{L^2}^2,
\eeno
where $\p_h$ is the horizontal derivative with respect to $x_h$.
Since $\na_h v^3|_{x_3=\pm\delta}=0$, using \eqref{dc5}, we have
\beno
\|\sqrt\lambda\na^{k-1}(\p_h\vv v)\|_{L^2}^2\leq C\bigl(\sum_{l=0}^{k-2}\|\sqrt\lambda\na^l\curl(\p_h\vv v)\|_{L^2}^2+\|\sqrt\lambda\na_h\vv v\|_{L^2}^2\bigr).
\eeno
Using \eqref{dc4}, we have
\beq\label{dc6}
\|\sqrt\lambda\na^{k-1}(\p_h\vv v)\|_{L^2}^2\leq C\bigl(\sum_{l=0}^{k-1}\|\sqrt\lambda\na^l\curl\vv v\|_{L^2}^2+\|\sqrt\lambda\vv v\|_{L^2}^2\bigr).
\eeq

On the other hand, since
\beno
\p_3 v^1=(\p_3 v^1-\p_1 v^3)+\p_1 v^3,\quad \p_3 v^2=(\p_3 v^2-\p_2 v^3)+\p_2 v^3,\quad\p_3 v^3=-\p_1 v^1-\p_2 v^2.
\eeno
we have
\beno
\|\sqrt\lambda\na^{k-1}(\p_3\vv v)\|_{L^2}^2\leq \|\sqrt\lambda\na^{k-1}(\curl\vv v)\|_{L^2}^2+\|\sqrt\lambda\na^{k-1}(\p_h\vv v)\|_{L^2}^2.
\eeno
Then using \eqref{dc6}, we have
\beq\label{dc7}
\|\sqrt\lambda\na^{k-1}(\p_3\vv v)\|_{L^2}^2\leq C\bigl(\sum_{l=0}^{k-1}\|\sqrt\lambda\na^l\curl\vv v\|_{L^2}^2+\|\sqrt\lambda\vv v\|_{L^2}^2\bigr).
\eeq
Thanks to \eqref{dc6} and \eqref{dc7}, we obtain the desired result.
\end{proof}

\begin{lemma}[{\bf Sobolev inequality}]\label{Sobolev}
(i) For any $f(x)\in H_{v}^1(-\delta,\delta;H^k_h(\R^2))$, we have
\beq\label{s1}
\|f(x)\|_{L_v^\infty(H^k_h)}\leq C\bigl(\delta^{-\f{1}{2}}\|f\|_{L^2_v(H^k_h)}+\delta^{\f{1}{2}}\|\p_3f\|_{L^2_v(H^k_h)}\bigr).
\eeq

(ii) For any $f(x)\in H^2(\Omega_\delta)$, we have
\beq\label{s2}
\|f(x)\|_{L^\infty}\leq C\sum_{k+l\leq 2}\delta^{l-\f{1}{2}}\|\na_h^k\p_3^lf\|_{L^2}.
\eeq
\end{lemma}
\begin{proof}
For any function $f$ defined on $\Omega_\delta$, we set
\beno
\tilde{f}(x_h,x_3)=f(x_h,\delta x_3),\quad x\in\Omega_1=\R^2\times(-1,1).
\eeno
Then the general Sobolev inequality shows that
\beno
\|f(x)\|_{L_v^\infty(-\delta,\delta;H^k_h)}=\|\tilde{f}(x)\|_{L_v^\infty(-1,1;H^k_h)}\leq C\bigl(\|\tilde{f}(x)\|_{L_v^2(-1,1;H^k_h)}
+\|\p_3\tilde{f}(x)\|_{L_v^2(-1,1;H^k_h)}\bigr).
\eeno
Since
\beno
\tilde{f}(x_h,x_3)=f(x_h,\delta x_3),\quad \p_3\tilde{f}(x_h,x_3)=\delta(\p_3f)(x_h,\delta x_3),
\eeno
we have
\beno
\|f(x)\|_{L_v^\infty(-\delta,\delta;H^k_h)}\leq C\bigl(\delta^{-\f{1}{2}}\|f(x)\|_{L_v^2(-\delta,\delta;H^k_h)}
+\delta^{\f{1}{2}}\|\p_3f(x)\|_{L_v^2(-\delta,\delta;H^k_h)}\bigr).
\eeno
Similarly, we have
\beno
\|f(x)\|_{L^\infty(\Omega_\delta)}=\|\tilde{f}(x)\|_{L^\infty(\Omega_1)}\leq C\sum_{k+l\leq2}\|\na_h^k\p_3^l\tilde{f}\|_{L^2(\Omega_1)}
\leq C\sum_{k+l\leq 2}\delta^{l-\f{1}{2}}\|\na_h^k\p_3^lf\|_{L^2(\Omega_\delta)}.
\eeno
The lemma is proved.
\end{proof}

\subsection{The characteristic geometry}
We study the spacetime $[0,t^*]\times\Omega_\delta=[0,t^*]\times\{\R^2_{x_h}\times(-\delta,\delta)_{x_3}\}$ associated to solution $(\vv v,\vv b)$ or $(z_+,z_-)$ of MHD. We recall that there exists a natural foliation for $[0,t^*]\times\Omega_\delta$ as
\beno
[0,t^*]\times\Omega_\delta=\cup_{0\leq t\leq t^*} \Sigma_t,
\eeno
where $\Sigma_t$ is the constant time slice ($\Sigma_0$ is the initial time slice where the initial data are given).

Due to the linear characteristic hypersurface, there exist another foliations as
\beno
[0,t^*]\times\Omega_\delta=\cup_{u_+\in\R}C_{u_+}^+=\cup_{u_-\in\R}C_{u_-}^-,
\eeno
where
\beno\begin{aligned}
&C_{u_+}^+\eqdefa\{(t,x)\in[0,t^*]\times\Omega_\delta\,|\, u_+=x_1-t=\text{constant}\},\\
&C_{u_-}^-\eqdefa\{(t,x)\in[0,t^*]\times\Omega_\delta\,|\,u_-=x_1+t=\text{constant}\},
\end{aligned}\eeno
with $C_{u_+}^+$ and $C_{u_-}^-$ being the level sets $\{u_+=\text{constant}\}$ and  $\{u_-=\text{constant}\}$ respectively.

We denote by
\beno
S_{t,u_+}^+=C_{u_+}^+\cap\Sigma_t,\quad S_{t,u_-}^-=C_{u_-}^-\cap\Sigma_t.
\eeno
Therefore, for time $t$, there exist two foliations of $\Sigma_t$ as follows
\beno
\Sigma_t=\cup_{u_+\in\R} S_{t,u_+}^+=\cup_{u_-\in\R} S_{t,u_-}^-.
\eeno

In order to specify the region where the energy estimates take place, for given $t, u_+^1, u_+^2, u_-^1, u_-^2$ with $u_+^1<u_+^2$, $u_-^1<u_-^2$, we denote the following hypersurfaces/regions:
\beno\begin{aligned}
&\Sigma_t^{[u_+^1,u_+^2]}=\cup_{u_+\in[u_+^1,u_+^2]} S_{t,u_+}^+,\quad \Sigma_t^{[u_-^1,u_-^2]}=\cup_{u_-\in[u_-^1,u_-^2]} S_{t,u_-}^-,\\
&W_t^{[u_+^1,u_+^2]}=\cup_{\tau\in[0,t]}\Sigma_\tau^{[u_+^1,u_+^2]},\quad W_t^{[u_-^1,u_-^2]}=\cup_{\tau\in[0,t]}\Sigma_\tau^{[u_-^1,u_-^2]},\quad W_t=\cup_{\tau\in[0,t]}\Sigma_\tau.
\end{aligned}\eeno

We shall also study the spacetime $[0,t^*]\times\R^2$ associated to the solution of MHD in the horizontal direction of the thin domain (or that of 2D MHD). There also exists a natural foliation of  $[0,t^*]\times\R^2$ as
\beno
[0,t^*]\times\R^2=\cup_{0\leq t\leq t^*} \Sigma_{t,h},
\eeno
where $\Sigma_{t,h}$ is the constant time slice of the horizontal direction. We also define the two-dimensional linear characteristic hypersurface as follows:
\beno\begin{aligned}
&C_{u_+,h}^+\eqdefa\{(t,x_h)\in[0,t^*]\times\R^2\,|\, u_+=x_1-t=\text{constant}\},\\
&C_{u_-,h}^-\eqdefa\{(t,x_h)\in[0,t^*]\times\R^2\,|\,u_-=x_1+t=\text{constant}\},
\end{aligned}\eeno
where $C_{u_+,h}^+$ and $C_{u_-,h}^-$ being the level sets $\{u_+=\text{constant}\}$ and  $\{u_-=\text{constant}\}$ in $[0,t^*]\times\R^2$ respectively.
Then there exist another two foliations of $[0,t^*]\times\R^2$
\beno
[0,t^*]\times\R^2=\cup_{u_+\in\R}C_{u_+,h}^+=\cup_{u_-\in\R}C_{u_-,h}^-.
\eeno

We can also give the definitions to $S_{t,u_+,h}^+$, $S_{t,u_-,h}^-$, $\Sigma_{t,h}$, $\Sigma_{t,h}^{[u_+^1,u_+^2]}$, $\Sigma_{t,h}^{[u_-^1,u_-^2]}$, $W_{t,h}^{[u_+^1,u_+^2]}$ and $W_{t,h}^{[u_-^1,u_-^2]}$ in a similar way.

\section{Proof of Theorem \ref{global existence in thin domain}}
In this section, we will give a complete proof to Theorem \ref{global existence in thin domain}.
\subsection{The priori estimates of the linearized system}
We consider the following linearized system
\begin{equation}\label{Linerized system}
\begin{aligned}
&\p_t f_+-\p_1f_++z_-\cdot\na f_+=\r_+,\quad\text{in}\quad\Omega_\delta,\\
&\p_t f_-+\p_1f_-+z_+\cdot\na f_-=\r_-.
\end{aligned}
\end{equation}
Here $\div\, z_\pm=0$ and $z_+^3|_{x_3=\pm\delta}=0$, $z_-^3|_{x_3=\pm\delta}=0$.
We first have the following proposition.

\begin{proposition}\label{linearized prop}
Suppose that $\div\, z_\pm=0$, $z_+^3|_{x_3=\pm\delta}=0$, $z_-^3|_{x_3=\pm\delta}=0$ and
\beq\label{assumption 1}
\|z_\pm^1\|_{L^\infty_tL^\infty_x}\leq 1.
\eeq
If  $(f_+,f_-)$ is  a smooth solution to \eqref{Linerized system}, it  holds
\beq\label{estimate for linearized system}\begin{aligned}
&\sup_{0\leq\tau\leq t}\int_{\Sigma_\tau}\langle u_\mp\rangle^{2(1+\sigma)}|f_\pm|^2dx+\int_0^t\int_{\Sigma_\tau}\f{\langle u_\mp\rangle^{2(1+\sigma)}}{\langle u_\pm\rangle^{1+\sigma}}|f_\pm|^2dxd\tau\\
&\lesssim\int_{\Sigma_0}\langle u_\mp\rangle^{2(1+\sigma)}|f_\pm|^2dx+\int_0^t\int_{\Sigma_\tau}\langle u_\mp\rangle^{1+2\sigma}|z_\mp^1||f_\pm|^2dxd\tau+\Bigl|\int_0^t\int_{\Sigma_\tau}\r_\pm\cdot\langle u_\mp\rangle^{2(1+\sigma)}f_\pm dxd\tau\Bigr|\\
&\qquad
+\int_{\R}\f{1}{\langle u_\pm\rangle^{1+\sigma}}\Bigl|\int\int_{W_t^{[u_+,\infty]}/(W_t^{[-\infty,u_-]})}\r_\pm\cdot\langle u_\mp\rangle^{2(1+\sigma)}f_\pm dxd\tau\Bigr| d{u_\pm}.
\end{aligned}\eeq
In particular, we have
\beq\label{estimate for linearized system 0}\begin{aligned}
&\sup_{0\leq\tau\leq t}\int_{\Sigma_\tau}\langle u_\mp\rangle^{2(1+\sigma)}|f_\pm|^2dx+\int_0^t\int_{\Sigma_\tau}\f{\langle u_\mp\rangle^{2(1+\sigma)}}{\langle u_\pm\rangle^{1+\sigma}}|f_\pm|^2dxd\tau\\
&\lesssim\int_{\Sigma_0}\langle u_\mp\rangle^{2(1+\sigma)}|f_\pm|^2dx+\int_0^t\int_{\Sigma_\tau}\langle u_\mp\rangle^{1+2\sigma}|z_\mp^1||f_\pm|^2dxd\tau+\int_0^t\int_{\Sigma_\tau}|\r_\pm|\cdot\langle u_\mp\rangle^{2(1+\sigma)}|f_\pm| dxd\tau.
\end{aligned}\eeq
\end{proposition}
\begin{proof}
We start with the estimates for $f_+$ which is corresponding to the left-traveling (along the negative direction of $x_1$-axis) Alfv\'en wave. Multiplying $\langle u_-\rangle^{2(1+\sigma)}f_+$ to the first equation of \eqref{Linerized system}, and then integrating over $W_t=\cup_{\tau\in(0,t)} \Sigma_\tau$, we have
\beno
\f{1}{2}\int_0^t\int_{\Sigma_\tau}\bigl(\p_t-\p_1+z_-\cdot\na\bigr)|f_+|^2\cdot \langle u_-\rangle^{2(1+\sigma)}dxd\tau=\int_0^t\int_{\Sigma_\tau}\r_+\cdot\langle u_-\rangle^{2(1+\sigma)}f_+dxd\tau.
\eeno
Since $u_-=x_1+t$, the left hand side(l.h.s) of the above equality equals
\beno
\f{1}{2}\int_0^t\int_{\Sigma_\tau}\bigl(\p_t-\p_1+z_-\cdot\na\bigr)\bigl(\langle u_-\rangle^{2(1+\sigma)}|f_+|^2\bigr)dxd\tau-\f{1}{2}\int_0^t\int_{\Sigma_\tau}z_-^1\p_1\bigl(\langle u_-\rangle^{2(1+\sigma)}\bigr)|f_+|^2dxd\tau.
\eeno

We denote the vector field in the spacetime $[0,t^*]\times\Omega_\delta$ (also the operator) by
\beno
L_-=\p_t-\p_1+z_-\cdot\na,\quad T=\p_t,
\eeno
and denote by $\widetilde{\div}$ the divergence of $\R^4$ with standard Euclidean metric. We have
\beno\begin{aligned}
&\f{1}{2}\int_0^t\int_{\Sigma_\tau}L_-\bigl(\langle u_-\rangle^{2(1+\sigma)}|f_+|^2\bigr)dxd\tau\\
&=\f{1}{2}\int_0^t\int_{\Sigma_\tau}z_-^1\p_1\bigl(\langle u_-\rangle^{2(1+\sigma)}\bigr)|f_+|^2dxd\tau+\int_0^t\int_{\Sigma_\tau}\r_+\cdot\langle u_-\rangle^{2(1+\sigma)}f_+dxd\tau.
\end{aligned}\eeno
 Since $\div\,z_-=0$, we have $\widetilde{\div} L_-=0$.  Then using the Stokes formula and $z_-^3|_{x_3=\pm\delta}=0$, we have
\beno\begin{aligned}
&\f{1}{2}\int_0^t\int_{\Sigma_\tau}L_-\bigl(\langle u_-\rangle^{2(1+\sigma)}|f_+|^2\bigr)dxd\tau=\f{1}{2}\int_0^t\int_{\Sigma_\tau}\widetilde{\div}\bigl(\langle u_-\rangle^{2(1+\sigma)}|f_+|^2L_-\bigr)dxd\tau\\
&=\f{1}{2}\int_{\Sigma_t}\langle u_-\rangle^{2(1+\sigma)}|f_+|^2\langle L_-,T\rangle dx-\f{1}{2}\int_{\Sigma_0}\langle u_-\rangle^{2(1+\sigma)}|f_+|^2\langle L_-,T\rangle dx.
\end{aligned}\eeno
Noticing that $\langle L_-,T\rangle=1$, we have
\beq\label{L1}\begin{aligned}
\f{1}{2}\int_{\Sigma_t}\langle u_-\rangle^{2(1+\sigma)}|f_+|^2dx&=\f{1}{2}\int_{\Sigma_0}\langle u_-\rangle^{2(1+\sigma)}|f_+|^2dx+\f{1}{2}\int_0^t\int_{\Sigma_\tau}z_-^1\p_1\bigl(\langle u_-\rangle^{2(1+\sigma)}\bigr)|f_+|^2dxd\tau\\
&\qquad
+\int_0^t\int_{\Sigma_\tau}\r_+\cdot\langle u_-\rangle^{2(1+\sigma)}f_+dxd\tau.
\end{aligned}\eeq

To control the second term in the l.h.s of \eqref{estimate for linearized system}, we have to derive the local energy estimates. Multiplying $\langle u_-\rangle^{2(1+\sigma)}f_+$ to the first equation of \eqref{Linerized system}, and then integrating over $W_t^{[u_+,\infty]}$, similarly to the derivation of \eqref{L1}, we obtain
\beno\begin{aligned}
&\f{1}{2}\int\int_{W_t^{[u_+,\infty]}}L_-\bigl(\langle u_-\rangle^{2(1+\sigma)}|f_+|^2\bigr)dxd\tau\\
&=\f{1}{2}\int\int_{W_t^{[u_+,\infty]}}z_-^1\p_1\bigl(\langle u_-\rangle^{2(1+\sigma)}\bigr)|f_+|^2dxd\tau+\int\int_{W_t^{[u_+,\infty]}}\r_+\cdot\langle u_-\rangle^{2(1+\sigma)}f_+dxd\tau.
\end{aligned}\eeno
Using Stokes formula and the facts that $\widetilde{\div} L_-=0$ and $z_-^3|_{x_3=\pm\delta}=0$, we have
\beno\begin{aligned}
&\f{1}{2}\int\int_{W_t^{[u_+,\infty]}}L_-\bigl(\langle u_-\rangle^{2(1+\sigma)}|f_+|^2\bigr)dxd\tau=\f{1}{2}\int\int_{W_t^{[u_+,\infty]}}\widetilde{\div}\bigl(\langle u_-\rangle^{2(1+\sigma)}|f_+|^2L_-\bigr)dxd\tau\\
&=\f{1}{2}\int_{\Sigma_t^{[u_+,\infty]}}\langle u_-\rangle^{2(1+\sigma)}|f_+|^2\langle L_-,T\rangle dx-\f{1}{2}\int_{\Sigma_0^{[u_+,\infty]}}\langle u_-\rangle^{2(1+\sigma)}|f_+|^2\langle L_-,T\rangle dx\\
&\qquad
+\f12\int_{C_{u_+}^+}\langle u_-\rangle^{2(1+\sigma)}|f_+|^2\langle L_-,\nu^+\rangle d\sigma_+,
\end{aligned}\eeno
where $\nu^+$ is the unit outward normal to $C_{u_+}^+$ and
\beno
\nu^+=(1,-1,0,0)^T.
\eeno
Since $L_-=(1,-1+z_-^1,z_-^2,z_-^3)^T$, we have
\beno
\langle L_-,\nu^+\rangle=2-z_-^1.
\eeno
Then using the fact $\langle L_-,T\rangle =1$ and the assumption \eqref{assumption 1}, we have
\beq\label{estimate of flux}\begin{aligned}
&\f{1}{2}\int_{\Sigma_t^{[u_+,\infty]}}\langle u_-\rangle^{2(1+\sigma)}|f_+|^2dx+\f12\int_{C_{u_+}^+}\langle u_-\rangle^{2(1+\sigma)}|f_+|^2 d\sigma_+\\
&\leq\f{1}{2}\int_{\Sigma_0^{[u_+,\infty]}}\langle u_-\rangle^{2(1+\sigma)}|f_+|^2dx+\f{1}{2}\int\int_{W_t^{[u_+,\infty]}}|z_-^1||\p_1\bigl(\langle u_-\rangle^{2(1+\sigma)}\bigr)||f_+|^2dxd\tau\\
&\qquad+\bigl|\int\int_{W_t^{[u_+,\infty]}}\r_+\cdot\langle u_-\rangle^{2(1+\sigma)}f_+dxd\tau\bigr|.
\end{aligned}\eeq
Multiplying $\f{1}{\langle u_+\rangle^{1+\sigma}}$ to both sides of \eqref{estimate of flux}, and integrating the resulting inequality over $\R$, we have
\beq\label{flux 1}\begin{aligned}
&\f12\int_{\R}\f{1}{\langle u_+\rangle^{1+\sigma}}\int_{C_{u_+}^+}\langle u_-\rangle^{2(1+\sigma)}|f_+|^2d\sigma_+d{u_+}\leq\f{1}{2}\int_{\R}\f{1}{\langle u_+\rangle^{1+\sigma}}d{u_+}\int_{\Sigma_0}\langle u_-\rangle^{2(1+\sigma)}|f_+|^2dx\\
&\qquad+\f{1}{2}\int_{\R}\f{1}{\langle u_+\rangle^{1+\sigma}}d{u_+}\int\int_{W_t}|z_-^1||\p_1\bigl(\langle u_-\rangle^{2(1+\sigma)}\bigr)||f_+|^2dxd\tau\\
&\qquad
+\int_{\R}\f{1}{\langle u_+\rangle^{1+\sigma}}\bigl|\int\int_{W_t^{[u_+,\infty]}}\r_+\cdot\langle u_-\rangle^{2(1+\sigma)}f_+dxd\tau\bigr|d{u_+}.
\end{aligned}\eeq
By the definition of $C_{u_+}^+$, we have
$
d\sigma_+=\sqrt 2 d\tau dx_2dx_3.
$
Then we have
\beno
\int_{\R}\f{1}{\langle u_+\rangle^{1+\sigma}}\int_{C_{u_+}^+}\langle u_-\rangle^{2(1+\sigma)}|f_+|^2d\sigma_+d{u_+}=\sqrt 2\int\int_{W_t}\f{\langle u_-\rangle^{2(1+\sigma)}}{\langle u_+\rangle^{1+\sigma}}|f_+|^2d\tau dx_2dx_3d{u_+}.
\eeno
We define the   variables transformation $\Phi$ from $W_t$ to $W_t$  such that
\beno\begin{aligned}
&(\tau,x_1,x_2,x_3)\mapsto \Phi(\tau,x)=(\tau,u_+,x_2,x_3)
\end{aligned}\eeno
with $u_+=x_1-\tau$. We have
$
\det(d\Phi)=1.
$
Then we have
\beq\label{coordinate transform}
\int_{\R}\f{1}{\langle u_+\rangle^{1+\sigma}}\int_{C_{u_+}^+}\langle u_-\rangle^{2(1+\sigma)}|f_+|^2d\sigma_+d{u_+}=\sqrt 2\int\int_{W_t}\f{\langle u_-\rangle^{2(1+\sigma)}}{\langle u_+\rangle^{1+\sigma}}|f_+|^2dxd\tau.
\eeq
Since $\int_{\R}\f{1}{\langle u_+\rangle^{1+\sigma}}d{u_+}$ is finite, we deduce from \eqref{flux 1} that
\beq\label{integration form of flux}\begin{aligned}
&\int_0^t\int_{\Sigma_\tau}\f{\langle u_-\rangle^{2(1+\sigma)}}{\langle u_+\rangle^{1+\sigma}}|f_+|^2dxd\tau\lesssim\int_{\Sigma_0}\langle u_-\rangle^{2(1+\sigma)}|f_+|^2dx\\
&\quad+\int\int_{W_t}|z_-^1||\p_1\bigl(\langle u_-\rangle^{2(1+\sigma)}\bigr)||f_+|^2dxd\tau+\int_{\R}\f{1}{\langle u_+\rangle^{1+\sigma}}\Bigl|\int\int_{W_t^{[u_+,\infty]}}\r_+\cdot\langle u_-\rangle^{2(1+\sigma)}f_+dxd\tau\Bigr| d{u_+}.
\end{aligned}\eeq

Combining \eqref{L1} and \eqref{integration form of flux}, and using the fact that $|\p_1\bigl(\langle u_-\rangle^{2(1+\sigma)}\bigr)|\lesssim \langle u_-\rangle^{1+2\sigma}$, we obtain that
\beno\begin{aligned}
&\sup_{0\leq\tau\leq t}\int_{\Sigma_\tau}\langle u_-\rangle^{2(1+\sigma)}|f_+|^2dx+\int_0^t\int_{\Sigma_\tau}\f{\langle u_-\rangle^{2(1+\sigma)}}{\langle u_+\rangle^{1+\sigma}}|f_+|^2dxd\tau\\
&\lesssim\int_{\Sigma_0}\langle u_-\rangle^{2(1+\sigma)}|f_+|^2dx+\int_0^t\int_{\Sigma_\tau}\langle u_-\rangle^{1+2\sigma}|z_-^1||f_+|^2dxd\tau\\
&\qquad+\Bigl|\int_0^t\int_{\Sigma_\tau}\r_+\cdot\langle u_-\rangle^{2(1+\sigma)}f_+dxd\tau\Bigr|
+\int_{\R}\f{1}{\langle u_+\rangle^{1+\sigma}}\Bigl|\int\int_{W_t^{[u_+,\infty]}}\r_+\cdot\langle u_-\rangle^{2(1+\sigma)}f_+dxd\tau\Bigr| d{u_+}.
\end{aligned}\eeno
Similar estimate holds for $f_-$. Then we arrive at \eqref{estimate for linearized system}. Since $\int_{\R}\f{1}{\langle u_\pm\rangle^{1+\sigma}}du_\pm<\infty$, we get \eqref{estimate for linearized system 0}. The proposition is proved.
\end{proof}

\subsection{The {\it a priori} estimates for the solutions to the MHD system \eqref{MHD}}
In this subsection, we shall use Proposition \ref{linearized prop} to derive the {\it a priori} estimates for the solutions of \eqref{MHD}. We shall derive the uniform estimates with respect to  $\delta$. In order to do that, thanks to \eqref{scaling}, we introduce the following energy functional
\beno
\sup_{t>0}\bigl[\sum_{k+l\leq N_*}\delta^{2(l-\f12)}E_\pm^{(k,l)}(\zpm(t))+\sum_{k\leq N_*-1}\delta^{-3}E_\pm^{(k,0)}(\zpm^3(t))\bigr].
\eeno
However, to control the above norms, we need other energy in the l.h.s of \eqref{total energy estimates}.

\subsubsection{The uniform estimates of $\na_h^kz_\pm$.}
The estimates for $\na_h^kz_\pm$ is stated in the following proposition.
\begin{proposition}\label{al h prop}
Assume that $(z_+,z_-)$ are the smooth solutions to \eqref{MHD}. Let $N_*=2N$, $N\in\Z_{\geq 5}$, and
\beq\label{assumption 2}
\|z_\pm^1\|_{L^\infty_tL^\infty_x}\leq 1.
\eeq
Then we have
\beq\label{estimate for al h}\begin{aligned}
&\quad\sum_{k\leq N_*}\delta^{-1}[E_\pm^{(k,0)}(z_\pm)+F_\pm^{(k,0)}(z_\pm)]\\
&\lesssim\sum_{k\leq N_*}\delta^{-1}E_\pm^{(k,0)}(z_{\pm,0})
+\Bigl(\sum_{k+l\leq N_*}\delta^{l-\f12}\bigl(E_\mp^{(k,l)}(z_\mp)\bigr)^{\f12}+\sum_{k\leq N_*-1}\delta^{-\f32}\bigl(E_\mp^{(k,0)}(z_\mp^3)\bigr)^{\f12}\\
&\quad+\sum_{k\leq N+2}\delta^{-\f12}\bigl(E_\mp^{(k,0)}(\p_3z_\mp)\bigr)^{\f12}\Bigr)\cdot\Bigl(\sum_{k+l\leq N_*}\delta^{2(l-\f12)}F_\pm^{(k,l)}(z_\pm)
+\sum_{k+l\leq N+2}\delta^{2(l-\f12)}F_\pm^{(k,l)}(\p_3z_\pm)\Bigr).
\end{aligned}\eeq
\end{proposition}
\begin{proof}
We shall divide the proof into several steps.

{\bf Step 1. Linearized system.}
We shall denote $\al_h=(\al_1,\al_2)\in(\Z_{\geq0})^2$ and
\beno
z_\pm^{(\al_h,0)}\eqdefa\p_h^{\al_h}z_\pm=\p_1^{\al_1}\p_2^{\al_2}z_\pm.
\eeno

Applying $\p_h^{\al_h}$ to both sides of \eqref{MHD}, we have
\begin{equation}\label{Linerize MHD for al h}
\begin{aligned}
&\p_t z_+^{(\al_h,0)}-\p_1z_+^{(\al_h,0)}+z_-\cdot\na z_+^{(\al_h,0)}=\r_+^{(\al_h,0)}-\p_h^{\al_h}\na p,\quad\text{in}\quad\Omega_\delta,\\
&\p_t z_-^{(\al_h,0)}+\p_1z_-^{(\al_h,0)}+z_+\cdot\na z_-^{(\al_h,0)}=\r_-^{(\al_h,0)}-\p_h^{\al_h}\na p,
\end{aligned}
\end{equation}
where
\beq\label{nonlinear for al h}\begin{aligned}
&\r_+^{(\al_h,0)}=-\p_h^{\al_h}(z_-\cdot\na z_+)+z_-\cdot\na\p_h^{\al_h} z_+,\\
&\r_-^{(\al_h,0)}=-\p_h^{\al_h}(z_+\cdot\na z_-)+z_+\cdot\na\p_h^{\al_h} z_-.
\end{aligned}\eeq

We only give the estimates for $z_+^{(\al_h,0)}$. Applying Proposition \ref{linearized prop} to the first equation of \eqref{Linerize MHD for al h}, we obtain that
\beq\label{al h 1}\begin{aligned}
&\sup_{0\leq\tau\leq t}\int_{\Sigma_\tau}\langle u_-\rangle^{2(1+\sigma)}|z_+^{(\al_h,0)}|^2dx+\int_0^t\int_{\Sigma_\tau}\f{\langle u_-\rangle^{2(1+\sigma)}}{\langle u_+\rangle^{1+\sigma}}|z_+^{(\al_h,0)}|^2dxd\tau\\
&\lesssim\int_{\Sigma_0}\langle u_-\rangle^{2(1+\sigma)}|z_+^{(\al_h,0)}|^2dx+\int_0^t\int_{\Sigma_\tau}\langle u_-\rangle^{1+2\sigma}|z_-^1||z_+^{(\al_h,0)}|^2dxd\tau\\
&\qquad+\Bigl|\int_0^t\int_{\Sigma_\tau}(\r_+^{(\al_h,0)}-\p_h^{\al_h}\na p)\cdot\langle u_-\rangle^{2(1+\sigma)}z_+^{(\al_h,0)} dxd\tau\Bigr|\\
&\qquad
+\int_{\R}\f{1}{\langle u_+\rangle^{1+\sigma}}\Bigl|\int\int_{W_t^{[u_+,\infty]}}(\r_+^{(\al_h,0)}-\p_h^{\al_h}\na p)\cdot\langle u_-\rangle^{2(1+\sigma)}z_+^{(\al_h,0)} dxd\tau\Bigr| d{u_+}.
\end{aligned}\eeq

{\bf Step 2. Estimates of the nonlinear terms.} In this step, we estimate the nonlinear terms in the r.h.s of \eqref{al h 1} term by term.

{\it Step 2.1. Estimate of $\int_0^t\int_{\Sigma_\tau}\langle u_-\rangle^{2(1+\sigma)}|z_-^1|\cdot|z_+^{(\al_h,0)}|^2dxd\tau$.} It is bounded by
\beno\begin{aligned}
&\|\langle u_+\rangle^{1+\sigma}z_-^1\|_{L^\infty_tL^\infty_x}\|\f{\langle u_-\rangle^{1+\sigma}}{\langle u_+\rangle^{\f12(1+\sigma)}}z_+^{(\al_h,0)}\|_{L^2_tL^2_x}^2\\
&\stackrel{\text{Sobolev}}{\lesssim}\sum_{k=0}^2\Bigl(\delta^{-\f12}\bigl(E_-^{(k,0)}(z_-)\bigr)^{\f12}
+\delta^{\f12}\bigl(E_-^{(k,1)}(z_-)\bigr)^{\f12}\Bigr)\cdot F_+^{(\al_h,0)}(z_+).
\end{aligned}\eeno
Then we obtain that
\beq\label{est for pressure 16}\begin{aligned}
&\delta^{-1}\int_0^t\int_{\Sigma_\tau}\langle u_-\rangle^{2(1+\sigma)}|z_-^1|\cdot|z_+^{(\al_h,0)}|^2dxd\tau
\lesssim\sum_{k+l\leq 2N}\delta^{l-\f12}\bigl(E_-^{(k,l)}(z_-)\bigr)^{\f12}
\cdot\delta^{-1} F_+^{(\al_h,0)}(z_+).
\end{aligned}\eeq

{\it Step 2.2. Estimate of term $\int_0^t\int_{\Sigma_\tau}\p_h^{\al_h}\na p\cdot\langle u_-\rangle^{2(1+\sigma)}z_+^{(\al_h,0)} dxd\tau$.} Firstly, we deal with the case $|\al_h|\geq 1$.
Since $\div z_+=0, z_+^3|_{x_3=\pm\delta}=0$, using the integration by parts, we have
\beq\label{est for pressure 0}\begin{aligned}
&|\int_0^t\int_{\Sigma_\tau}\p_h^{\al_h}\na p\cdot\langle u_-\rangle^{2(1+\sigma)}z_+^{(\al_h,0)} dxd\tau|
=|-\int_0^t\int_{\Sigma_\tau}\p_h^{\al_h}p\na(\langle u_-\rangle^{2(1+\sigma)})\cdot z_+^{(\al_h,0)} dxd\tau|\\
&\quad\lesssim\bigl(\int_0^t\int_{\Sigma_\tau}\langle u_-\rangle^{2(1+\sigma)}\langle u_+\rangle^{1+\sigma}|\p_h^{\al_h} p|^2dxd\tau\bigr)^{\f{1}{2}}\cdot\bigl(\int_0^t\int_{\Sigma_\tau}\f{\langle u_-\rangle^{2(1+\sigma)}}{\langle u_+\rangle^{1+\sigma}}|z_+^{(\al_h,0)}|^2dxd\tau\bigr)^{\f{1}{2}}.
\end{aligned}\eeq
Then we only need to control the term
\beno
\bigl(\int_0^t\int_{\Sigma_\tau}\langle u_-\rangle^{2(1+\sigma)}\langle u_+\rangle^{1+\sigma}|\p_h^{\al_h}p|^2dxd\tau\bigr)^{\f{1}{2}}
=\|\langle u_-\rangle^{1+\sigma}\langle u_+\rangle^{\f{1}{2}(1+\sigma)}\p_h^{\al_h}p\|_{L^2_tL^2_x}.
\eeno

Thanks to \eqref{expression of pressure}, using the facts that $\div\,z_\pm=0$ and $z_+^3|_{x_3=\pm\delta}=0$, $z_-^3|_{x_3=\pm\delta}=0$, and integrating by parts, we obtain by \eqref{Green transformation}
that
\beno
\p_h^{\al_h}p(\tau,x)=(-1)^{|\al_h|}\int_{\Omega_\delta}\p_i\p_j\p_h^{\al_h}G_\delta(x,y)(z_+^jz_-^i)(\tau,y)dy
=\int_{\Omega_\delta}\p_i\p_jG_\delta(x,y)\p_h^{\al_h}(z_+^jz_-^i)(\tau,y)dy.
\eeno
We choose a smooth cut-off function $\theta(r)$ so that
\beno
\theta(r)=\left\{\begin{aligned}
&1,\quad\text{for}\quad|r|\leq1,\\
&0,\quad\text{for}\quad|r|\geq2.
\end{aligned}\right.\eeno
Then we split $\p_h^{\al_h}p$ into two parts
\beno\begin{aligned}
\p_h^{\al_h}p(\tau,x)=&\underbrace{\int_{\Omega_\delta}\p_i\p_jG_\delta(x,y)\theta(|x_h-y_h|)\p_h^{\al_h}(z_+^jz_-^i)(\tau,y)dy}_{A_1(\tau,x)}\\
&\quad
+\underbrace{\int_{\Omega_\delta}\p_i\p_jG_\delta(x,y)(1-\theta(|x_h-y_h|))\p_h^{\al_h}(z_+^jz_-^i)(\tau,y)dy}_{A_2(\tau,x)}.
\end{aligned}\eeno

According to the decomposition, we have
\beno\begin{aligned}
&\|\langle u_-\rangle^{1+\sigma}\langle u_+\rangle^{\f{1}{2}(1+\sigma)}\p_h^{\al_h}p\|_{L^2_tL^2_x}\\
&\lesssim\|\langle u_-\rangle^{1+\sigma}\langle u_+\rangle^{\f{1}{2}(1+\sigma)}A_1(\tau,x)\|_{L^2_tL^2_x}+\|\langle u_-\rangle^{1+\sigma}\langle u_+\rangle^{\f{1}{2}(1+\sigma)}A_2(\tau,x)\|_{L^2_tL^2_x}.
\end{aligned}\eeno

{\it (i) Estimate of $A_1$.} Firstly, we have
\beno
A_1(\tau,x)=\sum_{\beta_h\leq\al_h}C_{\al_h,\beta_h}\underbrace{\int_{\Omega_\delta}\p_i\p_jG_\delta(x,y)\theta(|x_h-y_h|)
(\p_h^{\al_h-\beta_h}z_+^j\p_h^{\beta_h}z_-^i)(\tau,y)dy}_{A_{1,\al_h,\beta_h}(\tau,x)}.
\eeno

If {\bf $|\beta_h|\leq N$}, using the facts that $\div\,z_+=0$ and $z_+^3|_{x_3=\pm\delta}=0$, and integrating by parts, we have
\beno\begin{aligned}
A_{1,\al_h,\beta_h}(\tau,x)=&\underbrace{-\int_{\Omega_\delta}\p_iG_\delta(x,y)\theta(|x_h-y_h|)
(\p_h^{\al_h-\beta_h}z_+^j\p_h^{\beta_h}\p_jz_-^i)(\tau,y)dy}_{A_{1,\al_h,\beta_h}^1(\tau,x)}\\
&\quad\underbrace{-\int_{\Omega_\delta}\p_iG_\delta(x,y)\p_j\theta(|x_h-y_h|)
(\p_h^{\al_h-\beta_h}z_+^j\p_h^{\beta_h}z_-^i)(\tau,y)dy}_{A_{1,\al_h,\beta_h}^2(\tau,x)}.
\end{aligned}\eeno

For $A_{1,\al_h,\beta_h}^1$, using Corollary \ref{derivative of Green cor}, we have
\beq\label{p3}
|A_{1,\al_h,\beta_h}^1(\tau,x)|\lesssim\int_{-\delta}^\delta\int_{|x_h-y_h|\leq 2}\f{1}{\delta}\f{1}{|x_h-y_h|}
|(\p_h^{\al_h-\beta_h}z_+^j\p_h^{\beta_h}\p_jz_-^i)(\tau,y)|dy_hdy_3.
\eeq

Before going further, we need the following auxiliary lemma concerning the weights.

\begin{lemma}\label{weight 1}
For $|x_h-y_h|\leq 2$, we have
\beq\label{x-y<2}
\langle u_\pm(t,x_1)\rangle\leq\sqrt 5\langle u_\pm(t,y_1)\rangle.
\eeq
\end{lemma}
\begin{proof}
Since $u_\pm(t,x_1)=x_1\mp t$, we have
\beno
|u_\pm(t,x_1)|\leq |u_\pm(t,y_1)|+|x_1-y_1|\leq |u_\pm(t,y_1)|+2.
\eeno
Then we obtain
\beno
\langle u_\pm(t,x_1)\rangle\leq\bigl(1+2|u_\pm(t,y_1)|^2+4\bigr)^{\f12}\leq\sqrt5\langle u_\pm(t,y_1)\rangle.
\eeno
The lemma is proved.
\end{proof}

Thanks to \eqref{p3} and \eqref{x-y<2}, we have
\beno\begin{aligned}
&\langle u_-(\tau,x_1)\rangle^{1+\sigma}\langle u_+(\tau,x_1)\rangle^{\f12(1+\sigma)}|A_{1,\al_h,\beta_h}^1(\tau,x)|\\
&\lesssim \int_{-\delta}^\delta\int_{|x_h-y_h|\leq 2}\f{1}{\delta}\f{1}{|x_h-y_h|}
\Bigl(\langle u_-\rangle^{1+\sigma}\langle u_+\rangle^{\f12(1+\sigma)}|\p_h^{\al_h-\beta_h}z_+^j\p_h^{\beta_h}\p_jz_-^i|\Bigr)(\tau,y)dy_hdy_3.
\end{aligned}\eeno
Using Young inequality for the horizontal variables $x_h$ and using also H\"older inequality, we have
\beno\begin{aligned}
&\|\langle u_-\rangle^{1+\sigma}\langle u_+\rangle^{\f12(1+\sigma)}A_{1,\al_h,\beta_h}^1\|_{L_t^2L^2_x}\\
&\lesssim\delta^{-\f12}\|\f{1}{|x_h|}\|_{L_h^1(|x_h|\leq 2)}\|\f{\langle u_-\rangle^{1+\sigma}}{\langle u_+\rangle^{\f12(1+\sigma)}}|\p_h^{\al_h-\beta_h}z_+^j|\cdot\langle u_+\rangle^{1+\sigma}|\p_h^{\beta_h}\p_jz_-^i|\|_{L^2_tL^2_hL^1_v}\\
&\lesssim\delta^{-\f12}\|\langle u_+\rangle^{1+\sigma}\p_h^{\beta_h}\p_jz_-^i\|_{L^\infty_tL^\infty_hL^2_v}\|\f{\langle u_-\rangle^{1+\sigma}}{\langle u_+\rangle^{\f12(1+\sigma)}}\p_h^{\al_h-\beta_h}z_+^j\|_{L^2_tL^2_x}\\
&\lesssim\delta^{-\f12}\Bigl(\|\langle u_+\rangle^{1+\sigma}\p_h^{\beta_h}\na_hz_-\|_{L^\infty_tH^2_hL^2_v}\|\f{\langle u_-\rangle^{1+\sigma}}{\langle u_+\rangle^{\f12(1+\sigma)}}\p_h^{\al_h-\beta_h}z_+^h\|_{L^2_tL^2_x}\\
&\qquad+
\|\langle u_+\rangle^{1+\sigma}\p_h^{\beta_h}\p_3z_-\|_{L^\infty_tH^2_hL^2_v}\|\f{\langle u_-\rangle^{1+\sigma}}{\langle u_+\rangle^{\f12(1+\sigma)}}\p_h^{\al_h-\beta_h}z_+^3\|_{L^2_tL^2_x}\Bigr).
\end{aligned}\eeno
Then we have
\beq\label{p1}\begin{aligned}
&\|\langle u_-\rangle^{1+\sigma}\langle u_+\rangle^{\f12(1+\sigma)}A_{1,\al_h,\beta_h}^1\|_{L_t^2L^2_x}
\lesssim\delta^{-\f12}\Bigl(\sum_{k\leq 2}\bigl(E_-^{(|\beta_h|+1+k,0)}(z_-)\bigr)^{\f12}\bigl(F_+^{(\al_h-\beta_h,0)}(z_+^h)\bigr)^{\f12}\\
&\qquad
+\sum_{k\leq2}\bigl(E_-^{(|\beta_h|+k,1)}(z_-)\bigr)^{\f12}\bigl(F_+^{(\al_h-\beta_h,0)}(z_+^3)\bigr)^{\f12}\Bigr).
\end{aligned}\eeq

For $A_{1,\al_h,\beta_h}^2$, using Corollary \ref{derivative of Green cor}, we have
\beno
|A_{1,\al_h,\beta_h}^2(\tau,x)|\lesssim \int_{-\delta}^\delta\int_{1\leq|x_h-y_h|\leq 2}\f{1}{\delta}\f{1}{|x_h-y_h|}
\Bigl(|\p_h^{\al_h-\beta_h}z_+^j\p_h^{\beta_h}z_-^i|\Bigr)(\tau,y)dy_hdy_3.
\eeno
Similar argument to $A_{1,\al_h,\beta_h}^1$ can be applied  and then we have
\beq\label{p2}
\|\langle u_-\rangle^{1+\sigma}\langle u_+\rangle^{\f12(1+\sigma)}A_{1,\al_h,\beta_h}^2\|_{L_t^2L^2_x}
\lesssim\delta^{-\f12}\sum_{k\leq2}\bigl(E_-^{(|\beta_h|+k,0)}(z_-)\bigr)^{\f12}\bigl(F_+^{(\al_h-\beta_h,0)}(z_+)\bigr)^{\f12}.
\eeq

Thanks to \eqref{p1} and \eqref{p2},  for $|\beta_h|\leq N$, we have
\beq\label{est for pressure 1}\begin{aligned}
&\delta^{-\f12}\|\langle u_-\rangle^{1+\sigma}\langle u_+\rangle^{\f12(1+\sigma)}A_{1,\al_h,\beta_h}\|_{L_t^2L^2_x}\\
&\lesssim\Bigl(\sum_{k\leq N+3}\delta^{-\f12}\bigl(E_-^{(k,0)}(z_-)\bigr)^{\f12}+\sum_{k\leq N+2}\delta^{-\f12}\bigl(E_-^{(k,0)}(\p_3z_-)\bigr)^{\f12}\Bigr)\cdot\sum_{k\leq|\al_h|}\delta^{-\f12}\bigl(F_+^{(k,0)}(z_+)\bigr)^{\f12}.
\end{aligned}\eeq

\begin{remark}
Notice that  term $\sum_{k\leq N+2}\delta^{-\f12}\bigl(E_-^{(k,0)}(\p_3z_-)\bigr)^{\f12}\cdot\sum_{k\leq|\al_h|}\delta^{-\f12}\bigl(F_+^{(k,0)}(z_+)\bigr)^{\f12}$ on the r.h.s of \eqref{est for pressure 1} comes from the second term on the r.h.s of \eqref{p1}. Indeed, for
 $|\al_h|=N_*=2N, |\beta_h|=0$, we have no idea to give the upper bound for the term $\delta^{-\f32}\bigl(F_+^{(\al_h-\beta_h,0)}(z_+^3)\bigr)^{\f12}$ in \eqref{p1}.
\end{remark}

Similarly, for $|\beta_h|>N$, using the facts that $\div\,z_-=0$ and $z_-^3|_{x_3=\pm\delta}=0$, and integrating by parts, we have
\beno\begin{aligned}
A_{1,\al_h,\beta_h}=&-\int_{\Omega_\delta}\p_jG_\delta(x,y)\theta(|x_h-y_h|)
(\p_h^{\al_h-\beta_h}\p_iz_+^j\p_h^{\beta_h}z_-^i)(\tau,y)dy\\
&\quad-\int_{\Omega_\delta}\p_jG_\delta(x,y)\p_i\theta(|x_h-y_h|)
(\p_h^{\al_h-\beta_h}z_+^j\p_h^{\beta_h}z_-^i)(\tau,y)dy.
\end{aligned}\eeno
Then  for $|\beta_h|>N$, we have
\beq\label{est for pressure 2}\begin{aligned}
&\delta^{-\f12}\|\langle u_-\rangle^{1+\sigma}\langle u_+\rangle^{\f12(1+\sigma)}A_{1,\al_h,\beta_h}\|_{L_t^2L^2_x}\\
&\lesssim\sum_{k\leq|\al_h|}\delta^{-\f12}\bigl(E_-^{(k,0)}(z_-)\bigr)^{\f12}\cdot\Bigl(\sum_{k\leq N+2}\delta^{-\f12}\bigl(F_+^{(k,0)}(z_+)\bigr)^{\f12}+\sum_{k\leq N+1}\delta^{-\f12}\bigl(F_+^{(k,0)}(\p_3z_+)\bigr)^{\f12}\Bigr).
\end{aligned}\eeq
Thanks to \eqref{est for pressure 1} and \eqref{est for pressure 2}, notice that $N\geq 5$, we have
\beq\label{est for pressure 3}\begin{aligned}
&\delta^{-\f12}\|\langle u_-\rangle^{1+\sigma}\langle u_+\rangle^{\f12(1+\sigma)}A_1\|_{L_t^2L^2_x}
\lesssim\Bigl(\sum_{k\leq2N}\delta^{-\f12}\bigl(E_-^{(k,0)}(z_-)\bigr)^{\f12}+\sum_{k\leq N+2}\delta^{-\f12}\bigl(E_-^{(k,0)}(\p_3z_-)\bigr)^{\f12}\Bigr)\\
&\qquad\times\Bigl(\sum_{k\leq2N}\delta^{-\f12}\bigl(F_+^{(k,0)}(z_+)\bigr)^{\f12}+\sum_{k\leq N+2}\delta^{-\f12}\bigl(F_+^{(k,0)}(\p_3z_+)\bigr)^{\f12}\Bigr).
\end{aligned}\eeq

\medskip

{\it (ii) Estimate of $A_2$.} For $A_2$, integrating by parts,  for $\gamma_h\leq\al_h$ and $|\gamma_h|=1$,  we have
\beno\begin{aligned}
A_2(\tau,x)=&-\underbrace{\int_{\Omega_\delta}\p_h^{\gamma_h}\p_i\p_jG_\delta(x,y)(1-\theta(|x_h-y_h|))
\p_h^{\al_h-\gamma_h}(z_+^jz_-^i)(\tau,y)dy}_{A_2^1(\tau,x)}\\
&+\underbrace{\int_{\Omega_\delta}\p_i\p_jG_\delta(x,y)\p_h^{\gamma_h}\theta(|x_h-y_h|)
\p_h^{\al_h-\gamma_h}(z_+^jz_-^i)(\tau,y)dy}_{A_2^2(\tau,x)}.
\end{aligned}\eeno

For $A_2^2$, the integration happens for $y\in\{|x_h-y_h|\leq 2\}\times(-\delta,\delta)$. Similar derivation as that for $A_1$, replacing $\theta(|x_h-y_h|)$ and $\al_h$ in $A_1$ by $\p_h^{\gamma_h}\theta(|x_h-y_h|)$ and $\al_h-\gamma_h$,  gives rise to
\beq\label{est for pressure 4}\begin{aligned}
&\delta^{-\f12}\|\langle u_-\rangle^{1+\sigma}\langle u_+\rangle^{\f12(1+\sigma)}A_2^2\|_{L_t^2L^2_x}
\lesssim\Bigl(\sum_{k\leq2N}\delta^{-\f12}\bigl(E_-^{(k,0)}(z_-)\bigr)^{\f12}+\sum_{k\leq N+2}\delta^{-\f12}\bigl(E_-^{(k,0)}(\p_3z_-)\bigr)^{\f12}\Bigr)\\
&\qquad\times\Bigl(\sum_{k\leq2N}\delta^{-\f12}\bigl(F_+^{(k,0)}(z_+)\bigr)^{\f12}+\sum_{k\leq N+2}\delta^{-\f12}\bigl(F_+^{(k,0)}(\p_3z_+)\bigr)^{\f12}\Bigr).
\end{aligned}\eeq

For $A_2^1$, using Corollary \ref{derivative of Green cor}, we have
\beq\label{p4}
|A_2^1(\tau,x)|\lesssim\int_{-\delta}^\delta\int_{|x_h-y_h|\geq 1}\f{1}{\delta}\f{1}{|x_h-y_h|^3}|\p_h^{\al_h-\gamma_h}(z_+^jz_-^i)(\tau,y)|dy_hdy_3.
\eeq

To estimate $A_2^1$, we need the following lemma.
\begin{lemma}\label{weight 2}
For $|x_h-y_h|\geq 1$, we have
\beq\label{x-y>1}
\langle u_\pm(t,x_1)\rangle\leq\sqrt 3|x_h-y_h|\langle u_\pm(t,y_1)\rangle.
\eeq
\end{lemma}
\begin{proof}
Since $u_\pm(t,x_1)=x_1\mp t$, we have
\beno
|u_\pm(t,x_1)|\leq |u_\pm(t,y_1)|+|x_1-y_1|.
\eeno
Then we have for $|x_h-y_h|\geq 1$
\beno
\langle u_\pm(t,x_1)\rangle\leq\bigl(1+2|u_\pm(t,y_1)|^2+2|x_1-y_1|^2\bigr)^{\f12}\leq\sqrt3 |x_h-y_h|\bigl(1+|u_\pm(t,y_1)|^2\bigr)^{\f12}.
\eeno
The lemma is proved.
\end{proof}

Thanks to \eqref{p4} and \eqref{x-y>1}, we have
\beno\begin{aligned}
&\langle u_-(\tau,x_1)\rangle^{1+\sigma}\langle u_+(\tau,x_1)\rangle^{\f12(1+\sigma)}|A_2^1(\tau,x)|\\
&\lesssim \int_{-\delta}^\delta\int_{|x_h-y_h|\geq 1}\f{1}{\delta}\f{1}{|x_h-y_h|^{3-\f{3}{2}(1+\sigma)}}
\Bigl(\langle u_-\rangle^{1+\sigma}\langle u_+\rangle^{\f12(1+\sigma)}|\p_h^{\al_h-\gamma_h}(z_+^jz_-^i)|\Bigr)(\tau,y)dy_hdy_3.
\end{aligned}\eeno
Using Young inequality and H\"older inequality, we have
\beno\begin{aligned}
&\|\langle u_-\rangle^{1+\sigma}\langle u_+\rangle^{\f12(1+\sigma)}A_2^1\|_{L^2_tL^2_x}\\
&\lesssim\delta^{-\f12}\|\f{1}{|x_h|^{\f{3}{2}(1-\sigma)}}\|_{L^2_h(|x_h|\geq 1)}\|\langle u_-\rangle^{1+\sigma}\langle u_+\rangle^{\f12(1+\sigma)}\p_h^{\al_h-\gamma_h}(z_+^jz_-^i)\|_{L^2_tL^1_x}.
\end{aligned}\eeno

On one hand, since $\sigma\in(0,\f13)$, $\|\f{1}{|x_h|^{\f{3}{2}(1-\sigma)}}\|_{L^2_h(|x_h|\geq 1)}<\infty$. On the other hand, there holds
\beno\begin{aligned}
&\|\langle u_-\rangle^{1+\sigma}\langle u_+\rangle^{\f12(1+\sigma)}\p_h^{\al_h-\gamma_h}(z_+^jz_-^i)\|_{L^2_tL^1_x}
\lesssim\sum_{\beta_h\leq\al_h-\gamma_h}\|\langle u_-\rangle^{1+\sigma}\langle u_+\rangle^{\f12(1+\sigma)}\p_h^{\al_h-\gamma_h-\beta_h}z_+^j\p_h^{\beta_h}z_-^i\|_{L^2_tL^1_x}\\
&\lesssim\sum_{\beta_h\leq\al_h-\gamma_h}\|\langle u_+\rangle^{1+\sigma}\p_h^{\beta_h}z_-\|_{L^\infty_tL^2_x}\cdot
\|\f{\langle u_-\rangle^{1+\sigma}}{\langle u_+\rangle^{\f12(1+\sigma)}}\p_h^{\al_h-\gamma_h-\beta_h}z_+\|_{L^2_tL^2_x}\\
&\lesssim\sum_{k\leq|\al_h|-1}\bigl(E_-^{(k,0)}(z_-)\bigr)^{\f12}\cdot\sum_{k\leq |\al_h|-1}\bigl(F_+^{(k,0)}(z_+)\bigr)^{\f12}.
\end{aligned}\eeno
Then we have
\beq\label{est for pressure 5}
\delta^{-\f12}\|\langle u_-\rangle^{1+\sigma}\langle u_+\rangle^{\f12(1+\sigma)}A_2^1\|_{L^2_tL^2_x}\lesssim\sum_{k\leq2N}\delta^{-\f12}\bigl(E_-^{(k,0)}(z_-)\bigr)^{\f12}\cdot\sum_{k\leq 2N}\delta^{-\f12}\bigl(F_+^{(k,0)}(z_+)\bigr)^{\f12}.
\eeq

Due to \eqref{est for pressure 4} and \eqref{est for pressure 5}, we have
\beq\label{est for pressure 6}\begin{aligned}
&\delta^{-\f12}\|\langle u_-\rangle^{1+\sigma}\langle u_+\rangle^{\f12(1+\sigma)}A_2\|_{L_t^2L^2_x}
\lesssim\Bigl(\sum_{k\leq2N}\delta^{-\f12}\bigl(E_-^{(k,0)}(z_-)\bigr)^{\f12}+\sum_{k\leq N+2}\delta^{-\f12}\bigl(E_-^{(k,0)}(\p_3z_-)\bigr)^{\f12}\Bigr)\\
&\qquad\times\Bigl(\sum_{k\leq2N}\delta^{-\f12}\bigl(F_+^{(k,0)}(z_+)\bigr)^{\f12}+\sum_{k\leq N+2}\delta^{-\f12}\bigl(F_+^{(k,0)}(\p_3z_+)\bigr)^{\f12}\Bigr).
\end{aligned}\eeq

Thanks to \eqref{est for pressure 3} and \eqref{est for pressure 6}, we obtain that
$\delta^{-\f12}\|\langle u_-\rangle^{1+\sigma}\langle u_+\rangle^{\f{1}{2}(1+\sigma)}\p_h^{\al_h}p\|_{L^2_tL^2_x}$
is bounded by the r.h.s of  \eqref{est for pressure 6}. Then due to \eqref{est for pressure 0},  for $|\al_h|\geq 1$, we have
\beq\label{est for pressure 7}\begin{aligned}
&\delta^{-1}|\int_0^t\int_{\Sigma_\tau}\p_h^{\al_h}\na p\cdot\langle u_-\rangle^{2(1+\sigma)}z_+^{(\al_h,0)} dxd\tau|\\
&
\lesssim\Bigl(\sum_{k\leq2N}\delta^{-\f12}\bigl(E_-^{(k,0)}(z_-)\bigr)^{\f12}+\sum_{k\leq N+2}\delta^{-\f12}\bigl(E_-^{(k,0)}(\p_3z_-)\bigr)^{\f12}\Bigr)\\
&\quad\times\Bigl(\sum_{k\leq2N}\delta^{-1}F_+^{(k,0)}(z_+)+\sum_{k\leq N+2}\delta^{-1}F_+^{(k,0)}(\p_3z_+)\Bigr).
\end{aligned}\eeq

Actually, \eqref{est for pressure 7} holds also for $|\al_h|=0$. Indeed,
for case $|\al_h|=0$, we have
\beq\label{est for pressure 8}\begin{aligned}
&|\int_0^t\int_{\Sigma_\tau}\na p\cdot\langle u_-\rangle^{2(1+\sigma)}z_+ dxd\tau|\\
&\lesssim\|\langle u_-\rangle^{1+\sigma}\langle u_+\rangle^{\f12(1+\sigma)}\na p\|_{L^2_tL^2_x}\cdot\bigl(\int_0^t\int_{\Sigma_\tau}\f{\langle u_-\rangle^{2(1+\sigma)}}{\langle u_+\rangle^{1+\sigma}}|z_+|^2dxd\tau\bigr)^{\f{1}{2}}.
\end{aligned}\eeq
By similar derivation as that for $\delta^{-\f12}\|\langle u_-\rangle^{1+\sigma}\langle u_+\rangle^{\f{1}{2}(1+\sigma)}\p_h^{\al_h}p\|_{L^2_tL^2_x}$, we could bound $\delta^{-\f12}\|\langle u_-\rangle^{1+\sigma}\langle u_+\rangle^{\f12(1+\sigma)}\na p\|_{L^2_tL^2_x}$ by the r.h.s of \eqref{est for pressure 6}. Then $\delta^{-1}|\int_0^t\int_{\Sigma_\tau}\na p\cdot\langle u_-\rangle^{2(1+\sigma)}z_+ dxd\tau|$ is bounded by the r.h.s of \eqref{est for pressure 7}.

{\it Step 2.3. Estimate of term $\int_{\R}\f{1}{\langle u_+\rangle^{1+\sigma}}\Bigl|\int\int_{W_t^{[u_+,\infty]}}\p_h^{\al_h}\na p\cdot\langle u_-\rangle^{2(1+\sigma)}z_+^{(\al_h,0)} dxd\tau\Bigr| d{u_+}$.}
Using the facts that $\div\,z_+=0$ and $z_+^3|_{x_3=\pm\delta}=0$,  and integrating by parts, for $|\al_h|\geq 1$, we have
\beno\begin{aligned}
&\quad\int\int_{W_t^{[u_+,\infty]}}\p_h^{\al_h}\na p\cdot\langle u_-\rangle^{2(1+\sigma)}z_+^{(\al_h,0)} dxd\tau\\
&=-\int\int_{W_t^{[u_+,\infty]}}\p_h^{\al_h}p\cdot\p_1\langle u_-\rangle^{2(1+\sigma)}z_+^{1(\al_h,0)} dxd\tau
-\int_{C_{u_+}^+}\p_h^{\al_h}p\cdot\langle u_-\rangle^{2(1+\sigma)}z_+^{1(\al_h,0)}d\sigma_+.
\end{aligned}\eeno
Then we have
\beno\begin{aligned}
&\quad\int_{\R}\f{1}{\langle u_+\rangle^{1+\sigma}}\Bigl|\int\int_{W_t^{[u_+,\infty]}}\p_h^{\al_h}\na p\cdot\langle u_-\rangle^{2(1+\sigma)}z_+^{(\al_h,0)} dxd\tau\Bigr|d{u_+}\\
&\lesssim\int_{\R}\f{1}{\langle u_+\rangle^{1+\sigma}}d{u_+}\int\int_{W_t}|\p_h^{\al_h}p|\cdot\langle u_-\rangle^{2(1+\sigma)}|z_+^{(\al_h,0)}| dxd\tau\\
&\qquad
+\int_{\R}\f{1}{\langle u_+\rangle^{1+\sigma}}\int_{C_{u_+}^+}\p_h^{\al_h}p\cdot\langle u_-\rangle^{2(1+\sigma)}z_+^{1(\al_h,0)}d\sigma_+d{u_+}.
\end{aligned}\eeno
Thanks to the facts
$
d\sigma_+=\sqrt 2dx_2dx_3d\tau
$
and   \eqref{coordinate transform}, we have
\beno
\int_{\R}\f{1}{\langle u_+\rangle^{1+\sigma}}\int_{C_{u_+}^+}\p_h^{\al_h}p\cdot\langle u_-\rangle^{2(1+\sigma)}z_+^{1(\al_h,0)}d\sigma_+d{u_+}
=\sqrt2\int_0^t\int_{\Sigma_\tau}\f{\langle u_-\rangle^{2(1+\sigma)}}{\langle u_+\rangle^{1+\sigma}}\p_h^{\al_h}p\cdot z_+^{1(\al_h,0)}dxd\tau.
\eeno
Thus for $|\al_h|\geq1$, we have
\beno\begin{aligned}
&\quad\int_{\R}\f{1}{\langle u_+\rangle^{1+\sigma}}\Bigl|\int\int_{W_t^{[u_+,\infty]}}\p_h^{\al_h}\na p\cdot\langle u_-\rangle^{2(1+\sigma)}z_+^{(\al_h,0)} dxd\tau\Bigr|d{u_+}\\
&\lesssim\int_0^t\int_{\Sigma_\tau}|\p_h^{\al_h}p|\cdot\langle u_-\rangle^{2(1+\sigma)}|z_+^{(\al_h,0)}| dxd\tau.
\end{aligned}\eeno
And for $|\al_h|=0$, we have
\beno\begin{aligned}
&\quad\int_{\R}\f{1}{\langle u_+\rangle^{1+\sigma}}\Bigl|\int\int_{W_t^{[u_+,\infty]}}\na p\cdot\langle u_-\rangle^{2(1+\sigma)}z_+ dxd\tau\Bigr|d{u_+}\\
&\lesssim\int_0^t\int_{\Sigma_\tau}|\na p|\cdot\langle u_-\rangle^{2(1+\sigma)}|z_+| dxd\tau.
\end{aligned}\eeno

Therefore, the nonlinear term $\delta^{-1}\int_{\R}\f{1}{\langle u_+\rangle^{1+\sigma}}\Bigl|\int\int_{W_t^{[u_+,\infty]}}\p_h^{\al_h}\na p\cdot\langle u_-\rangle^{2(1+\sigma)}z_+^{(\al_h,0)} dxd\tau\Bigr| d{u_+}$ has the same estimates as $\delta^{-1}\Bigl|\int_0^t\int_{\Sigma_\tau}\p_h^{\al_h}\na p\cdot\langle u_-\rangle^{2(1+\sigma)}z_+^{(\al_h,0)} dxd\tau\Bigr|$ which is bounded by the r.h.s of \eqref{est for pressure 7}.

{\it Step 2.4. Estimate of term $\int_0^t\int_{\Sigma_\tau}\r_+^{(\al_h,0)}\cdot\langle u_-\rangle^{2(1+\sigma)}z_+^{(\al_h,0)} dxd\tau$.} Firstly, using H\"older inequality, we have
\beq\label{est for pressure 9}\begin{aligned}
&\Bigl|\int_0^t\int_{\Sigma_\tau}\r_+^{(\al_h,0)}\cdot\langle u_-\rangle^{2(1+\sigma)}z_+^{(\al_h,0)} dxd\tau\Bigr|\\
&\lesssim\|\langle u_-\rangle^{1+\sigma}\langle u_+\rangle^{\f12(1+\sigma)}\r_+^{(\al_h,0)}\|_{L^2_tL^2_x}
\cdot\|\f{\langle u_-\rangle^{1+\sigma}}{\langle u_+\rangle^{\f12(1+\sigma)}}z_+^{(\al_h,0)}\|_{L^2_tL^2_x}.
\end{aligned}\eeq
We only need to bound the term $\|\langle u_-\rangle^{1+\sigma}\langle u_+\rangle^{\f12(1+\sigma)}\r_+^{(\al_h,0)}\|_{L^2_tL^2_x}$.

The expression of $\r_+^{(\al_h,0)}$ shows that
\beno
|\r_+^{(\al_h,0)}|\lesssim\sum_{\beta_h<\al_h}|\p_h^{\al_h-\beta_h}z_-\cdot\na\p_h^{\beta_h}z_+|
\lesssim\sum_{\beta_h<\al_h}\bigl(\underbrace{|\p_h^{\al_h-\beta_h}z_-^h|\cdot|\na_h^{|\beta_h|+1}z_+|}_{I_1}
+\underbrace{|\p_h^{\al_h-\beta_h}z_-^3|\cdot|\na_h^{|\beta_h|}\p_3z_+|}_{I_2}\bigr).
\eeno
For $I_1$, if $|\beta_h|<N$, by H\"older inequality, we have
\beno
\|\langle u_-\rangle^{1+\sigma}\langle u_+\rangle^{\f12(1+\sigma)}I_1\|_{L^2_tL^2_x}\lesssim\|\langle u_+\rangle^{1+\sigma}\p_h^{\al_h-\beta_h}z_-^h\|_{L^\infty_tL^2_x}\|\f{\langle u_-\rangle^{1+\sigma}}{\langle u_+\rangle^{\f12(1+\sigma)}}\na_h^{|\beta_h|+1}z_+\|_{L^2_tL^\infty_x}.
\eeno
Thanks to the Sobolev inequality (see  \eqref{s2}), we have
\beno\begin{aligned}
&\|\f{\langle u_-\rangle^{1+\sigma}}{\langle u_+\rangle^{\f12(1+\sigma)}}\na_h^{|\beta_h|+1}z_+\|_{L^2_tL^\infty_x}\lesssim\sum_{k+l\leq 2}\delta^{l-\f12}\|\f{\langle u_-\rangle^{1+\sigma}}{\langle u_+\rangle^{\f12(1+\sigma)}}\na_h^{|\beta_h|+1+k}\p_3^lz_+\|_{L^2_tL^2_x}.
\end{aligned}\eeno
Then  for $|\beta_h|<N$, we obtain that
\beq\label{est for pressure 10}\begin{aligned}
&\delta^{-\f12}\|\langle u_-\rangle^{1+\sigma}\langle u_+\rangle^{\f12(1+\sigma)}I_1\|_{L^2_tL^2_x}
\lesssim\delta^{-\f12}\bigl(E_-^{(|\al_h-\beta_h|,0)}(z_-)\bigr)^{\f12}\cdot
\sum_{k+l\leq N+2}\delta^{l-\f12}\bigl(F_+^{(k,l)}(z_+)\bigr)^{\f12}.
\end{aligned}\eeq
Similarly,  for $|\beta_h|<N$, we have
\beq\label{est for pressure 11}\begin{aligned}
&\delta^{-\f12}\|\langle u_-\rangle^{1+\sigma}\langle u_+\rangle^{\f12(1+\sigma)}I_2\|_{L^2_tL^2_x}
\lesssim\delta^{-\f12}\bigl(E_-^{(|\al_h-\beta_h|,0)}(z_-^3)\bigr)^{\f12}\cdot
\sum_{k+l\leq N+1}\delta^{l-\f12}\bigl(F_+^{(k,l)}(\p_3z_+)\bigr)^{\f12}.
\end{aligned}\eeq

If $N\leq|\beta_h|\leq |\al_h|-1\leq 2N-1$, by the similar argument as  that for \eqref{est for pressure 10}, it holds that \beno
\|\langle u_-\rangle^{1+\sigma}\langle u_+\rangle^{\f12(1+\sigma)}I_1\|_{L^2_tL^2_x}\lesssim\|\langle u_+\rangle^{1+\sigma}\p_h^{\al_h-\beta_h}z_-^h\|_{L^\infty_tL^\infty_x}\|\f{\langle u_-\rangle^{1+\sigma}}{\langle u_+\rangle^{\f12(1+\sigma)}}\na_h^{|\beta_h|+1}z_+\|_{L^2_tL^2_x}.
\eeno
and then
\beq\label{est for pressure 12}\begin{aligned}
&\delta^{-\f12}\|\langle u_-\rangle^{1+\sigma}\langle u_+\rangle^{\f12(1+\sigma)}I_1\|_{L^2_tL^2_x}
\stackrel{\eqref{s2}}{\lesssim}\sum_{k+l\leq N+2}\delta^{l-\f12}\bigl(E_-^{(k,l)}(z_-)\bigr)^{\f12}
\cdot
\delta^{-\f12}\bigl(F_+^{(|\beta_h|+1,0)}(z_+)\bigr)^{\f12}.
\end{aligned}\eeq
For $N\leq|\beta_h|\leq |\al_h|-1\leq 2N-1$, we also have
\beq\label{est for pressure 13}\begin{aligned}
&\delta^{-\f12}\|\langle u_-\rangle^{1+\sigma}\langle u_+\rangle^{\f12(1+\sigma)}I_2\|_{L^2_tL^2_x}
\lesssim\sum_{k+l\leq N+2}\delta^{l-\f32}\bigl(E_-^{(k,l)}(z_-^3)\bigr)^{\f12}\cdot
\delta^{\f12}\bigl(F_+^{(|\beta_h|,1)}(z_+)\bigr)^{\f12}\\
&\stackrel{\div z_-=0}{\lesssim}\Bigl(\sum_{k\leq N+2}\delta^{-\f32}\bigl(E_-^{(k,0)}(z_-^3)\bigr)^{\f12}+\sum_{k+l\leq N+2}\delta^{l-\f12}\bigl(E_-^{(k,l)}(z_-^h)\bigr)^{\f12}\Bigr)
\cdot
\delta^{\f12}\bigl(F_+^{(|\beta_h|,1)}(z_+)\bigr)^{\f12}.
\end{aligned}\eeq

\bigskip

Thanks to \eqref{est for pressure 10}, \eqref{est for pressure 11}, \eqref{est for pressure 12} and \eqref{est for pressure 13},
 for $|\al_h|\leq 2N$, we obtain that
\beq\label{est for pressure 14}\begin{aligned}
&\delta^{-\f12}\|\langle u_-\rangle^{1+\sigma}\langle u_+\rangle^{\f12(1+\sigma)}\r_+^{(\al_h,0)}\|_{L^2_tL^2_x}\\
&\lesssim\Bigl(\sum_{k+l\leq 2N}\delta^{l-\f12}\bigl(E_-^{(k,l)}(z_-)\bigr)^{\f12}+\sum_{k\leq 2N-1}\delta^{-\f32}\bigl(E_-^{(k,0)}(z_-^3)\bigr)^{\f12}\Bigr)\\
&\qquad\times\Bigl(\sum_{k+l\leq 2N}\delta^{l-\f12}\bigl(F_+^{(k,l)}(z_+)\bigr)^{\f12}
+\sum_{k+l\leq N+2}\delta^{l-\f12}\bigl(F_+^{(k,l)}(\p_3z_+)\bigr)^{\f12}\Bigr).
\end{aligned}\eeq
Due to \eqref{est for pressure 9} and \eqref{est for pressure 14}, we obtain that
\beq\label{est for pressure 15}\begin{aligned}
&\delta^{-1}\Bigl|\int_0^t\int_{\Sigma_\tau}\r_+^{(\al_h,0)}\cdot\langle u_-\rangle^{2(1+\sigma)}z_+^{(\al_h,0)} dxd\tau\Bigr|\\
&\lesssim\Bigl(\sum_{k+l\leq 2N}\delta^{l-\f12}\bigl(E_-^{(k,l)}(z_-)\bigr)^{\f12}+\sum_{k\leq 2N-1}\delta^{-\f32}\bigl(E_-^{(k,0)}(z_-^3)\bigr)^{\f12}\Bigr)\\
&\qquad\times\Bigl(\sum_{k+l\leq 2N}\delta^{2(l-\f12)}F_+^{(k,l)}(z_+)
+\sum_{k+l\leq N+2}\delta^{2(l-\f12)}F_+^{(k,l)}(\p_3z_+)\Bigr).
\end{aligned}\eeq

{\it Step 2.5. Estimate of $\int_{\R}\f{1}{\langle u_+\rangle^{1+\sigma}}\Bigl|\int\int_{W_t^{[u_+,\infty]}}\r_+^{(\al_h,0)}\cdot\langle u_-\rangle^{2(1+\sigma)}z_+^{(\al_h,0)} dxd\tau\Bigr| d{u_+}$.} Since $\int_{\R}\f{1}{\langle u_+\rangle^{1+\sigma}}du_+<\infty$, the quantity is bounded by
\beno
\int_0^t\int_{\Sigma_\tau}|\r_+^{(\al_h,0)}|\cdot\langle u_-\rangle^{2(1+\sigma)}|z_+^{(\al_h,0)}|dxd\tau.
\eeno
Then it is has the same estimate as that for $\int_0^t\int_{\Sigma_\tau}\r_+^{(\al_h,0)}\cdot\langle u_-\rangle^{2(1+\sigma)}z_+^{(\al_h,0)} dxd\tau$.

\medskip

{\bf Step 3. The {\it a priori} estimate for $\na_h^kz_+$ for $k\leq 2N$.} Combining \eqref{al h 1} and the results from Step 2.1 to Step 2.5, we finally arrive at
\beno\begin{aligned}
&\sum_{k\leq N_*}\delta^{-1}[E_+^{(k,0)}(z_+)+F_+^{(k,0)}(z_+)]\\
&\lesssim\sum_{k\leq N_*}\delta^{-1}E_+^{(k,0)}(z_{+,0})
+\Bigl(\sum_{k+l\leq 2N}\delta^{l-\f12}\bigl(E_-^{(k,l)}(z_-)\bigr)^{\f12}+\sum_{k\leq 2N-1}\delta^{-\f32}\bigl(E_-^{(k,0)}(z_-^3)\bigr)^{\f12}\\
&\quad+\sum_{k\leq N+2}\delta^{-\f12}\bigl(E_-^{(k,0)}(\p_3z_-)\bigr)^{\f12}\Bigr)\cdot\Bigl(\sum_{k+l\leq 2N}\delta^{2(l-\f12)}F_+^{(k,l)}(z_+)
+\sum_{k+l\leq N+2}\delta^{2(l-\f12)}F_+^{(k,l)}(\p_3z_+)\Bigr).
\end{aligned}\eeno
This is exactly the inequality \eqref{estimate for al h} for $z_+$. Similar estimate holds for $z_-$. Then the proposition is proved.
\end{proof}

\subsubsection{The uniform estimates of $\na_h^k\p_3^lz_\pm$ for $k+l\leq N_*=2N$.}
In this subsection, we derive the uniform estimates concerning  $\na_h^k\p_3^lz_\pm$ with the coefficient $\delta^{l-\f12}$.
\begin{proposition}\label{al h l v prop}
Assume that $(z_+,z_-)$ are the smooth solutions to \eqref{MHD}. Let $N_*=2N$, $N\in\Z_{\geq 5}$ and  \eqref{assumption 1} hold.
We have
\beq\label{estimate for al h l v}\begin{aligned}
&\quad\sum_{k+l\leq N_*}\delta^{2(l-\f12)}[E_\pm^{(k,l)}(z_\pm)+F_\pm^{(k,l)}(z_\pm)]\\
&\lesssim\sum_{k+l\leq N_*}\delta^{2(l-\f12)}E_\pm^{(k,l)}(z_{\pm,0})
+\Bigl(\sum_{k+l\leq N_*}\delta^{l-\f12}\bigl(E_\mp^{(k,l)}(z_\mp)\bigr)^{\f12}+\sum_{k\leq N_*-1}\delta^{-\f32}\bigl(E_\mp^{(k,0)}(z_\mp^3)\bigr)^{\f12}\\
&\quad+\sum_{k+l\leq N+2}\delta^{l-\f12}\bigl(E_\mp^{(k,l)}(\p_3z_\mp)\bigr)^{\f12}\Bigr)\cdot\Bigl(\sum_{k+l\leq N_*}\delta^{2(l-\f12)}F_\pm^{(k,l)}(z_\pm)
+\sum_{k+l\leq N+2}\delta^{2(l-\f12)}F_\pm^{(k,l)}(\p_3z_\pm)\Bigr).
\end{aligned}\eeq
\end{proposition}
\begin{proof} We only prove \eqref{estimate for al h l v} for $z_+$. We shall divide the proof into several steps.

{\bf Step 1. Reduction of the uniform estimates.} For any $\al_h\in(\Z_{\geq0})^2$,
since  $\div\,\p_h^{\al_h}z_+=0$ and $\p_h^{\al_h}z_+^3|_{x_3=\pm\delta}=0$, using div-curl formula \eqref{div-curl-general}, for $k+l\leq N_*=2N$ and $l\geq 1$, we have

\beno\begin{aligned}
&\|\langle u_-\rangle^{1+\sigma}\p_h^{\al_h}\p_3^lz_+\|_{L^2_x}\lesssim\|\langle u_-\rangle^{1+\sigma}\na^l(\p_h^{\al_h}z_+)\|_{L^2_x}\\
&
\lesssim\sum_{l_1=0}^{l-1}\|\langle u_-\rangle^{1+\sigma}\na^{l_1}\p_h^{\al_h}(\curl z_+)\|_{L^2_x}+\|\langle u_-\rangle^{1+\sigma}\p_h^{\al_h}z_+\|_{L^2_x}.
\end{aligned}\eeno
Then we obtain that
\beq\label{hv1}\begin{aligned}
&\delta^{l-\f12}\|\langle u_-\rangle^{1+\sigma}\na_h^k\p_3^lz_+\|_{L^2_x}\\
&\lesssim\sum_{k_1+l_1\leq l-1}\delta^{l_1+\f12}\|\langle u_-\rangle^{1+\sigma}\na_h^{k+k_1}\p_3^{l_1}(\curl z_+)\|_{L^2_x}+\delta^{-\f12}\|\langle u_-\rangle^{1+\sigma}\na_h^kz_+\|_{L^2_x}.
\end{aligned}\eeq
Therefore we only need to give  the uniform estimate of the term
\beno
\delta^{l+\f12}\|\langle u_-\rangle^{1+\sigma}\na_h^k\p_3^lj_+\|_{L^2_x},\quad\text{for}\quad k+l\leq 2N-1,
\eeno
where $j_+=\curl z_+$ satisfies the first equation of \eqref{Voricity equation}, that is
\beq\label{voticity for z +}
\p_t  j_+ -\p_1j_++z_-\cdot\na j_+= -\na z_-\wedge\na z_+,
\eeq
with $\na z_-\wedge\na z_+=-\na z_-^k\times\p_k z_+$.

Let $\al_h\in(\Z_{\geq0})^2,\ l\in\Z_{\geq0}$ and $|\al_h|+l\leq 2N-1$. Applying $\p_h^{\al_h}\p_3^l$ to both sides of \eqref{voticity for z +}, we have
\beq\label{hv2}
\p_t  j_+^{(\al_h,l)}-\p_1j_+^{(\al_h,l)} +z_- \cdot \nabla j_+^{(\al_h,l)} =\r_{+,1}^{(\al_h,l)}+\r_{+,2}^{(\al_h,l)},
\eeq
where  $j_+^{(\al_h,l)}\eqdefa\p_h^{\al_h}\p_3^lj_+$ and
\beno\begin{aligned}
&\r_{+,1}^{(\al_h,l)}=-\p_h^{\al_h}\p_3^l(\na z_-\wedge\na z_+),\\
&\r_{+,2}^{(\al_h,l)}=-\p_h^{\al_h}\p_3^l(z_-\cdot\na j_+)+z_- \cdot \nabla \p_h^{\al_h}\p_3^lj_+.
\end{aligned}\eeno

Thanks to Proposition \ref{linearized prop}, we have
\beq\label{hv3}\begin{aligned}
&\sup_{0\leq\tau\leq t}\int_{\Sigma_\tau}\langle u_-\rangle^{2(1+\sigma)}|j_+^{(\al_h,l)}|^2dx
+\int_0^t\int_{\Sigma_\tau}\f{\langle u_-\rangle^{2(1+\sigma)}}{\langle u_+\rangle^{1+\sigma}}|j_+^{(\al_h,l)}|^2dxd\tau\\
&\lesssim\int_{\Sigma_0}\langle u_-\rangle^{2(1+\sigma)}|j_+^{(\al_h,l)}|^2dx+\int_0^t\int_{\Sigma_\tau}\langle u_-\rangle^{2(1+\sigma)}|z_-^1|\cdot|j_+^{(\al_h,l)}|^2dxd\tau\\
&\quad+\int_0^t\int_{\Sigma_\tau}(|\r_{+,1}^{(\al_h,l)}|+|\r_{+,2}^{(\al_h,l)}|)\cdot\langle u_-\rangle^{2(1+\sigma)}|j_+^{(\al_h,l)}|dxd\tau.
\end{aligned}\eeq

\medskip

{\bf Step 2. Estimates of the nonlinear terms in the r.h.s of \eqref{hv3}.}

{\it Step 2.1. Estimate of term $\int_0^t\int_{\Sigma_\tau}\langle u_-\rangle^{2(1+\sigma)}|z_-^1|\cdot|j_+^{(\al_h,l)}|^2dxd\tau$.} It is bounded by
\beno\begin{aligned}
&\|\langle u_+\rangle^{1+\sigma}z_-^1\|_{L^\infty_tL^\infty_x}\|\f{\langle u_-\rangle^{1+\sigma}}{\langle u_+\rangle^{\f12(1+\sigma)}}j_+^{(\al_h,l)}\|_{L^2_tL^2_x}^2
\stackrel{\text{Sobolev}}{\lesssim}\sum_{k_1+l_1\leq2}\delta^{l_1-\f12}\bigl(E_-^{(k_1,l_1)}(z_-)\bigr)^{\f12}\cdot F_+^{(\al_h,l)}(j_+).
\end{aligned}\eeno
Then we have
\beq\label{hv4}\begin{aligned}
&\delta^{2(l+\f12)}\int_0^t\int_{\Sigma_\tau}\langle u_-\rangle^{2(1+\sigma)}|z_-^1|\cdot|j_+^{(\al_h,l)}|^2dxd\tau
\lesssim\sum_{k_1+l_1\leq 2}\delta^{l_1-\f12}\bigl(E_-^{(k_1,l_1)}(z_-)\bigr)^{\f12}\cdot\delta^{2(l+\f12)} F_+^{(\al_h,l)}(j_+).
\end{aligned}\eeq

\medskip

{\it Step 2.2. Estimate of term $\int_0^t\int_{\Sigma_\tau}|\r_{+,1}^{(\al_h,l)}|\cdot\langle u_-\rangle^{2(1+\sigma)}|j_+^{(\al_h,l)}|dxd\tau$.}
By H\"older inequality, we have
\beq\label{hv4a}\begin{aligned}
&\int_0^t\int_{\Sigma_\tau}|\r_{+,1}^{(\al_h,l)}|\cdot\langle u_-\rangle^{2(1+\sigma)}|j_+^{(\al_h,l)}|dxd\tau
\lesssim\|\langle u_-\rangle^{1+\sigma}\langle u_+\rangle^{\f12(1+\sigma)}\r_{+,1}^{(\al_h,l)}\|_{L^2_tL^2_x}
\cdot\bigl(F_+^{(\al_h,l)}(j_+)\bigr)^{\f12}.
\end{aligned}\eeq

Now, we want to give the  upper bound of $\|\langle u_-\rangle^{1+\sigma}\langle u_+\rangle^{\f12(1+\sigma)}\r_{+,1}^{(\al_h,l)}\|_{L^2_tL^2_x}$.
By the definition of $\na z_-\wedge\na z_+$, and by using the fact that $\div\,z_-=0$,  we have
\beq\label{hv5}\begin{aligned}
|\r_{+,1}^{(\al_h,l)}|\lesssim&\sum_{l_1\leq l}\sum_{\beta_h\leq\al_h}\bigl(\underbrace{|\p_h^{\al_h-\beta_h}\p_3^{l-l_1}\na_hz_-^h|\cdot|\p_h^{\beta_h}\p_3^{l_1}\na_hz_+|}_{I^1_{\beta_h,l_1}}
+\underbrace{|\p_h^{\al_h-\beta_h}\p_3^{l-l_1}\na_hz_-|\cdot|\p_h^{\beta_h}\p_3^{l_1}\p_3z_+|}_{I^2_{\beta_h,l_1}}\\
&\quad+\underbrace{|\p_h^{\al_h-\beta_h}\p_3^{l-l_1}\p_3z_-^h|\cdot|\p_h^{\beta_h}\p_3^{l_1}\na_hz_+|}_{I^3_{\beta_h,l_1}}\bigr).
\end{aligned}\eeq

For $I^1_{\beta_h,l_1}$, if $|\beta_h|+l_1+1\leq N$, we have
\beno
\|\langle u_-\rangle^{1+\sigma}\langle u_+\rangle^{\f12(1+\sigma)}I^1_{\beta_h,l_1}\|_{L^2_tL^2_x}
\lesssim\|\langle u_+\rangle^{1+\sigma}\na_h^{|\al_h-\beta_h|+1}\p_3^{l-l_1}z_-^h\|_{L^\infty_tL^2_x}
\cdot\|\f{\langle u_-\rangle^{1+\sigma}}{\langle u_+\rangle^{\f12(1+\sigma)}}\na_h^{|\beta_h|+1}\p_3^{l_1}z_+\|_{L^2_tL^\infty_x}.
\eeno
By virtue of Lemma \ref{Sobolev} (ii), we have
\beno\begin{aligned}
&\|\f{\langle u_-\rangle^{1+\sigma}}{\langle u_+\rangle^{\f12(1+\sigma)}}\na_h^{|\beta_h|+1}\p_3^{l_1}z_+\|_{L^2_tL^\infty_x}
\lesssim\sum_{k_2+l_2\leq2}\delta^{l_2-\f12}\bigl(F_+^{(|\beta_h|+1+k_2,l_1+l_2)}(z_+)\bigr)^{\f12}\\
&\lesssim\delta^{-l_1}\sum_{k_2+l_2\leq N+2}\delta^{l_2-\f12}\bigl(F_+^{(k_2,l_2)}(z_+)\bigr)^{\f12}.
\end{aligned}\eeno
Since
\beno
\|\langle u_+\rangle^{1+\sigma}\na_h^{|\al_h-\beta_h|+1}\p_3^{l-l_1}z_-^h\|_{L^\infty_tL^2_x}
\lesssim\sum_{k_1\leq|\al_h|+1}\bigl(E_-^{(k_1,l-l_1)}(z_-)\bigr)^{\f12},
\eeno
 if $|\beta_h|+l_1+1\leq N$, then we have
\beq\label{hv6}\begin{aligned}
&\quad\delta^{l+\f12}\|\langle u_-\rangle^{1+\sigma}\langle u_+\rangle^{\f12(1+\sigma)}I^1_{\beta_h,l_1}\|_{L^2_tL^2_x}\\
&
\lesssim\sum_{k_1\leq|\al_h|+1}\delta^{l-l_1+\f12}\bigl(E_-^{(k_1,l-l_1)}(z_-)\bigr)^{\f12}
\cdot\sum_{k_2+l_2\leq N+2}\delta^{l_2-\f12}\bigl(F_+^{(k_2,l_2)}(z_+)\bigr)^{\f12}.
\end{aligned}\eeq
If $N+1\leq|\beta_h|+l_1+1\leq |\al_h|+l+1\leq 2N$, we have
\beno
\|\langle u_-\rangle^{1+\sigma}\langle u_+\rangle^{\f12(1+\sigma)}I^1_{\beta_h,l_1}\|_{L^2_tL^2_x}
\lesssim\|\langle u_+\rangle^{1+\sigma}\na_h^{|\al_h-\beta_h|+1}\p_3^{l-l_1}z_-^h\|_{L^\infty_tL^\infty_x}
\cdot\|\f{\langle u_-\rangle^{1+\sigma}}{\langle u_+\rangle^{\f12(1+\sigma)}}\na_h^{|\beta_h|+1}\p_3^{l_1}z_+\|_{L^2_tL^2_x}.
\eeno

Following the similar derivation as that for \eqref{hv6}, we have
\beno\begin{aligned}
&\|\langle u_+\rangle^{1+\sigma}\na_h^{|\al_h-\beta_h|+1}\p_3^{l-l_1}z_-^h\|_{L^\infty_tL^\infty_x}
\lesssim\sum_{k_2+l_2\leq2}\delta^{l_2-\f12}\bigl(E_-^{(|\al_h-\beta_h|+1+k_2,l-l_1+l_2)}(z_-)\bigr)^{\f12}\\
&\lesssim\delta^{l_1-l}\sum_{k_2+l_2\leq N+2}\delta^{l_2-\f12}\bigl(E_-^{(k_2,l_2)}(z_-)\bigr)^{\f12},
\end{aligned}\eeno
and
\beno
\|\f{\langle u_-\rangle^{1+\sigma}}{\langle u_+\rangle^{\f12(1+\sigma)}}\na_h^{|\beta_h|+1}\p_3^{l_1}z_+\|_{L^2_tL^2_x}
\lesssim\sum_{k_1\leq|\al_h|+1}\bigl(F_+^{(k_1,l_1)}(z_+)\bigr)^{\f12}.
\eeno
Then we obtain  that for $N+1\leq|\beta_h|+l_1+1\leq |\al_h|+l+1\leq 2N$, it holds
\beq\label{hv7}\begin{aligned}
&\delta^{l+\f12}\|\langle u_-\rangle^{1+\sigma}\langle u_+\rangle^{\f12(1+\sigma)}I^1_{\beta_h,l_1}\|_{L^2_tL^2_x}\\
&\lesssim\sum_{k_2+l_2\leq N+2}\delta^{l_2-\f12}\bigl(E_-^{(k_2,l_2)}(z_-)\bigr)^{\f12}
\cdot\sum_{k_1\leq|\al_h|+1}\delta^{l_1+\f12}\bigl(F_+^{(k_1,l_1)}(z_+)\bigr)^{\f12}.
\end{aligned}\eeq

Similarly, for $I^2_{\beta_h,l_1}$ and $I^3_{\beta_h,l_1}$, if  $|\beta_h|+l_1+1\leq N$, then we have
\beq\label{hv8}\begin{aligned}
&\delta^{l+\f12}\|\langle u_-\rangle^{1+\sigma}\langle u_+\rangle^{\f12(1+\sigma)}I^2_{\beta_h,l_1}\|_{L^2_tL^2_x}\\
&\lesssim\sum_{k_1\leq|\al_h|+1}\delta^{l-l_1-\f12}\bigl(E_-^{(k_1,l-l_1)}(z_-)\bigr)^{\f12}
\cdot\sum_{k_2+l_2\leq N+1}\delta^{l_2+\f12}\bigl(F_+^{(k_2,l_2+1)}(z_+)\bigr)^{\f12},
\end{aligned}\eeq
\beq\label{hv9}\begin{aligned}
&\delta^{l+\f12}\|\langle u_-\rangle^{1+\sigma}\langle u_+\rangle^{\f12(1+\sigma)}I^3_{\beta_h,l_1}\|_{L^2_tL^2_x}\\
&\lesssim\sum_{k_1\leq|\al_h|}\delta^{l-l_1+\f12}\bigl(E_-^{(k_1,l-l_1+1)}(z_-)\bigr)^{\f12}
\cdot\sum_{k_2+l_2\leq N+2}\delta^{l_2-\f12}\bigl(F_+^{(k_2,l_2)}(z_+)\bigr)^{\f12}.
\end{aligned}\eeq
While for $N+1\leq|\beta_h|+l_1+1\leq |\al_h|+l+1\leq 2N$,
\beq\label{hv10}\begin{aligned}
&\delta^{l+\f12}\|\langle u_-\rangle^{1+\sigma}\langle u_+\rangle^{\f12(1+\sigma)}I^2_{\beta_h,l_1}\|_{L^2_tL^2_x}\\
&\lesssim\sum_{k_2+l_2\leq N+2}\delta^{l_2-\f12}\bigl(E_-^{(k_2,l_2)}(z_-)\bigr)^{\f12}
\cdot\sum_{k_1\leq|\al_h|}\delta^{l_1+\f12}\bigl(F_+^{(k_1,l_1+1)}(z_+)\bigr)^{\f12}.
\end{aligned}\eeq
\beq\label{hv11}\begin{aligned}
&\delta^{l+\f12}\|\langle u_-\rangle^{1+\sigma}\langle u_+\rangle^{\f12(1+\sigma)}I^3_{\beta_h,l_1}\|_{L^2_tL^2_x}\\
&\lesssim\sum_{k_2+l_2\leq N+1}\delta^{l_2+\f12}\bigl(E_-^{(k_2,l_2+1)}(z_-)\bigr)^{\f12}
\cdot\sum_{k_1\leq|\al_h|+1}\delta^{l_1-\f12}\bigl(F_+^{(k_1,l_1)}(z_+)\bigr)^{\f12}.
\end{aligned}\eeq

Thanks to \eqref{hv5} and inequalities \eqref{hv6} to \eqref{hv11}, noticing that $|\al_h|+l\leq 2N-1$ and $N\in\Z_{\geq5}$, we have
\beno\begin{aligned}
&\delta^{l+\f12}\|\langle u_-\rangle^{1+\sigma}\langle u_+\rangle^{\f12(1+\sigma)}\r_{+,1}^{(\al_h,l)}\|_{L^2_tL^2_x}\\
&
\lesssim\sum_{k_1+l_1\leq 2N}\delta^{l_1-\f12}\bigl(E_-^{(k_1,l_1)}(z_-)\bigr)^{\f12}
\cdot\sum_{k_1+l_1\leq 2N}\delta^{l_1-\f12}\bigl(F_+^{(k_1,l_1)}(z_+)\bigr)^{\f12}.
\end{aligned}\eeno
Then due to \eqref{hv4a} and the inequality
\beq\label{curl to z}
F_+^{(\al_h,l)}(j_+)\lesssim F_+^{(|\al_h|+1,l)}(z_+)+F_+^{(|\al_h|,l+1)}(z_+),
\eeq
we have
\beq\label{hv12}\begin{aligned}
&\delta^{2(l+\f12)}\int_0^t\int_{\Sigma_\tau}|\r_{+,1}^{(\al_h,l)}|\cdot\langle u_-\rangle^{2(1+\sigma)}|j_+^{(\al_h,l)}|dxd\tau\\
&
\lesssim\sum_{k_1+l_1\leq 2N}\delta^{l_1-\f12}\bigl(E_-^{(k_1,l_1)}(z_-)\bigr)^{\f12}
\cdot\sum_{k_1+l_1\leq 2N}\delta^{2(l_1-\f12)}F_+^{(k_1,l_1)}(z_+).
\end{aligned}\eeq

{\it Step 2.3. Estimate of term $\int_0^t\int_{\Sigma_\tau}|\r_{+,2}^{(\al_h,l)}|\cdot\langle u_-\rangle^{2(1+\sigma)}|j_+^{(\al_h,l)}|dxd\tau$.}
By H\"older inequality, we have
\beq\label{hv13}\begin{aligned}
&\int_0^t\int_{\Sigma_\tau}|\r_{+,2}^{(\al_h,l)}|\cdot\langle u_-\rangle^{2(1+\sigma)}|j_+^{(\al_h,l)}|dxd\tau
\lesssim\|\langle u_-\rangle^{1+\sigma}\langle u_+\rangle^{\f12(1+\sigma)}\r_{+,2}^{(\al_h,l)}\|_{L^2_tL^2_x}
\cdot\bigl(F_+^{(\al_h,l)}(j_+)\bigr)^{\f12}.
\end{aligned}\eeq

For $\r_{+,2}^{(\al_h,l)}$, we have
\beq\label{hv13a}
|\r_{+,2}^{(\al_h,l)}|\lesssim\bigl(\sum_{l_1\leq l}\sum_{\beta_h<\al_h}+\sum_{l_1<l}\sum_{\beta_h\leq\al_h}\bigr)
\bigl(\underbrace{|\p_h^{\al_h-\beta_h}\p_3^{l-l_1}z_-^h|\cdot|\na_h\p_h^{\beta_h}\p_3^{l_1}j_+|}_{II^1_{\beta_h,l_1}}
+\underbrace{|\p_h^{\al_h-\beta_h}\p_3^{l-l_1}z_-^3|\cdot|\p_3\p_h^{\beta_h}\p_3^{l_1}j_+|}_{II^2_{\beta_h,l_1}}\bigr).
\eeq

For $II^1_{\beta_h,l_1}$, if $|\al_h-\beta_h|+l-l_1\leq 2N-3$ (that is $|\beta_h|+l_1\geq|\al_h|+l-2N+3$), by similar derivation, we have
\beno\begin{aligned}
&\delta^{l+\f12}\|\langle u_-\rangle^{1+\sigma}\langle u_+\rangle^{\f12(1+\sigma)}II^1_{\beta_h,l_1}\|_{L^2_tL^2_x}\\
&\lesssim\sum_{k_2+l_2\leq2}\delta^{l-l_1+l_2-\f12}\bigl(E_-^{(|\al_h-\beta_h|+k_2,l-l_1+l_2)}(z_-^h)\bigr)^{\f12}\cdot
\delta^{l_1+\f12}\bigl(F_+^{(|\beta_h|+1,l_1)}(j_+)\bigr)^{\f12}.
\end{aligned}\eeno
Since $|\al_h-\beta_h|+l-l_1\leq 2N-3$, $|\beta_h|+l_1\leq|\al_h|+l-1$ and $k_2+l_2\leq2$, we have
\beno
|\al_h-\beta_h|+k_2+l-l_1+l_2\leq 2N-1,\quad |\beta_h|+l_1+1\leq|\al_h|+l\leq 2N-1.
\eeno
Then  for $|\beta_h|+l_1\geq|\al_h|+l-2N+3$, we get
\beq\label{hv14}\begin{aligned}
&\delta^{l+\f12}\|\langle u_-\rangle^{1+\sigma}\langle u_+\rangle^{\f12(1+\sigma)}II^1_{\beta_h,l_1}\|_{L^2_tL^2_x}\\
&\lesssim\sum_{k_2+l_2\leq 2N-1}\delta^{l_2-\f12}\bigl(E_-^{(k_2,l_2)}(z_-^h)\bigr)^{\f12}\cdot\sum_{k_2+l_2\leq 2N-1}\delta^{l_2+\f12}\bigl(F_+^{(k_2,l_2)}(j_+)\bigr)^{\f12}.
\end{aligned}\eeq
While if $|\beta_h|+l_1\leq|\al_h|+l-2N+2\leq1$, we have
\beq\label{hv15}\begin{aligned}
&\delta^{l+\f12}\|\langle u_-\rangle^{1+\sigma}\langle u_+\rangle^{\f12(1+\sigma)}II^1_{\beta_h,l_1}\|_{L^2_tL^2_x}\\
&\lesssim\delta^{l-l_1-\f12}\bigl(E_-^{(|\al_h-\beta_h|,l-l_1)}(z_-^h)\bigr)^{\f12}\cdot
\sum_{k_2+l_2\leq2}\delta^{l_1+l_2+\f12}\bigl(F_+^{(|\beta_h|+1+k_2,l_1+l_2)}(j_+)\bigr)^{\f12},\\
&\lesssim\sum_{k_2+l_2\leq 2N-1}\delta^{l_2-\f12}\bigl(E_-^{(k_2,l_2)}(z_-^h)\bigr)^{\f12}\cdot\sum_{k_2+l_2\leq4}\delta^{l_2+\f12}\bigl(F_+^{(k_2,l_2)}(j_+)\bigr)^{\f12}.
\end{aligned}\eeq
Thanks to \eqref{hv14}, \eqref{hv15} and the fact $N\in\Z_{\geq5}$, we get
\beq\label{hv16}\begin{aligned}
&\delta^{l+\f12}\|\langle u_-\rangle^{1+\sigma}\langle u_+\rangle^{\f12(1+\sigma)}II^1_{\beta_h,l_1}\|_{L^2_tL^2_x}\\
&\lesssim\sum_{k_2+l_2\leq 2N-1}\delta^{l_2-\f12}\bigl(E_-^{(k_2,l_2)}(z_-^h)\bigr)^{\f12}\cdot\sum_{k_2+l_2\leq 2N-1}\delta^{l_2+\f12}\bigl(F_+^{(k_2,l_2)}(j_+)\bigr)^{\f12}.
\end{aligned}\eeq

Similarly, for $II^2_{\beta_h,l_1}$, if $|\al_h-\beta_h|+l-l_1\leq 2N-3$ (that is $|\beta_h|+l_1\geq|\al_h|+l-2N+3$), it holds
\beno\begin{aligned}
&\delta^{l+\f12}\|\langle u_-\rangle^{1+\sigma}\langle u_+\rangle^{\f12(1+\sigma)}II^2_{\beta_h,l_1}\|_{L^2_tL^2_x}\\
&\lesssim\sum_{k_2+l_2\leq 2N-1}\delta^{l_2-\f32}\bigl(E_-^{(k_2,l_2)}(z_-^3)\bigr)^{\f12}\cdot\sum_{k_2+l_2\leq 2N-2}\delta^{l_2+\f32}\bigl(F_+^{(k_2,l_2+1)}(j_+)\bigr)^{\f12}.
\end{aligned}\eeno
While for $|\beta_h|+l_1\leq|\al_h|+l-2N+2\leq1$,  one has
\beno\begin{aligned}
&\|\langle u_-\rangle^{1+\sigma}\langle u_+\rangle^{\f12(1+\sigma)}II^2_{\beta_h,l_1}\|_{L^2_tL^2_x}\\
&\lesssim\sum_{k_2+l_2\leq 2N-1}\delta^{l_2-\f32}\bigl(E_-^{(k_2,l_2)}(z_-^3)\bigr)^{\f12}\cdot\sum_{k_2+l_2\leq3}\delta^{l_2+\f32}\bigl(F_+^{(k_2,l_2+1)}(j_+)\bigr)^{\f12}
\end{aligned}\eeno

Since $\div\, z_-=0$, we have
\beno\begin{aligned}
&\sum_{k_2+l_2\leq 2N-1}\delta^{l_2-\f32}\bigl(E_-^{(k_2,l_2)}(z_-^3)\bigr)^{\f12}\lesssim\sum_{k_2+l_2\leq 2N}\delta^{l_2-\f12}\bigl(E_-^{(k_2,l_2)}(z_-^h)\bigr)^{\f12}+\sum_{k_2\leq 2N-1}\delta^{-\f32}\bigl(E_-^{(k_2,0)}(z_-^3)\bigr)^{\f12}.
\end{aligned}\eeno
Then for $N\in\Z_{\geq5}$, we obtain that
\beq\label{hv18}\begin{aligned}
&\delta^{l+\f12}\|\langle u_-\rangle^{1+\sigma}\langle u_+\rangle^{\f12(1+\sigma)}II^2_{\beta_h,l_1}\|_{L^2_tL^2_x}\lesssim\Bigl(\sum_{k_2+l_2\leq 2N}\delta^{l_2-\f12}\bigl(E_-^{(k_2,l_2)}(z_-^h)\bigr)^{\f12}\\
&\qquad+\sum_{k_2\leq 2N-1}\delta^{-\f32}\bigl(E_-^{(k_2,0)}(z_-^3)\bigr)^{\f12}\Bigr)\cdot\sum_{k_2+l_2\leq 2N-1}\delta^{l_2+\f12}\bigl(F_+^{(k_2,l_2)}(j_+)\bigr)^{\f12}.
\end{aligned}\eeq

Thanks to \eqref{hv13}, \eqref{hv16} and \eqref{hv18}, using \eqref{curl to z},
we have
\beq\label{hv19}\begin{aligned}
&\delta^{2(l+\f12)}\int_0^t\int_{\Sigma_\tau}|\r_{+,2}^{(\al_h,l)}|\cdot\langle u_-\rangle^{2(1+\sigma)}|j_+^{(\al_h,l)}|dxd\tau\\
&
\lesssim\Bigl(\sum_{k_1+l_1\leq 2N}\delta^{l_1-\f12}\bigl(E_-^{(k_1,l_1)}(z_-)\bigr)^{\f12}+\sum_{k_1\leq 2N-1}\delta^{-\f32}\bigl(E_-^{(k_1,0)}(z_-^3)\bigr)^{\f12}\Bigr)
\cdot\sum_{k_1+l_1\leq 2N}\delta^{2(l_1-\f12)}F_+^{(k_1,l_1)}(z_+).
\end{aligned}\eeq

\medskip

{\bf Step 3. The {\it a priori} estimate of $\na_h^k\p_3^lz_+$ for $k+l\leq 2N$ and $l\geq 1$.} Combining \eqref{hv3}, \eqref{hv4}, \eqref{hv12}, \eqref{hv19} together with \eqref{curl to z},  for $k+l\leq 2N-1$, we obtain that
\beq\label{hv20}\begin{aligned}
&\delta^{2(l+\f12)}\Bigl(E_+^{(k,l)}(j_+)+F_+^{(k,l)}(j_+)\Bigr)\lesssim \delta^{2(l+\f12)}E_+^{(k,l)}(\curl z_{+,0})+\Bigl(\sum_{k_1+l_1\leq 2N}\delta^{l_1-\f12}\bigl(E_-^{(k_1,l_1)}(z_-)\bigr)^{\f12}\\
&\qquad+\sum_{k_1\leq 2N-1}\delta^{-\f32}\bigl(E_-^{(k_1,0)}(z_-^3)\bigr)^{\f12}\Bigr)
\cdot\sum_{k_1+l_1\leq 2N}\delta^{2(l_1-\f12)}F_+^{(k_1,l_1)}(z_+).
\end{aligned}\eeq

Thanks to \eqref{hv20}, \eqref{hv1},  \eqref{estimate for al h} and \eqref{curl to z}, we obtain the desired inequality \eqref{estimate for al h l v} for $z_+$. Similar estimate holds for $z_-$. Then it ends the proof of the proposition.
\end{proof}

\subsubsection{The uniform estimates of $\na_h^k\p_3^l(\p_3z_\pm)$ for $k+l\leq N+2$.}
In this subsection, we want to derive the uniform estimates concerning  $\na_h^k\p_3^l(\p_3z_\pm)$ with the lower order coefficient $\delta^{l-\f12}$. The main estimate is presented in the following proposition.
\begin{proposition}\label{al h l v-1 prop}
Assume that $(z_+,z_-)$ are the smooth solutions to \eqref{MHD}. Let $N_*=2N$, $N\in\Z_{\geq 5}$, and  \eqref{assumption 1} hold.
We have
\beq\label{estimate for al h l v-1}\begin{aligned}
&\quad\sum_{k+l\leq N+2}\delta^{2(l-\f12)}[E_\pm^{(k,l)}(\p_3z_\pm)+F_\pm^{(k,l)}(\p_3z_\pm)]\\
&\lesssim\sum_{k+l\leq N+2}\delta^{2(l-\f12)}E_\pm^{(k,l)}(\p_3z_{\pm,0})+\sum_{k+l\leq N_*}\delta^{2(l-\f12)}E_\pm^{(k,l)}(z_{\pm,0})\\
&\quad
+\Bigl(\sum_{k+l\leq N_*}\delta^{l-\f12}\bigl(E_\mp^{(k,l)}(z_\mp)\bigr)^{\f12}+\sum_{k\leq N_*-1}\delta^{-\f32}\bigl(E_\mp^{(k,0)}(z_\mp^3)\bigr)^{\f12}+\sum_{k+l\leq N+2}\delta^{l-\f12}\bigl(E_\mp^{(k,l)}(\p_3z_\mp)\bigr)^{\f12}\Bigr)\\
&\quad\times\Bigl(\sum_{k+l\leq N_*}\delta^{2(l-\f12)}F_\pm^{(k,l)}(z_\pm)
+\sum_{k+l\leq N+2}\delta^{2(l-\f12)}F_\pm^{(k,l)}(\p_3z_\pm)\Bigr).
\end{aligned}\eeq
\end{proposition}
\begin{proof}
To prove the proposition, we only need to modify the proof of Proposition \ref{al h l v prop}.

{\bf Step 1. Estimates of nonlinear terms in  the r.h.s of \eqref{hv3} for $|\al_h|+l\leq N+2$.}

{\it Step 1.1. Estimate of nonlinear term $\int_0^t\int_{\Sigma_\tau}\langle u_-\rangle^{2(1+\sigma)}|z_-^1|\cdot|j_+^{(\al_h,l)}|^2dxd\tau$.}
By virtue of \eqref{hv4}, we have
\beq\label{hv21}\begin{aligned}
&\delta^{2(l-\f12)}\int_0^t\int_{\Sigma_\tau}\langle u_-\rangle^{2(1+\sigma)}|z_-^1|\cdot|j_+^{(\al_h,l)}|^2dxd\tau\\
&\lesssim\sum_{k=0}^2\Bigl(\delta^{-\f12}\bigl(E_-^{(k,0)}(z_-)\bigr)^{\f12}
+\delta^{\f12}\bigl(E_-^{(k,1)}(z_-)\bigr)^{\f12}\Bigr)\cdot\delta^{2(l-\f12)} F_+^{(\al_h,l)}(j_+).
\end{aligned}\eeq

\medskip

{\it Step 1.2. Estimate of term $\int_0^t\int_{\Sigma_\tau}|\r_{+,1}^{(\al_h,l)}|\cdot\langle u_-\rangle^{2(1+\sigma)}|j_+^{(\al_h,l)}|dxd\tau$.}
Thanks to \eqref{hv4a}, we only need to derive the bound of $\|\langle u_-\rangle^{1+\sigma}\langle u_+\rangle^{\f12(1+\sigma)}\r_{+,1}^{(\al_h,l)}\|_{L^2_tL^2_x}$.

For $I^1_{\beta_h,l_1}=|\p_h^{\al_h-\beta_h}\p_3^{l-l_1}\na_hz_-^h|\cdot|\p_h^{\beta_h}\p_3^{l_1}\na_hz_+|$
and $I^2_{\beta_h,l_1}=|\p_h^{\al_h-\beta_h}\p_3^{l-l_1}\na_hz_-|\cdot|\p_h^{\beta_h}\p_3^{l_1}\p_3z_+|$ in \eqref{hv5}, by similar derivation of \eqref{hv7}, we have
\beq\label{hv22}\begin{aligned}
&\delta^{l-\f12}\|\langle u_-\rangle^{1+\sigma}\langle u_+\rangle^{\f12(1+\sigma)}I^1_{\beta_h,l_1}\|_{L^2_tL^2_x}\\
&\lesssim\sum_{k_2+l_2\leq|\al_h|+l+3}\delta^{l_2-\f12}\bigl(E_-^{(k_2,l_2)}(z_-)\bigr)^{\f12}
\cdot\sum_{k_1\leq|\al_h|+1}\delta^{l_1-\f12}\bigl(F_+^{(k_1,l_1)}(z_+)\bigr)^{\f12},
\end{aligned}\eeq
and
\beq\label{hv23}\begin{aligned}
&\delta^{l-\f12}\|\langle u_-\rangle^{1+\sigma}\langle u_+\rangle^{\f12(1+\sigma)}I^2_{\beta_h,l_1}\|_{L^2_tL^2_x}\\
&\lesssim\sum_{k_2+l_2\leq|\al_h|+l+3}\delta^{l_2-\f12}\bigl(E_-^{(k_2,l_2)}(z_-)\bigr)^{\f12}
\cdot\sum_{k_1\leq|\al_h|}\delta^{l_1-\f12}\bigl(F_+^{(k_1,l_1)}(\p_3z_+)\bigr)^{\f12}.
\end{aligned}\eeq

For $I^3_{\beta_h,l_1}=|\p_h^{\al_h-\beta_h}\p_3^{l-l_1}\p_3z_-^h|\cdot|\p_h^{\beta_h}\p_3^{l_1}\na_hz_+|$  in \eqref{hv5}, by similar derivation of \eqref{hv6}, we have
\beq\label{hv24}\begin{aligned}
&\delta^{l-\f12}\|\langle u_-\rangle^{1+\sigma}\langle u_+\rangle^{\f12(1+\sigma)}I^3_{\beta_h,l_1}\|_{L^2_tL^2_x}\\
&\lesssim\sum_{k_2+l_2\leq |\al_h|+l}\delta^{l_2-\f12}\bigl(E_-^{(k_2,l_2)}(\p_3z_-)\bigr)^{\f12}
\cdot\sum_{k_2+l_2\leq|\al_h|+l+3}\delta^{l_2-\f12}\bigl(F_+^{(k_2,l_2)}(z_+)\bigr)^{\f12}.
\end{aligned}\eeq

Thanks  to \eqref{hv22}, \eqref{hv23} and \eqref{hv24},  we obtain the estimate of
$\delta^{l-\f12}\|\langle u_-\rangle^{1+\sigma}\langle u_+\rangle^{\f12(1+\sigma)}\r_{+,1}^{(\al_h,l)}\|_{L^2_tL^2_x}$.
Then using \eqref{hv4a} and inequality \eqref{curl to z},
 for $|\al_h|+l\leq N+2$ and $N\in\Z_{\geq5}$, we get
\beq\label{hv25}\begin{aligned}
&\delta^{2(l-\f12)}\int_0^t\int_{\Sigma_\tau}|\r_{+,1}^{(\al_h,l)}|\cdot\langle u_-\rangle^{2(1+\sigma)}|j_+^{(\al_h,l)}|dxd\tau\\
&
\lesssim\Bigl(\sum_{k_1+l_1\leq 2N}\delta^{l_1-\f12}\bigl(E_-^{(k_1,l_1)}(z_-)\bigr)^{\f12}+\sum_{k_1+l_1\leq N+2}\delta^{l_1-\f12}\bigl(E_-^{(k_1,l_1)}(\p_3z_-)\bigr)^{\f12}\Bigr)\\
&\qquad
\cdot\Bigl(\sum_{k_1+l_1\leq 2N}\delta^{2(l_1-\f12)}F_+^{(k_1,l_1)}(z_+)+\sum_{k_1+l_1\leq N+2}\delta^{2(l_1-\f12)}F_+^{(k_1,l_1)}(\p_3z_+)\Bigr).
\end{aligned}\eeq

{\it Step 1.3. Estimate of term $\int_0^t\int_{\Sigma_\tau}|\r_{+,2}^{(\al_h,l)}|\cdot\langle u_-\rangle^{2(1+\sigma)}|j_+^{(\al_h,l)}|dxd\tau$.}
By virtue of \eqref{hv13}, we only need to control $\delta^{l-\f12}\|\langle u_-\rangle^{1+\sigma}\langle u_+\rangle^{\f12(1+\sigma)}\r_{+,2}^{(\al_h,l)}\|_{L^2_tL^2_x}$.

For $II^1_{\beta_h,l_1}=|\p_h^{\al_h-\beta_h}\p_3^{l-l_1}z_-^h|\cdot|\na_h\p_h^{\beta_h}\p_3^{l_1}j_+|$ and $II^2_{\beta_h,l_1}=|\p_h^{\al_h-\beta_h}\p_3^{l-l_1}z_-^3|\cdot|\p_h^{\beta_h}\p_3^{l_1+1}j_+|$ in \eqref{hv13a}, by similar derivation as that for \eqref{hv14}, we have
\beq\label{hv26}\begin{aligned}
&\delta^{l-\f12}\|\langle u_-\rangle^{1+\sigma}\langle u_+\rangle^{\f12(1+\sigma)}II^1_{\beta_h,l_1}\|_{L^2_tL^2_x}\\
&\lesssim\sum_{k_2+l_2\leq |\al_h|+l+2}\delta^{l_2-\f12}\bigl(E_-^{(k_2,l_2)}(z_-^h)\bigr)^{\f12}\cdot
\delta^{l_1-\f12}\bigl(F_+^{(|\beta_h|+1,l_1)}(j_+)\bigr)^{\f12},
\end{aligned}\eeq
and
\beq\label{hv27}\begin{aligned}
&\delta^{l-\f12}\|\langle u_-\rangle^{1+\sigma}\langle u_+\rangle^{\f12(1+\sigma)}II^2_{\beta_h,l_1}\|_{L^2_tL^2_x}\\
&\lesssim\sum_{k_2+l_2\leq |\al_h|+l+2}\delta^{l_2-\f32}\bigl(E_-^{(k_2,l_2)}(z_-^3)\bigr)^{\f12}\cdot
\delta^{l_1+\f12}\bigl(F_+^{(|\beta_h|,l_1+1)}(j_+)\bigr)^{\f12}\\
&\stackrel{\div\,z_-=0}{\lesssim}\Bigl(\sum_{k_2+l_2\leq |\al_h|+l+2}\delta^{l_2-\f12}\bigl(E_-^{(k_2,l_2)}(z_-^h)\bigr)^{\f12}
+\sum_{k_2\leq |\al_h|+l+2}\delta^{-\f32}\bigl(E_-^{(k_2,0)}(z_-^3)\bigr)^{\f12}\Bigr)\cdot
\delta^{l_1+\f12}\bigl(F_+^{(|\beta_h|,l_1+1)}(j_+)\bigr)^{\f12}.
\end{aligned}\eeq

Thanks to \eqref{hv26} and \eqref{hv27},  we obtain the estimates of $\delta^{l-\f12}\|\langle u_-\rangle^{1+\sigma}\langle u_+\rangle^{\f12(1+\sigma)}\r_{+,2}^{(\al_h,l)}\|_{L^2_tL^2_x}$. Since $|\al_h|+l\leq N+2$, $|\beta_h|+l_1\leq|\al_h|+l-1$ and $N\in \Z_{\geq5}$,   by using \eqref{hv13} and \eqref{curl to z}, we obtain that
\beq\label{hv28}\begin{aligned}
&\delta^{2(l-\f12)}\int_0^t\int_{\Sigma_\tau}|\r_{+,2}^{(\al_h,l)}|\cdot\langle u_-\rangle^{2(1+\sigma)}|j_+^{(\al_h,l)}|dxd\tau\\
&
\lesssim\Bigl(\sum_{k_1+l_1\leq 2N}\delta^{l_1-\f12}\bigl(E_-^{(k_1,l_1)}(z_-)\bigr)^{\f12}+\sum_{k_1\leq 2N-1}\delta^{-\f32}\bigl(E_-^{(k_1,0)}(z_-^3)\bigr)^{\f12}\Bigr)\\
&\qquad\qquad\times\Bigl(\sum_{k_1+l_1\leq 2N}\delta^{2(l_1-\f12)}F_+^{(k_1,l_1)}(z_+)+\sum_{k_1+l_1\leq N+2}\delta^{2(l_1-\f12)}F_+^{(k_1,l_1)}(\p_3z_+)\Bigr).
\end{aligned}\eeq

\medskip

{\bf Step 2. The {\it a priori} estimate of $\na_h^k\p_3^l\p_3z_+$ for $k+l\leq N+2$.} Combining \eqref{hv3}, \eqref{hv21}, \eqref{hv25}, \eqref{hv28} together with \eqref{curl to z}, we obtain  that for $k+l\leq N+2$, it holds
\beq\label{hv29}\begin{aligned}
&\delta^{2(l-\f12)}\Bigl(E_+^{(k,l)}(j_+)+F_+^{(k,l)}(j_+)\Bigr)\lesssim \delta^{2(l-\f12)}E_+^{(k,l)}(\curl z_{+,0})+\Bigl(\sum_{k_1+l_1\leq 2N}\delta^{l_1-\f12}\bigl(E_-^{(k_1,l_1)}(z_-)\bigr)^{\f12}\\
&\qquad+\sum_{k_1\leq 2N-1}\delta^{-\f32}\bigl(E_-^{(k_1,0)}(z_-^3)\bigr)^{\f12}+\sum_{k_1+l_1\leq N+2}\delta^{l_1-\f12}\bigl(E_-^{(k_1,l_1)}(\p_3z_-)\bigr)^{\f12}\Bigr)\\
&\qquad\times\Bigl(\sum_{k_1+l_1\leq 2N}\delta^{2(l_1-\f12)}F_+^{(k_1,l_1)}(z_+)
+\sum_{k_1+l_1\leq N+2}\delta^{2(l_1-\f12)}F_+^{(k_1,l_1)}(\p_3z_+)\Bigr).
\end{aligned}\eeq

Thanks to \eqref{hv29} and \eqref{hv1}, by using \eqref{estimate for al h} and \eqref{curl to z}, we obtain the desired inequality \eqref{estimate for al h l v-1} for $z_+$. Similar estimate holds for $z_-$. Then it ends the proof of the proposition.
\end{proof}

\subsection{Proof of main uniform {\it a priori}  estimates and Theorem \ref{global existence in thin domain}}
The existence part of Theorem \ref{global existence in thin domain} follows the standard method of continuity. We shall only prove the uniform energy estimate \eqref{total energy estimates}.

{\bf Step 1. Ansatz and closure of the continuity argument.} To use the method of continuity, we first need the following ansatz. We assume that
\beq\label{ansatz 1}
\|z_\pm^1\|_{L^\infty_t L^\infty_x}\leq 1,
\eeq
and
\beq\label{ansatz 2}\begin{aligned}
&\delta^{2(l-\f12)}E_\pm^{(k,l)}(\zpm)\leq 2C_1\varepsilon^2,\quad \delta^{2(l-\f12)}F_{\pm}^{(k,l)}(\zpm)\leq 2C_1\varepsilon^2,\quad\text{for}\quad k+l\leq N_*,\\
&\delta^{-3}E_\pm^{(k,0)}(\zpm^3)\leq 2C_1\varepsilon^2,\quad \delta^{-3}F_\pm^{(k,0)}(\zpm^3)\leq 2C_1\varepsilon^2\quad\text{for}\quad k\leq N_*-1,\\
&\delta^{2(l-\f12)}E_{\pm}^{(k,l)}(\p_3\zpm)\leq 2C_1\varepsilon^2,\quad \delta^{2(l-\f12)}F_{\pm}^{(k,l)}(\p_3\zpm)\leq 2C_1\varepsilon^2,\quad\text{for}\quad k+l\leq N+2,
\end{aligned}\eeq
where $C_1$ would be determined by the energy estimates.

To close the continuity argument, we need to prove that there exists a small enough $\varepsilon_0$ such that for all $\varepsilon\leq \varepsilon_0$ the constant $2$ in \eqref{ansatz 2} can be improved to  be $1$, i.e.,
\beno\begin{aligned}
&\delta^{2(l-\f12)}E_\pm^{(k,l)}(\zpm)\leq C_1\varepsilon^2,\quad \delta^{2(l-\f12)}F_{\pm}^{(k,l)}(\zpm)\leq C_1\varepsilon^2,\quad\text{for}\quad k+l\leq N_*,\\
&\delta^{-3}E_\pm^{(k,0)}(\zpm^3)\leq C_1\varepsilon^2,\quad \delta^{-3}F_\pm^{(k,0)}(\zpm^3)\leq C_1\varepsilon^2\quad\text{for}\quad k\leq N_*-1,\\
&\delta^{2(l-\f12)}E_{\pm}^{(k,l)}(\p_3\zpm)\leq C_1\varepsilon^2,\quad \delta^{2(l-\f12)}F_{\pm}^{(k,l)}(\p_3\zpm)\leq C_1\varepsilon^2,\quad\text{for}\quad k+l\leq N+2.
\end{aligned}\eeno
For $\varepsilon\leq \varepsilon_0$, we also improve the ansatz \eqref{ansatz 1} to
$
\|z_\pm^1\|_{L^\infty_t L^\infty_x}\leq \f12.
$

Actually, by Sobolev inequality, we have
\beno
\|z_\pm^1\|_{L^\infty_t L^\infty_x}\leq C_2\sum_{k+l\leq 2}\Bigl(\delta^{l-\f12}\bigl(E_\pm^{(k,l)}(z_\pm)\bigr)^{\f12}\leq 6\sqrt{C_1}C_2\varepsilon_0.
\eeno
Then by taking $\varepsilon_0$ small enough, we could prove that $\|z_\pm^1\|_{L^\infty_t L^\infty_x}\leq \f12$.

\medskip

{\bf Step 2. The uniform {\it a priori} estimate under the ansatz \eqref{ansatz 1} and \eqref{ansatz 2}.}
Since $\div\,z_\pm=0$ and $z_+^3|_{x_3=\pm\delta}=0$, $z_-^3|_{x_3=\pm\delta}=0$,  we have
\beno
z_\pm^3(t,x_h,x_3)=\int_{-\delta}^{x_3}(\p_3z_\pm^3)(t,x_h,s)ds=-\int_{-\delta}^{x_3}(\na_h\cdot z_\pm^h)(t,x_h,s)ds.
\eeno
which implies
\beno
|\na_h^kz_\pm^3(t,x)|\leq\int_{-\delta}^\delta|\na_h^{k+1} z_\pm^h(t,x_h,x_3)|dx_3,\quad\text{for any}\quad k\in\Z_{\geq0}.
\eeno
Then by H\"older inequality,  for any $k\in\Z_{\geq0}$, we obtain that
\beq\label{estimate for z3}
E_\pm^{(k,0)}(\zpm^3)\lesssim\delta E_\pm^{(k+1,0)}(\zpm^h),\quad F_\pm^{(k,0)}(\zpm^3)\lesssim\delta F_\pm^{(k+1,0)}(\zpm^h).
\eeq

Thanks to \eqref{estimate for z3}, \eqref{estimate for al h} of Proposition \ref{al h prop}, \eqref{estimate for al h l v} of Proposition \ref{al h l v prop} and \eqref{estimate for al h l v-1} of Proposition \ref{al h l v-1 prop}, we obtain that
\beno\begin{aligned}
&\mathcal{E}(t^*)\eqdefa\sum_{+,-}\Bigl(\sum_{k+l\leq N_*}\delta^{2(l-\f12)}\bigl[E_\pm^{(k,l)}(\zpm)+F_\pm^{(k,l)}(\zpm)\bigr]+\sum_{k\leq N_*-1}\delta^{-3}\bigl[E_\pm^{(k,0)}(\zpm^3)+F_\pm^{(k,0)}(\zpm^3)\bigr]\\
&\qquad+\sum_{k+l\leq N+2}\delta^{2(l-\f12)}\bigl[E_{\pm}^{(k,l)}(\p_3\zpm)+F_{\pm}^{(k,l)}(\p_3\zpm)\bigr]\Bigr)\\
&\leq C_3\mathcal{E}(0)+C_3\sum_{+,-}\Bigl[\Bigl(\sum_{k+l\leq N_*}\delta^{l-\f12}\bigl(E_\mp^{(k,l)}(z_\mp)\bigr)^{\f12}+\sum_{k\leq N_*-1}\delta^{-\f32}\bigl(E_\mp^{(k,0)}(z_\mp^3)\bigr)^{\f12}+\sum_{k+l\leq N+2}\delta^{l-\f12}\bigl(E_\mp^{(k,l)}(\p_3z_\mp)\bigr)^{\f12}\Bigr)\\
&\quad
\times\Bigl(\sum_{k+l\leq N_*}\delta^{2(l-\f12)}F_\pm^{(k,l)}(z_\pm)+\sum_{k\leq N_*-1}\delta^{-3}F_\pm^{(k,0)}(z_\pm^3)
+\sum_{k+l\leq N+2}\delta^{2(l-\f12)}F_\pm^{(k,l)}(\p_3z_\pm)\Bigr)\Bigr].
\end{aligned}\eeno
Using \eqref{ansatz 2}, we could obtain that
\beno
\mathcal{E}(t^*)\leq C_3\mathcal{E}(0)+C_3C_4C_1\varepsilon\mathcal{E}(t^*).
\eeno
Taking $\varepsilon_0\leq\f{1}{2C_3C_4C_1}$, we obtain for any $\varepsilon\leq\varepsilon_0$, $t^*<\infty$
\beno
\mathcal{E}(t^*)\leq 2C_3\mathcal{E}(0)\leq C_1\varepsilon^2,
\eeno
where we choose $C_1\geq 2C_3$.   Then Theorem \ref{global existence in thin domain} is proved.

\section{Proof of Theorem \ref{stability}}
In this section, we give the proof to Theorem \ref{stability}. To do that, we first consider the following linearized system
\begin{equation}\label{Linerized system for 2D}
\begin{aligned}
&\p_t f_+-\p_1f_++\bar{z}_-^h\cdot\na_h f_+=\r_+,\quad\text{in}\quad\Omega_\delta,\\
&\p_t f_-+\p_1f_-+\bar{z}_+^h\cdot\na_h f_-=\r_-.
\end{aligned}
\end{equation}
Note that $\div_h\bar{z}_\pm^h=\na_h\cdot\bar{z}_\pm^h=0$. Thanks to the proof of Proposition \ref{linearized prop}, we obtain the following proposition.
\begin{proposition}\label{linearized prop for 2D}
Assume that $\div_h\,\bar{z}_\pm^h=0$ and it holds
\beq\label{assumption 1a}
\|\bar{z}_\pm^1\|_{L^\infty_tL^\infty_h}\leq 1.
\eeq
Then for any  smooth solutions  $(f_+,f_-)$ to \eqref{Linerized system}, we have
\beq\label{estimate for linearized system 2D}\begin{aligned}
&\sup_{0\leq\tau\leq t}\int_{\Sigma_{\tau,h}}\langle u_\mp\rangle^{2(1+\sigma)}|f_\pm|^2dx_h+\int_0^t\int_{\Sigma_{\tau,h}}\f{\langle u_\mp\rangle^{2(1+\sigma)}}{\langle u_\pm\rangle^{1+\sigma}}|f_\pm|^2dx_hd\tau\\
&\lesssim\int_{\Sigma_{0,h}}\langle u_\mp\rangle^{2(1+\sigma)}|f_\pm|^2dx_h+\int_0^t\int_{\Sigma_{\tau,h}}\langle u_\mp\rangle^{1+2\sigma}|\bar{z}_\mp^1||f_\pm|^2dx_hd\tau+\Bigl|\int_0^t\int_{\Sigma_{\tau,h}}\r_\pm\cdot\langle u_\mp\rangle^{2(1+\sigma)}f_\pm dx_hd\tau\Bigr|\\
&\qquad
+\int_{\R}\f{1}{\langle u_\pm\rangle^{1+\sigma}}\Bigl|\int\int_{W_{t,h}^{[u_+,\infty]}/(W_{t,h}^{[-\infty,u_-]})}\r_\pm\cdot\langle u_\mp\rangle^{2(1+\sigma)}f_\pm dx_hd\tau\Bigr| d{u_\pm}.
\end{aligned}\eeq
In particular, it holds
\beq\label{estimate for linearized system 0 2D}\begin{aligned}
&\sup_{0\leq\tau\leq t}\int_{\Sigma_{\tau,h}}\langle u_\mp\rangle^{2(1+\sigma)}|f_\pm|^2dx_h+\int_0^t\int_{\Sigma_{\tau,h}}\f{\langle u_\mp\rangle^{2(1+\sigma)}}{\langle u_\pm\rangle^{1+\sigma}}|f_\pm|^2dx_hd\tau\\
&\lesssim\int_{\Sigma_{0,h}}\langle u_\mp\rangle^{2(1+\sigma)}|f_\pm|^2dx_h+\int_0^t\int_{\Sigma_{\tau,h}}\langle u_\mp\rangle^{1+2\sigma}|\bar{z}_\mp^1||f_\pm|^2dx_hd\tau+\int_0^t\int_{\Sigma_{\tau,h}}|\r_\pm|\cdot\langle u_\mp\rangle^{2(1+\sigma)}|f_\pm| dx_hd\tau.
\end{aligned}\eeq
\end{proposition}

With Proposition \ref{linearized prop for 2D}, we now present the proof of  Theorem \ref{stability}.

\medskip

{\bf Proof of Theorem \ref{stability}.}  We shall divide the proof into several steps.

{\bf Step 1. Linearized system of $w_\pm^h$.} For any $\al_h=(\al_1,\al_2)\in(\Z_{\geq0})^2$, we set
\beno
w_\pm^{h(\al_h)}\eqdefa\p_h^{\al_h}w_\pm^h=\p_1^{\al_1}\p_2^{\al_2}w_\pm^h.
\eeno
Applying $\p_h^{\al_h}$ to both sides of equations for $w_\pm^h$ in \eqref{equation for remainder}, we have
\beq\label{rem1}\begin{aligned}
&\p_tw_+^{h(\al_h)}-\p_1w_+^{h(\al_h)}+\bar{z}_-^h\cdot\na_hw_+^{h(\al_h)}=\r_+^{h(\al_h)},\\
&\p_tw_-^{h(\al_h)}+\p_1w_-^{h(\al_h)}+\bar{z}_+^h\cdot\na_hw_-^{h(\al_h)}=\r_-^{h(\al_h)},
\end{aligned}\eeq
where
\beno\begin{aligned}
&\r_\pm^{h(\al_h)}=\underbrace{-(I-M_\delta)(\p_h^{\al_h}\na_hp)}_{\r_{\pm,1}^{h(\al_h)}}\underbrace{-(I-M_\delta)[\p_h^{\al_h}(w_\mp^h\cdot\na_h z_\pm^h)]}_{\r_{\pm,2}^{h(\al_h)}}\\
&\qquad\quad\underbrace{-(I-M_\delta)[\p_h^{\al_h}(w_\mp^3\p_3 z_\pm^h)]}_{\r_{\pm,3}^{h(\al_h)}}
\underbrace{-\p_h^{\al_h}(\bar{z}_\mp^h\cdot\na_hw_\pm^h)+\bar{z}_\mp^h\cdot\na_h\p_h^{\al_h}w_\pm^h}_{\r_{\pm,4}^{h(\al_h)}}.
\end{aligned}\eeno

We shall only give the estimates for $w_+^{h(\al_h)}$. Applying Proposition \ref{linearized prop for 2D} to the first equation of \eqref{rem1},  for any $x_3\in(-\delta,\delta)$, we have
\beq\label{rem2}\begin{aligned}
&\sup_{0\leq\tau\leq t}\int_{\Sigma_{\tau,h}}\langle u_-\rangle^{2(1+\sigma)}|w_+^{h(\al_h)}(\cdot,x_3)|^2dx_h+\int_0^t\int_{\Sigma_{\tau,h}}\f{\langle u_-\rangle^{2(1+\sigma)}}{\langle u_+\rangle^{1+\sigma}}|w_+^{h(\al_h)}(\cdot,x_3)|^2dx_hd\tau\\
&\lesssim\int_{\Sigma_{0,h}}\langle u_-\rangle^{2(1+\sigma)}|w_+^{h(\al_h)}(\cdot,x_3)|^2dx_h+\int_0^t\int_{\Sigma_{\tau,h}}\langle u_-\rangle^{1+2\sigma}|\bar{z}_-^1||w_+^{h(\al_h)}(\cdot,x_3)|^2dx_hd\tau\\
&\qquad+\int_0^t\int_{\Sigma_{\tau,h}}|\r_+^{h(\al_h)}(\cdot,x_3)|\cdot\langle u_-\rangle^{2(1+\sigma)}|w_+^{h(\al_h)}(\cdot,x_3)| dx_hd\tau.
\end{aligned}\eeq

\medskip

{\bf Step 2. Estimates of the nonlinear terms in the r.h.s of \eqref{rem2}.}

{\it Step 2.1. Estimate of the second terms on the r.h.s of \eqref{rem2}.} Using H\"older and Sobolev inequalities, we have
\beq\label{rem3}
\int_0^t\int_{\Sigma_{\tau,h}}\langle u_-\rangle^{1+2\sigma}|\bar{z}_-^1||w_+^{h(\al_h)}(\cdot,x_3)|^2dx_hd\tau\lesssim\sum_{k\leq 2}\bigl(E_{-,h}^{(k)}(\bar{z}_-^1)\bigr)^{\f12}\cdot F_{+,h}^{(\al_h)}(w_+^h(\cdot,x_3)).
\eeq

{\it Step 2.2. Estimate of the last terms on the r.h.s of \eqref{rem2}.} By H\"older inequality, we first have
\beq\label{rem4}\begin{aligned}
&\int_0^t\int_{\Sigma_{\tau,h}}|\r_+^{h(\al_h)}(\cdot,x_3)|\cdot\langle u_-\rangle^{2(1+\sigma)}|w_+^{h(\al_h)}(\cdot,x_3)| dx_hd\tau\\
&
\lesssim\|\langle u_+\rangle^{\f12(1+\sigma)}\langle u_-\rangle^{1+\sigma}\r_+^{h(\al_h)}(\cdot,x_3)\|_{L^2_tL^2_h}\bigl(F_{+,h}^{(\al_h)}(w_+^h(\cdot,x_3))\bigr)^{\f12}.
\end{aligned}\eeq
We only need to control $\|\langle u_+\rangle^{\f12(1+\sigma)}\langle u_-\rangle^{1+\sigma}\r_+^{h(\al_h)}(\cdot,x_3)\|_{L^2_tL^2_h}$.

{\it Step 2.2.1. Estimate of $\|\langle u_+\rangle^{\f12(1+\sigma)}\langle u_-\rangle^{1+\sigma}\r_{+,1}^{h(\al_h)}(\cdot,x_3)\|_{L^2_tL^2_h}$.}
For $\r_{+,1}^{h(\al_h)}(\cdot,x_3)$, since
\beq\label{rem15}
|(I-M_\delta)f(x)|=|\f{1}{2\delta}\int_{-\delta}^\delta(f(x_h,x_3)-f(x_h,y_3))dy_3|\leq\delta\|\p_3f(x_h,\cdot)\|_{L^\infty_v},
\eeq
we have
\beq\label{rem6}\begin{aligned}
&\|\langle u_+\rangle^{\f12(1+\sigma)}\langle u_-\rangle^{1+\sigma}\r_{+,1}^{h(\al_h)}(\cdot,x_3)\|_{L^2_tL^2_h}\\
&=\|(I-M_\delta)\bigl(\langle u_+\rangle^{\f12(1+\sigma)}\langle u_-\rangle^{1+\sigma}\p_h^{\al_h}\na_hp(\cdot,x_3)\bigr)\|_{L^2_tL^2_h}\\
&\leq\delta\|\langle u_+\rangle^{\f12(1+\sigma)}\langle u_-\rangle^{1+\sigma}\na_h^{|\al_h|+1}\p_3p\|_{L^2_tL^2_hL^\infty_v}.
\end{aligned}\eeq
Thanks to \eqref{expression of pressure},  for any $\al_h\in(\Z_{\geq0})^2$, we have
\beno
\p_h^{\al_h}\p_3p(\tau,x)=\int_{\Omega_\delta}\p_3G_\delta(x,y)\p_h^{\al_h}(\p_iz_+^j\p_jz_-^i)(\tau,y)dy.
\eeno
Let $\theta(r)$ be the smooth cut-off function defining in Step 2.1 of the proof to Proposition \ref{al h prop}. Using \eqref{Green transformation} and integrating by parts, we have
\beno\begin{aligned}
\p_h^{\al_h}\p_3p(\tau,x)&=\underbrace{\int_{\Omega_\delta}\p_3G_\delta(x,y)\theta(|x_h-y_h|)\p_h^{\al_h}(\p_iz_+^j\p_jz_-^i)(\tau,y)dy}_{A_1(\tau,x)}\\
&\quad+\underbrace{\int_{\Omega_\delta}\p_3G_\delta(x,y)\bigl(1-\theta(|x_h-y_h|)\bigr)\p_h^{\al_h}(\p_iz_+^j\p_jz_-^i)(\tau,y)dy}_{A_2(\tau,x)}.
\end{aligned}\eeno

(i) Estimate of $\|\langle u_+\rangle^{\f12(1+\sigma)}\langle u_-\rangle^{1+\sigma}A_1(\tau,x)\|_{L^2_tL^2_hL^\infty_v}$. Firstly, by virtue of \eqref{derivative of Green function} and \eqref{x-y<2}, we have
\beno\begin{aligned}
&\quad\langle u_+(\tau,x_1)\rangle^{\f12(1+\sigma)}\langle u_-(\tau,x_1)\rangle^{1+\sigma}|A_1(\tau,x)|\\
&\lesssim\f{1}{\delta}\int_{-\delta}^\delta\int_{|x_h-y_h|\leq 2}\f{1}{|x_h-y_h|}\bigl(\langle u_+\rangle^{\f12(1+\sigma)}\langle u_-\rangle^{1+\sigma}|\p_h^{\al_h}(\p_iz_+^j\p_jz_-^i)|\bigr)(\tau,y)dy_hdy_3.
\end{aligned}\eeno
Using Young inequality for the horizontal variables $x_h$ and then Sobolev inequality for the vertical variable $x_3$, we have
\beq\label{rem7}\begin{aligned}
&\quad\|\langle u_+\rangle^{\f12(1+\sigma)}\langle u_-\rangle^{1+\sigma}A_1(\tau,x)\|_{L^2_tL^2_hL^\infty_v}\\
&\lesssim\|\f{1}{|x_h|}\|_{L^1_h(|x_h|\leq2)}\|\langle u_+\rangle^{\f12(1+\sigma)}\langle u_-\rangle^{1+\sigma}\p_h^{\al_h}(\p_iz_+^j\p_jz_-^i)\|_{L^2_tL^2_hL^\infty_v}\\
&\lesssim\sum_{l\leq1}\delta^{l-\f12}\|\langle u_+\rangle^{\f12(1+\sigma)}\langle u_-\rangle^{1+\sigma}\p_h^{\al_h}\p_3^l(\p_iz_+^j\p_jz_-^i)\|_{L^2_tL^2_x}.
\end{aligned}\eeq

Thanks to H\"older inequality and the condition $\div\,z_\pm=0$, we have
\beno\begin{aligned}
&\quad\|\langle u_+\rangle^{\f12(1+\sigma)}\langle u_-\rangle^{1+\sigma}\p_h^{\al_h}(\p_iz_+^j\p_jz_-^i)\|_{L^2_tL^2_x}\\
&\lesssim\sum_{\beta_h\leq\al_h}\|\langle u_+\rangle^{1+\sigma}\p_h^{\beta_h}\p_jz_-^i
\cdot\f{\langle u_-\rangle^{1+\sigma}}{\langle u_+\rangle^{\f12(1+\sigma)}}\p_h^{\al_h-\beta_h}\p_iz_+^j\|_{L^2_tL^2_x}\\
&\lesssim\sum_{\beta_h\leq\al_h}\Bigl(\|\langle u_+\rangle^{1+\sigma}\p_h^{\beta_h}\na_hz_-^h\|_{L^\infty_tL^\infty_x}
\|\f{\langle u_-\rangle^{1+\sigma}}{\langle u_+\rangle^{\f12(1+\sigma)}}\p_h^{\al_h-\beta_h}\na_hz_+^h\|_{L^2_tL^2_x}\\
&\quad+\|\langle u_+\rangle^{1+\sigma}\p_h^{\beta_h}\p_3z_-^h\|_{L^\infty_tL^\infty_x}
\|\f{\langle u_-\rangle^{1+\sigma}}{\langle u_+\rangle^{\f12(1+\sigma)}}\p_h^{\al_h-\beta_h}\na_hz_+^3\|_{L^2_tL^2_x}\\
&\quad+\|\langle u_+\rangle^{1+\sigma}\p_h^{\beta_h}\na_hz_-^3\|_{L^\infty_tL^\infty_x}
\|\f{\langle u_-\rangle^{1+\sigma}}{\langle u_+\rangle^{\f12(1+\sigma)}}\p_h^{\al_h-\beta_h}\p_3z_+^h\|_{L^2_tL^2_x}\Bigr).
\end{aligned}\eeno
By Sobolev inequality \eqref{s2}, we have
\beno\begin{aligned}
&\quad\delta^{-\f12}\|\langle u_+\rangle^{\f12(1+\sigma)}\langle u_-\rangle^{1+\sigma}\p_h^{\al_h}(\p_iz_+^j\p_jz_-^i)\|_{L^2_tL^2_x}\\
&\lesssim\sum_{k+l\leq |\al_h|+3}\delta^{l-\f12}\bigl(E_-^{(k,l)}(z_-)\bigr)^{\f12}\cdot\sum_{k\leq|\al_h|+1}\delta^{-\f12}\bigl(F_+^{(k,0)}(z_+)\bigr)^{\f12}\\
&\quad+\sum_{k+l\leq |\al_h|+2}\delta^{l+\f12}\bigl(E_-^{(k,l+1)}(z_-)\bigr)^{\f12}\cdot\sum_{k\leq|\al_h|+1}\delta^{-\f32}\bigl(F_+^{(k,0)}(z_+^3)\bigr)^{\f12}\\
&\quad+\sum_{k+l\leq |\al_h|+3}\delta^{l-\f32}\bigl(E_-^{(k,l)}(z_-^3)\bigr)^{\f12}\cdot\sum_{k\leq|\al_h|}\delta^{\f12}\bigl(F_+^{(k,1)}(z_+)\bigr)^{\f12}.
\end{aligned}\eeno
Then for $|\al_h|\leq 2N-4$, using $\div\,z_-=0$, we have
\beq\label{rem8}\begin{aligned}
&\quad\delta^{-\f12}\|\langle u_+\rangle^{\f12(1+\sigma)}\langle u_-\rangle^{1+\sigma}\p_h^{\al_h}(\p_iz_+^j\p_jz_-^i)\|_{L^2_tL^2_x}\\
&\lesssim\Bigl(\sum_{k+l\leq2N}\delta^{l-\f12}\bigl(E_-^{(k,l)}(z_-)\bigr)^{\f12}+\sum_{k\leq2N-1}\delta^{-\f32}\bigl(E_-^{(k,0)}(z_-^3)\bigr)^{\f12}\Bigr)\\
&\quad
\times\Bigl(\sum_{k+l\leq2N}\delta^{l-\f12}\bigl(F_+^{(k,l)}(z_+)\bigr)^{\f12}+\sum_{k\leq2N-1}\delta^{-\f32}\bigl(F_+^{(k,0)}(z_+^3)\bigr)^{\f12}\Bigr).
\end{aligned}\eeq
Similarly, for $|\al_h|\leq 2N-4$, we have the same estimate for the term
\beno
\delta^{\f12}\|\langle u_+\rangle^{\f12(1+\sigma)}\langle u_-\rangle^{1+\sigma}\p_h^{\al_h}\p_3(\p_iz_+^j\p_jz_-^i)\|_{L^2_tL^2_x}.
\eeno
With the help of \eqref{rem7} and \eqref{rem8}, for $|\al_h|\leq 2N-4$,  we have
\beq\label{rem9}\begin{aligned}
&\quad\|\langle u_+\rangle^{\f12(1+\sigma)}\langle u_-\rangle^{1+\sigma}A_1\|_{L^\infty_vL^2_tL^2_h}\\
&\lesssim\Bigl(\sum_{k+l\leq2N}\delta^{l-\f12}\bigl(E_-^{(k,l)}(z_-)\bigr)^{\f12}+\sum_{k\leq2N-1}\delta^{-\f32}\bigl(E_-^{(k,0)}(z_-^3)\bigr)^{\f12}\Bigr)\\
&\quad
\times\Bigl(\sum_{k+l\leq2N}\delta^{l-\f12}\bigl(F_+^{(k,l)}(z_+)\bigr)^{\f12}+\sum_{k\leq2N-1}\delta^{-\f32}\bigl(F_+^{(k,0)}(z_+^3)\bigr)^{\f12}\Bigr).
\end{aligned}\eeq

(ii) Estimate of $\|\langle u_+\rangle^{\f12(1+\sigma)}\langle u_-\rangle^{1+\sigma}A_2(\tau,x)\|_{L^2_tL^2_hL^\infty_v}$. For $A_2$, using $\div\,z_\pm=0,\, z^3_+|_{x_3=\pm\delta}=0, \, z^3_-|_{x_3=\pm\delta}=0$ and \eqref{Green transformation}, and integrating by parts, we have
\beno\begin{aligned}
&|A_2(\tau,x)|\lesssim\int_{\Omega_\delta}|\p_i\p_j\p_3G_\delta(x,y)|(1-\theta(|x_h-y_h|))|\p_h^{\al_h}(z_-^iz_+^j)(\tau,y)|dy\\
&\quad
+\int_{\Omega_\delta}\bigl(|\p_3G_\delta(x,y)||\theta''(|x_h-y_h|)+|\na\p_3G_\delta(x,y)||\theta'(|x_h-y_h|)\bigr)|\p_h^{\al_h}(z_-^iz_+^j)(\tau,y)|dy.
\end{aligned}\eeno
Due to \eqref{derivative of Green function}, we have
\beno\begin{aligned}
&|A_2(\tau,x)|\lesssim\underbrace{\f{1}{\delta}\int_{-\delta}^\delta\int_{|x_h-y_h|\geq1}\f{1}{|x_h-y_h|^3}|\p_h^{\al_h}(z_-^iz_+^j)(\tau,y)|dy_hdy_3}_{A_{21}(\tau,x)}\\
&\quad
+\underbrace{\f{1}{\delta}\int_{-\delta}^\delta\int_{1\leq|x_h-y_h|\leq2}\f{1}{|x_h-y_h|}|\p_h^{\al_h}(z_-^iz_+^j)(\tau,y)|dy_hdy_3}_{A_{22}(\tau,x)}.
\end{aligned}\eeno
Similarly to \eqref{rem9}, for $|\al_h|\leq 2N-4$, we have
\beq\label{rem10}
\|\langle u_+\rangle^{\f12(1+\sigma)}\langle u_-\rangle^{1+\sigma}A_{22}\|_{L^2_tL^2_hL^\infty_v}
\lesssim\sum_{k+l\leq2N}\delta^{l-\f12}\bigl(E_-^{(k,l)}(z_-)\bigr)^{\f12}\cdot\sum_{k+l\leq2N}\delta^{l-\f12}\bigl(F_+^{(k,l)}(z_+)\bigr)^{\f12}.
\eeq
Using \eqref{x-y>1}, we have
\beno\begin{aligned}
&\quad\langle u_+(\tau,x_1)\rangle^{\f12(1+\sigma)}\langle u_-(\tau,x_1)\rangle^{1+\sigma}A_{21}(\tau,x)\\
&
\lesssim\f{1}{\delta}\int_{-\delta}^\delta\int_{|x_h-y_h|\geq1}\f{1}{|x_h-y_h|^{\f{3}{2}(1-\sigma)}}\bigl(\langle u_+\rangle^{\f12(1+\sigma)}\langle u_-\rangle^{1+\sigma}|\p_h^{\al_h}(z_-^iz_+^j)|\bigr)(\tau,y)dy_hdy_3
\end{aligned}\eeno
Thanks to Young inequality and the fact $\sigma\in(0,\f13)$, we have
\beno\begin{aligned}
&\|\langle u_+\rangle^{\f12(1+\sigma)}\langle u_-\rangle^{1+\sigma}A_{21}\|_{L^2_tL^2_hL^\infty_v}
\lesssim\delta^{-1}\|\f{1}{|x_h|^{\f{3}{2}(1-\sigma)}}\|_{L^2_h(|x_h|\geq1)}\|\langle u_+\rangle^{\f12(1+\sigma)}\langle u_-\rangle^{1+\sigma}\p_h^{\al_h}(z_-^iz_+^j)\|_{L^2_tL^1_x}\\
&\lesssim\sum_{\beta_h\leq\al_h}\delta^{-\f12}\|\langle u_+\rangle^{1+\sigma}\p_h^{\al_h-\beta_h}z_-\|_{L^\infty_tL^2_x}
\cdot\delta^{-\f12}\|\f{\langle u_-\rangle^{1+\sigma}}{\langle u_+\rangle^{\f12(1+\sigma)}}\p_h^{\beta_h}z_+\|_{L^\infty_tL^2_x}.
\end{aligned}\eeno
Then for $|\al_h|\leq 2N-4$, we obtain that
\beq\label{rem11}
\|\langle u_+\rangle^{\f12(1+\sigma)}\langle u_-\rangle^{1+\sigma}A_{21}\|_{L^2_tL^2_hL^\infty_v}
\lesssim\sum_{k+l\leq2N}\delta^{l-\f12}\bigl(E_-^{(k,l)}(z_-)\bigr)^{\f12}\cdot\sum_{k+l\leq2N}\delta^{l-\f12}\bigl(F_+^{(k,l)}(z_+)\bigr)^{\f12}.
\eeq
Thanks to \eqref{rem10} and \eqref{rem11}, for  $|\al_h|\leq 2N-4$,  we have
\beq\label{rem12}
\|\langle u_+\rangle^{\f12(1+\sigma)}\langle u_-\rangle^{1+\sigma}A_2\|_{L^2_tL^2_hL^\infty_v}
\lesssim\sum_{k+l\leq2N}\delta^{l-\f12}\bigl(E_-^{(k,l)}(z_-)\bigr)^{\f12}\cdot\sum_{k+l\leq2N}\delta^{l-\f12}\bigl(F_+^{(k,l)}(z_+)\bigr)^{\f12}.
\eeq

\smallskip

Due to \eqref{rem9} and \eqref{rem12},  for $|\al_h|\leq 2N-4$, we obtain that
\beq\label{rem25}\begin{aligned}
&\quad\|\langle u_+\rangle^{\f12(1+\sigma)}\langle u_-\rangle^{1+\sigma}\p_h^{\al_h}\p_3p\|_{L^2_tL^2_hL^\infty_v}\\
&\lesssim\Bigl(\sum_{k+l\leq2N}\delta^{l-\f12}\bigl(E_-^{(k,l)}(z_-)\bigr)^{\f12}+\sum_{k\leq2N-1}\delta^{-\f32}\bigl(E_-^{(k,0)}(z_-^3)\bigr)^{\f12}\Bigr)\\
&\quad
\times\Bigl(\sum_{k+l\leq2N}\delta^{l-\f12}\bigl(F_+^{(k,l)}(z_+)\bigr)^{\f12}+\sum_{k\leq2N-1}\delta^{-\f32}\bigl(F_+^{(k,0)}(z_+^3)\bigr)^{\f12}\Bigr).
\end{aligned}\eeq
 Then by virtue of \eqref{rem25} and \eqref{rem6},  for $|\al_h|\leq 2N-5$, we have
\beq\label{rem13}\begin{aligned}
&\quad\|\langle u_+\rangle^{\f12(1+\sigma)}\langle u_-\rangle^{1+\sigma}\r_{+,1}^{h(\al_h)}(\cdot,x_3)\|_{L^2_tL^2_h}\\
&\lesssim\delta\Bigl(\sum_{k+l\leq2N}\delta^{l-\f12}\bigl(E_-^{(k,l)}(z_-)\bigr)^{\f12}+\sum_{k\leq2N-1}\delta^{-\f32}\bigl(E_-^{(k,0)}(z_-^3)\bigr)^{\f12}\Bigr)\\
&\quad
\times\Bigl(\sum_{k+l\leq2N}\delta^{l-\f12}\bigl(F_+^{(k,l)}(z_+)\bigr)^{\f12}+\sum_{k\leq2N-1}\delta^{-\f32}\bigl(F_+^{(k,0)}(z_+^3)\bigr)^{\f12}\Bigr).
\end{aligned}\eeq

{\it Step 2.2.2. Estimate of $\|\langle u_+\rangle^{\f12(1+\sigma)}\langle u_-\rangle^{1+\sigma}\r_{+,2}^{h(\al_h)}(\cdot,x_3)\|_{L^2_tL^2_h}$.}
 By the definition of $\r_{+,2}^{h(\al_h)}$, we have
\beno
\|\langle u_+\rangle^{\f12(1+\sigma)}\langle u_-\rangle^{1+\sigma}\r_{+,2}^{h(\al_h)}(\cdot,x_3)\|_{L^2_tL^2_h}
\lesssim\|\langle u_+\rangle^{\f12(1+\sigma)}\langle u_-\rangle^{1+\sigma}\p_h^{\al_h}(w_-^h\cdot\na_hz_+^h)\|_{L^\infty_vL^2_tL^2_h}.
\eeno
Similarly as \eqref{rem8}, for $|\al_h|\leq N\, (\,\leq 2N-4)$ and $N\in\Z_{\geq5}$,  we have
\beq\label{rem14}
\|\langle u_+\rangle^{\f12(1+\sigma)}\langle u_-\rangle^{1+\sigma}\r_{+,2}^{h(\al_h)}(\cdot,x_3)\|_{L^2_tL^2_h}
\lesssim\sum_{k\leq N}\sup_{x_3\in(-\delta,\delta)}\bigl(E_{-,h}^{(k)}(w_-^h(\cdot,x_3))\bigr)^{\f12}\cdot\sum_{k+l\leq 2N}\delta^{l-\f12}\bigl(F_+^{(k,l)}(z_+)\bigr)^{\f12}.
\eeq

\smallskip

{\it Step 2.2.3. Estimate of $\|\langle u_+\rangle^{\f12(1+\sigma)}\langle u_-\rangle^{1+\sigma}\r_{+,3}^{h(\al_h)}(\cdot,x_3)\|_{L^2_tL^2_h}$.} By the definition of $\r_{+,3}^{h(\al_h)}$ and $w_-^3=z_-^3$ as well as the facts \eqref{rem15} and $\div\,z_-=0$, we have
\beno\begin{aligned}
&\|\langle u_+\rangle^{\f12(1+\sigma)}\langle u_-\rangle^{1+\sigma}\r_{+,3}^{h(\al_h)}(\cdot,x_3)\|_{L^2_tL^2_h}\lesssim\delta\|\langle u_+\rangle^{\f12(1+\sigma)}\langle u_-\rangle^{1+\sigma}\p_h^{\al_h}\p_3(z_-^3\p_3z_+^h)\|_{L^2_tL^2_hL^\infty_v}\\
&\lesssim\delta\|\langle u_+\rangle^{\f12(1+\sigma)}\langle u_-\rangle^{1+\sigma}\p_h^{\al_h}(\na_h\cdot z_-^h\p_3z_+^h)\|_{L^2_tL^2_hL^\infty_v}
+\delta\|\langle u_+\rangle^{\f12(1+\sigma)}\langle u_-\rangle^{1+\sigma}\p_h^{\al_h}(z_-^3\p_3^2z_+^h)\|_{L^2_tL^2_hL^\infty_v}.
\end{aligned}\eeno
Following similar argument as that for  \eqref{rem8}, we have
\beno\begin{aligned}
&\|\langle u_+\rangle^{\f12(1+\sigma)}\langle u_-\rangle^{1+\sigma}\p_h^{\al_h}(\na_h\cdot z_-^h\p_3z_+^h)\|_{L^2_tL^2_hL^\infty_v}\\
&\lesssim\sum_{k+l\leq|\al_h|+3}\delta^{l-\f12}\bigl(E_-^{(k,l)}(z_-^h)\bigr)^{\f12}\cdot\sum_{k+l\leq|\al_h|+1}\delta^{l-\f12}\bigl(F_+^{(k,l)}(\p_3z_+)\bigr)^{\f12},
\end{aligned}\eeno
and
\beno\begin{aligned}
&\|\langle u_+\rangle^{\f12(1+\sigma)}\langle u_-\rangle^{1+\sigma}\p_h^{\al_h}(z_-^3\p_3^2z_+^h)\|_{L^2_tL^2_hL^\infty_v}\\
&\lesssim\sum_{k+l\leq|\al_h|+2}\delta^{l-\f32}\bigl(E_-^{(k,l)}(z_-^3)\bigr)^{\f12}
\cdot\sum_{k+l\leq|\al_h|+1}\delta^{l+\f12}\bigl(F_+^{(k,l)}(\p_3^2z_+^h)\bigr)^{\f12}\\
&\stackrel{div\,z_-=0}{\lesssim}\Bigl(\sum_{k+l\leq|\al_h|+2}\delta^{l-\f12}\bigl(E_-^{(k,l)}(z_-^h)\bigr)^{\f12}
+\sum_{k\leq|\al_h|+2}\delta^{-\f32}\bigl(E_-^{(k,0)}(z_-^3)\bigr)^{\f12}\Bigr)
\cdot\sum_{k+l\leq|\al_h|+2}\delta^{l-\f12}\bigl(F_+^{(k,l)}(\p_3z_+^h)\bigr)^{\f12}.
\end{aligned}\eeno

Therefore for $|\al_h|\leq N,\,N\in\Z_{\geq5}$, we obtain that
\beq\label{rem16}\begin{aligned}
&\quad\|\langle u_+\rangle^{\f12(1+\sigma)}\langle u_-\rangle^{1+\sigma}\r_{+,3}^{h(\al_h)}(\cdot,x_3)\|_{L^2_tL^2_h}\\
&\lesssim\delta\Bigl(\sum_{k+l\leq2N}\delta^{l-\f12}\bigl(E_-^{(k,l)}(z_-)\bigr)^{\f12}+\sum_{k\leq2N-1}\delta^{-\f32}\bigl(E_-^{(k,0)}(z_-^3)\bigr)^{\f12}\Bigr)\\
&\quad
\times\Bigl(\sum_{k+l\leq2N}\delta^{l-\f12}\bigl(F_+^{(k,l)}(z_+)\bigr)^{\f12}+\sum_{k+l\leq N+2}\delta^{l-\f12}\bigl(F_+^{(k,l)}(\p_3z_+)\bigr)^{\f12}\Bigr).
\end{aligned}\eeq

\smallskip

{\it Step 2.2.4. Estimate of $\|\langle u_+\rangle^{\f12(1+\sigma)}\langle u_-\rangle^{1+\sigma}\r_{+,4}^{h(\al_h)}(\cdot,x_3)\|_{L^2_tL^2_h}$.}
   By the definition of $\r_{+,4}^{h(\al_h)}$, we have
\beno\begin{aligned}
&\quad\|\langle u_+\rangle^{\f12(1+\sigma)}\langle u_-\rangle^{1+\sigma}\r_{+,4}^{h(\al_h)}(\cdot,x_3)\|_{L^2_tL^2_h}\\
&\lesssim\sum_{\beta_h<\al_h}\|\langle u_+\rangle^{1+\sigma}\p_h^{\al_h-\beta_h}\bar{z}_-^h\|_{L^\infty_tL^\infty_h}\cdot\|\f{\langle u_-\rangle^{1+\sigma}}{\langle u_+\rangle^{\f12(1+\sigma)}}\na_h\p_h^{\beta_h}w_+^h(\cdot,x_3)\|_{L^2_tL^2_h}\\
&\lesssim\sum_{k\leq|\al_h|+1}\bigl(E_{-,h}^{(k)}(\bar{z}_-^h)\bigr)^{\f12}\cdot\sum_{k\leq|\al_h|}\sup_{x_3\in(-\delta,\delta)}\bigl(F_{+,h}^{(k)}(w_+^h(\cdot,x_3))\bigr)^{\f12}.
\end{aligned}\eeno

Since $\bar{z}_-^h(\cdot)=\f{1}{2\delta}\int_{-\delta}^\delta z_-^h(\cdot,x_3)dx_3$,  for any $k\in\Z_{\geq0}$, we have
\beq\label{rem19}
\bigl(E_{-,h}^{(k)}(\bar{z}_-^h)\bigr)^{\f12}\lesssim\delta^{-\f12}\bigl(E_-^{(k,0)}(z_-^h)\bigr)^{\f12}.
\eeq
Then for $|\al_h|\leq N,\,N\in\Z_{\geq5}$, we obtain that
\beq\label{rem17}\begin{aligned}
&\|\langle u_+\rangle^{\f12(1+\sigma)}\langle u_-\rangle^{1+\sigma}\r_{+,4}^{h(\al_h)}(\cdot,x_3)\|_{L^2_tL^2_h}\\
&\lesssim\sum_{k+l\leq 2N}\delta^{l-\f12}\bigl(E_-^{(k,l)}(z_-^h)\bigr)^{\f12}\cdot\sum_{k\leq N}\sup_{x_3\in(-\delta,\delta)}\bigl(F_{+,h}^{(k)}(w_+^h(\cdot,x_3))\bigr)^{\f12}.
\end{aligned}\eeq

Thanks to \eqref{rem13}, \eqref{rem14}, \eqref{rem16} and \eqref{rem17}, we obtain the estimate for $\|\langle u_+\rangle^{\f12(1+\sigma)}\langle u_-\rangle^{1+\sigma}\r_{+}^{h(\al_h)}(\cdot,x_3)\|_{L^2_tL^2_h}$. Then using \eqref{rem4},  for  $|\al_h|\leq N,\,N\in\Z_{\geq5}$, we have
\beq\label{rem18}\begin{aligned}
&\int_0^t\int_{\Sigma_{\tau,h}}|\r_+^{h(\al_h)}(\cdot,x_3)|\cdot\langle u_-\rangle^{2(1+\sigma)}|w_+^{h(\al_h)}(\cdot,x_3)| dx_hd\tau\\
&\lesssim\delta\Bigl(\sum_{k+l\leq2N}\delta^{l-\f12}\bigl(E_-^{(k,l)}(z_-)\bigr)^{\f12}+\sum_{k\leq2N-1}\delta^{-\f32}\bigl(E_-^{(k,0)}(z_-^3)\bigr)^{\f12}\Bigr)
\cdot\Bigl(\sum_{k+l\leq2N}\delta^{l-\f12}\bigl(F_+^{(k,l)}(z_+)\bigr)^{\f12}\\
&\quad
+\sum_{k\leq2N-1}\delta^{-\f32}\bigl(F_+^{(k,0)}(z_+^3)\bigr)^{\f12}+\sum_{k+l\leq N+2}\delta^{l-\f12}\bigl(F_+^{(k,l)}(\p_3z_+)\bigr)^{\f12}\Bigr)\cdot\bigl(F_{+,h}^{(\al_h)}(w_+^h(\cdot,x_3))\bigr)^{\f12}\\
&\quad+\sum_{k\leq N}\sup_{x_3\in(-\delta,\delta)}\bigl(E_{-,h}^{(k)}(w_-^h(\cdot,x_3))\bigr)^{\f12}\cdot\sum_{k+l\leq2N}\delta^{l-\f12}\bigl(F_+^{(k,l)}(z_+)\bigr)^{\f12}\cdot\bigl(F_{+,h}^{(\al_h)}(w_+^h(\cdot,x_3))\bigr)^{\f12}\\
&\quad
+\sum_{k+l\leq 2N}\delta^{l-\f12}\bigl(E_-^{(k,l)}(z_-^h)\bigr)^{\f12}\cdot\sum_{k\leq N}\sup_{x_3\in(-\delta,\delta)}\bigl(F_{+,h}^{(k)}(w_+^h(\cdot,x_3))\bigr)^{\f12}\cdot\bigl(F_{+,h}^{(\al_h)}(w_+^h(\cdot,x_3))\bigr)^{\f12}.
\end{aligned}\eeq

{\bf Step 3. The {\it a priori} estimate for $w_\pm^h$.}  Thanks to \eqref{total energy estimates}, \eqref{rem2}, \eqref{rem3}, \eqref{rem19} and \eqref{rem18},  for  $|\al_h|\leq N,\,N\in\Z_{\geq5}$, we obtain that
\beno\begin{aligned}
&\sum_{k\leq N}\sup_{x_3\in(-\delta,\delta)}\Bigl(E_{+,h}^{(k)}(w_+^{h}(\cdot,x_3))+F_{+,h}^{(k)}(w_+^{h}(\cdot,x_3))\Bigr)\\
&\lesssim\sum_{k\leq N}\sup_{x_3\in(-\delta,\delta)}E_{+,h}^{(k)}(w_{+,0}^{h}(\cdot,x_3))+\delta\mathcal{E}(0)\sum_{k\leq N}\sup_{x_3\in(-\delta,\delta)}\bigl(F_{+,h}^{(k)}(w_+^h(\cdot,x_3))\bigr)^{\f12}\\
&\quad+\bigl(\mathcal{E}(0)\bigr)^{\f12}\sum_{k\leq N}\sup_{x_3\in(-\delta,\delta)}\Bigl(E_{-,h}^{(k)}(w_-^{h}(\cdot,x_3))+F_{+,h}^{(k)}(w_+^{h}(\cdot,x_3))\Bigr).
\end{aligned}\eeno
Similar estimate holds for $w_-^h$. Then using H\"{o}lder inequality, we have
\beno\begin{aligned}
&\sum_{+,-}\sum_{k\leq N}\sup_{x_3\in(-\delta,\delta)}\Bigl(E_{\pm,h}^{(k)}(w_\pm^{h}(\cdot,x_3))+F_{\pm,h}^{(k)}(w_\pm^{h}(\cdot,x_3))\Bigr)\\
&\leq C_1\sum_{+,-}\sum_{k\leq N}\sup_{x_3\in(-\delta,\delta)}E_{\pm,h}^{(k)}(w_{\pm,0}^{h}(\cdot,x_3))+C_1\delta^2\bigl(\mathcal{E}(0)\bigr)^2\\
&\quad+C_1\bigl(\mathcal{E}(0)\bigr)^{\f12}\sum_{+,-}\sum_{k\leq N}\sup_{x_3\in(-\delta,\delta)}\Bigl(E_{\pm,h}^{(k)}(w_\pm^{h}(\cdot,x_3))+F_{\pm,h}^{(k)}(w_\pm^{h}(\cdot,x_3))\Bigr).
\end{aligned}\eeno
Taking $\varepsilon_1\leq\varepsilon_0$ small enough, for $\mathcal{E}(0)\leq\varepsilon_1^2$, we have
\beq\label{rem20}\begin{aligned}
&\sum_{+,-}\sum_{k\leq N}\sup_{x_3\in(-\delta,\delta)}\Bigl(E_{\pm,h}^{(k)}(w_\pm^{h}(\cdot,x_3))+F_{\pm,h}^{(k)}(w_\pm^{h}(\cdot,x_3))\Bigr)\\
&\leq 2C_1\sum_{+,-}\sum_{k\leq N}\sup_{x_3\in(-\delta,\delta)}E_{\pm,h}^{(k)}(w_{\pm,0}^{h}(\cdot,x_3))+2C_1\delta^2\varepsilon_1^4.
\end{aligned}\eeq

Since $\div\, w_\pm=0$, $w_+^3|_{x_3=\pm\delta}=0$ and $w_-^3|_{x_3=\pm\delta}=0$, it holds
\beno
w_\pm^3(t,x)=\int_{-\delta}^{x_3}\p_3w_\pm^3(t,x_h,y_3)dy_3=-\int_{-\delta}^{x_3}\na_h\cdot w_\pm^h(t,x_h,y_3)dy_3,
\eeno
which togehter with \eqref{rem20} implies that
\beq\label{rem21}\begin{aligned}
&\delta^{-2}\sum_{+,-}\sum_{k\leq N-1}\sup_{x_3\in(-\delta,\delta)}\Bigl(E_{\pm,h}^{(k)}(w_\pm^3(\cdot,x_3))+F_{\pm,h}^{(k)}(w_\pm^3(\cdot,x_3))\Bigr)\\
&\leq C\sum_{+,-}\sum_{k\leq N}\sup_{x_3\in(-\delta,\delta)}E_{\pm,h}^{(k)}(w_{\pm,0}^{h}(\cdot,x_3))+C\delta^2\varepsilon_1^4.
\end{aligned}\eeq

Thanks to \eqref{rem20} and \eqref{rem21}, we arrive at \eqref{stability energy estimate}. Then it ends the proof of Theorem \ref{stability}.

\section{Proof of Theorem \ref{limitation}}
In this section, we shall use Theorem \ref{global existence in thin domain} and Theorem \ref{stability} to prove Theorem \ref{limitation}.

{\bf Proof of Theorem \ref{limitation}.}
Since $z_{\pm(\delta)}^h=\bar{z}_{\pm(\delta)}^h+w_{\pm(\delta)}^h$ and $z_{\pm(\delta)}^3=w_{\pm(\delta)}^3$, the investigation of the asymptotics  from 3D MHD to 2D MHD  can be reduced to prove \beq\label{lim1}\begin{aligned}
&\lim_{\delta\rightarrow0}w_{\pm(\delta)}^h(t,x_h,x_3)=0,\quad\text{in}\quad H^{N}(\R^2),\quad\text{uniform for } t>0,\,\, x_3\in(-1,1),\\
&\lim_{\delta\rightarrow0}w_{\pm(\delta)}^3(t,x_h,x_3)=0,\quad\text{in}\quad H^{N-1}(\R^2),\quad\text{uniform for } t>0,\,\, x_3\in(-1,1),
\end{aligned}\eeq
and
\beq\label{lim2}
\lim_{\delta\rightarrow0}\bar{z}_{\pm(\delta)}^h(t,x_h)=z_{\pm(0)}^h(t,x_h),\quad\text{in}\quad H^{N}(\R^2)\quad\text{uniform for } t>0.
\eeq
We split the proof into several steps.

{\bf Step 1. Uniform estimate of $w_{\pm(\delta)}$ and the proof of \eqref{lim1}.} Thanks to \eqref{total energy estimates} and \eqref{transform for energy functional},  for any $\delta\in(0,1]$, we obtain that
\beq\label{app4}\begin{aligned}
&\mathcal{E}_\delta(t)\leq C \mathcal{E}_\delta(0)
\end{aligned}\eeq
 Since
\beno
\bar{z}_{\pm(\delta),0}^h(x_h)=\f12\int_{-1}^1z_{\pm(\delta),0}^h(x_h,x_3)dx_3,\quad w_{\pm(\delta),0}^h=z_{\pm(\delta),0}^h-\bar{z}_{\pm(\delta),0}^h,
\eeno
we obtain that
\beq\label{lim3}\begin{aligned}
&E_{\pm,h}^{(k)}(\bar{z}_{\pm(\delta),0}^h-z_{\pm(0),0}^h)
\lesssim\sup_{x_3\in(-1,1)}E_{\pm,h}^{(k)}(z_{\pm(\delta),0}^h(\cdot,x_3)-z_{\pm(0),0}^h(\cdot)),\\
&\sup_{x_3\in(-1,1)}E_{\pm,h}^{(k)}(w_{\pm(\delta),0}^h(\cdot,x_3))
\lesssim\sup_{x_3\in(-1,1)}E_{\pm,h}^{(k)}(z_{\pm(\delta),0}^h(\cdot,x_3)-z_{\pm(0),0}^h(\cdot)).
\end{aligned}\eeq
Thanks to \eqref{stability energy estimate}, \eqref{scaling} and \eqref{lim3},  for $\delta\in(0,1]$, we obtain that
\beq\label{lim4}\begin{aligned}
&\sum_{+,-}\sup_{x_3\in(-1,1)}\Bigl(\sum_{k\leq N}\bigl(E_{\pm,h}^{(k)}(w_{\pm(\delta)}^h(\cdot,x_3))+F_{\pm,h}^{(k)}(w_{\pm(\delta)}^h(\cdot,x_3))\bigr)\\
&\quad+\sum_{k\leq N-1}\bigl(E_{\pm,h}^{(k)}(w_{\pm(\delta)}^3(\cdot,x_3))+F_{\pm,h}^{(k)}(w_{\pm(\delta)}^3(\cdot,x_3))\bigr)\Bigr)\\
&\lesssim\sum_{+,-}\sum_{k\leq N}\sup_{x_3\in(-1,1)}E_{\pm,h}^{(k)}(w_{\pm(\delta),0}^h(\cdot,x_3))+\delta^2\varepsilon_1^4\\
&\lesssim\sum_{+,-}\sum_{k\leq N}\sup_{x_3\in(-1,1)}E_{\pm,h}^{(k)}(z_{\pm(\delta),0}^h(\cdot,x_3)-z_{\pm(0),0}^h(\cdot))+\delta^2\varepsilon_1^4.
\end{aligned}\eeq

Since
\beno
\sum_{k\leq N}\sup_{x_3\in(-1,1)}E_{\pm,h}^{(k)}(z_{\pm(\delta),0}^h(\cdot,x_3)-z_{\pm(0),0}^h(\cdot))
\stackrel{\text{Sobolev}}{\lesssim}\sum_{k\leq N+1}E_{\pm}^{(k)}(z_{\pm(\delta),0}^h-z_{\pm(0),0}^h),
\eeno
using \eqref{lim ansatz} and \eqref{lim4}, we arrive at \eqref{lim1}.

{\bf Step 2. Uniform estimate of $\bar{z}_{\pm(\delta)}^h$.}
Since
\beno
\bar{z}_{\pm(\delta)}^h(t,x_h)=\f12\int_{-1}^1z_{\pm(\delta)}^h(t,x_h,x_3)dx_3,
\eeno
we deduce from \eqref{initial} and \eqref{app4} that
\beq\label{app5}
\sum_{k\leq N_*}\bigl(E_{\pm,h}^{(k)}(\bar{z}_{\pm(\delta)}^h)+F_{\pm,h}^{(k)}(\bar{z}_{\pm(\delta)}^h)\bigr)
\lesssim \sum_{k\leq N_*}\bigl(E_\pm^{(k)}(z_{\pm(\delta)}^h)+F_\pm^{(k)}(z_{\pm(\delta)}^h)\bigr)\lesssim\varepsilon_1^2.
\eeq

\medskip

{\bf Step 3. Convergence of the sequence $\{\bar{z}^h_{\pm(\delta)}\}_{0<\delta\leq1}$.}

{\it Step 3.1. Estimate of the difference $\bar{z}^h_{\pm(\delta)}-\bar{z}^h_{\pm(\delta')}$.}
By setting
\beno
\bar{z}_\pm^{h(\delta,\delta')}\eqdefa\bar{z}^h_{\pm(\delta)}-\bar{z}^h_{\pm(\delta')},
\quad p^{(\delta,\delta')}\eqdefa M_1p_{(\delta)}-M_1p_{(\delta')},
\eeno
we deduce from \eqref{scaling for app system} that $\bigl(\bar{z}_+^{h(\delta,\delta')},\,\bar{z}_-^{h(\delta,\delta')},\,p^{(\delta,\delta')}\bigr)$ satisfy
\beq\label{app7}
\begin{aligned}
\p_t\bar{z}_+^{h(\delta,\delta')}-\p_1\bar{z}_+^{h(\delta,\delta')}+\bar{z}_{-(\delta)}^h\cdot \nabla_h\bar{z}_+^{h(\delta,\delta')}&
=-\nabla_h p^{(\delta,\delta')}-\bar{z}_-^{h(\delta,\delta')}\cdot\na_h\bar{z}_{+(\delta')}^h-\bigl(R_{+(\delta)}-R_{+(\delta')}\bigr), \\
\p_t\bar{z}_-^{h(\delta,\delta')}+\p_1\bar{z}_-^{h(\delta,\delta')}+\bar{z}_{+(\delta)}^h\cdot \nabla_h\bar{z}_-^{h(\delta,\delta')}&
= -\nabla_h p^{(\delta,\delta')}-\bar{z}_+^{h(\delta,\delta')}\cdot\na_h\bar{z}_{-(\delta')}^h-(R_{-(\delta)}-R_{-(\delta')}), \\
\bar{z}_+^{h(\delta,\delta')}|_{t=0}=\bar{z}_{+(\delta),0}^h-\bar{z}_{+(\delta'),0}^h\eqdefa\bar{z}_{+,0}^{h(\delta,\delta')}&,\quad \bar{z}_-^{h(\delta,\delta')}|_{t=0}=\bar{z}^h_{-(\delta),0}-\bar{z}^h_{-(\delta'),0}\eqdefa\bar{z}_{-,0}^{h(\delta,\delta')},
\end{aligned}
\eeq
where
\beno
R_{\pm(\delta)}\eqdefa M_1\bigl(w_{\mp(\delta)}\cdot\na z_{\pm(\delta)}^h\bigr).
\eeno

We only give the estimate of $\bar{z}_+^{h(\delta,\delta')}$. Firstly, we derive the expression of pressure $p^{(\delta,\delta')}$.
To do that, taking $\na_h$ to the first equation of
\eqref{scaling for app system}, we have
\beno
-\Delta_h(M_1p_{(\delta)})=\na_h\cdot\bigl(\bar{z}_{-(\delta)}^h\cdot\na_h\bar{z}_{+(\delta)}^h\bigr)+\underbrace{M_1\bigl(\na_h\cdot(w_{-(\delta)}\cdot\na z_{+(\delta)}^h)\bigr)}_{\widetilde{R}_{(\delta)}(t,x_h)}.
\eeno
Moreover, it is easy to check that
\beno
\widetilde{R}_{(\delta)}(t,x_h)=\p_\al\p_\beta M_1\bigl(z_{+(\delta)}^\al z_{-(\delta)}^\beta\bigr)-\p_\al\p_\beta(\bar{z}_{+(\delta)}^\al\bar{z}_{-(\delta)}^\beta),
\eeno
where we used the Einstein's convection and $\alpha,\beta$ go through $\{1,2\}$.

Using the Green function on $\R^2$, we have (up to a constant)
\beq\label{app6}
M_1p_{(\delta)}(t,x_h)=-\f{1}{2\pi}\int_{\R^2}\log |x_h-y_h|\cdot\bigl(\p_\alpha\bar{z}^\beta_{-(\delta)}\p_\beta\bar{z}^\alpha_{+(\delta)}
+\widetilde{R}_{(\delta)}\bigr)(t,y_h)dy_h.
\eeq
Then we obtain that
\beq\label{app8}\begin{aligned}
p^{(\delta,\delta')}(t,x_h)=&-\f{1}{2\pi}\int_{\R^2}\log |x_h-y_h|\cdot\bigl(\p_\alpha\bar{z}_-^{\beta(\delta,\delta')}\p_\beta\bar{z}^\alpha_{+(\delta)}
\bigr)(t,y_h)dy_h\\
&-\f{1}{2\pi}\int_{\R^2}\log |x_h-y_h|\cdot\bigl(\p_\alpha\bar{z}^\beta_{-(\delta')}\p_\beta\bar{z}_+^{\alpha(\delta,\delta')}\bigr)(t,y_h)dy_h\\
&-\f{1}{2\pi}\int_{\R^2}\log |x_h-y_h|\cdot\widetilde{R}_{(\delta)}(t,y_h)dy_h+\f{1}{2\pi}\int_{\R^2}\log |x_h-y_h|\cdot\widetilde{R}_{(\delta')}(t,y_h)dy_h\\
\eqdefa&p_{11}^{(\delta,\delta')}(t,x_h)+p_{12}^{(\delta,\delta')}(t,x_h)-p_{2(\delta)}(t,x_h)+p_{2(\delta')}(t,x_h).
\end{aligned}\eeq

Now, for $\al_h\in\bigl(\Z_{\geq0}\bigr)^2$, applying $\p_h^{\al_h}$ to both sides of the first equation in \eqref{app7}, we have
\beq\label{app9}
\p_t\bigl(\p_h^{\al_h}\bar{z}_+^{h(\delta,\delta')}\bigr)-\p_1\bigl(\p_h^{\al_h}\bar{z}_+^{h(\delta,\delta')}\bigr)
+\bar{z}_{-(\delta)}^h\cdot \nabla_h\bigl(\p_h^{\al_h}\bar{z}_+^{h(\delta,\delta')}\bigr)
=-\na_h\p_h^{\al_h}p^{(\delta,\delta')}+\r_+^{(\al_h)},
\eeq
where
\beno
\r_+^{(\al_h)}=\underbrace{-\p_h^{\al_h}\bigl(\bar{z}_{-(\delta)}^h\cdot \nabla_h\bar{z}_+^{h(\delta,\delta')}\bigr)+\bar{z}_{-(\delta)}^h\cdot \nabla_h\bigl(\p_h^{\al_h}\bar{z}_+^{h(\delta,\delta')}\bigr)}_{\r_{+,1}^{(\al_h)}}
\underbrace{-\p_h^{\al_h}\bigl(\bar{z}_-^{h(\delta,\delta')}\cdot\na_h\bar{z}_{+(\delta')}^h\bigr)}_{\r_{+,2}^{(\al_h)}}
\underbrace{-\p_h^{\al_h}\bigl(R_{(\delta)}-R_{(\delta')}\bigr)}_{\r_{+,3}^{(\al_h)}}.
\eeno

Since $\div_h\,\bar{z}_{-(\delta)}^h=\na_h\cdot\bar{z}^h_{-(\delta)}=0$, applying Proposition \ref{linearized prop for 2D} to \eqref{app9}, we have
\beq\label{app10}\begin{aligned}
&\sup_{0\leq\tau\leq t}\int_{\Sigma_{\tau,h}}\langle u_-\rangle^{2(1+\sigma)}|\p_h^{\al_h}\bar{z}_+^{h(\delta,\delta')}|^2dx_h+\int_0^t\int_{\Sigma_{\tau,h}}\f{\langle u_-\rangle^{2(1+\sigma)}}{\langle u_+\rangle^{1+\sigma}}|\p_h^{\al_h}\bar{z}_+^{h(\delta,\delta')}|^2dx_hd\tau\\
&\lesssim\int_{\Sigma_{0,h}}\langle u_-\rangle^{2(1+\sigma)}|\p_h^{\al_h}\bar{z}_+^{h(\delta,\delta')}|^2dx_h+\int_0^t\int_{\Sigma_{\tau,h}}\langle u_-\rangle^{1+2\sigma}|\bar{z}_{-(\delta)}^1||\p_h^{\al_h}\bar{z}_+^{h(\delta,\delta')}|^2dx_hd\tau\\
&\qquad
+\Bigl|\int_0^t\int_{\Sigma_{\tau,h}}\na_h\p_h^{\al_h}p^{(\delta,\delta')}\cdot\langle u_-\rangle^{2(1+\sigma)}\p_h^{\al_h}\bar{z}_+^{h(\delta,\delta')}dx_hd\tau\Bigr|\\
&\qquad
+\int_{\R}\f{1}{\langle u_+\rangle^{1+\sigma}}\Bigl|\int\int_{W_{t,h}^{[u_+,\infty]}}\na_h\p_h^{\al_h}p^{(\delta,\delta')}\cdot\langle u_-\rangle^{2(1+\sigma)}\p_h^{\al_h}\bar{z}_+^{h(\delta,\delta')} dx_hd\tau\Bigr| d{u_+}\\
&\qquad
+\int_0^t\int_{\Sigma_{\tau,h}}|\r_+^{(\al_h)}|\cdot\langle u_-\rangle^{2(1+\sigma)}|\p_h^{\al_h}\bar{z}_+^{h(\delta,\delta')}|dx_hd\tau.
\end{aligned}\eeq

{\it Step 3.2. Estimates of the nonlinear terms on the r.h.s of \eqref{app10}.}

(i) For the second term on the r.h.s of \eqref{app10}, we have
\beq\label{app11}
\int_0^t\int_{\Sigma_{\tau,h}}\langle u_-\rangle^{1+2\sigma}|\bar{z}_{-(\delta)}^1||\p_h^{\al_h}\bar{z}_+^{h(\delta,\delta')}|^2dx_hd\tau\\
\lesssim\sum_{k\leq2}\bigl(E_{-,h}^{(k)}(\bar{z}_{-(\delta)}^1)\bigr)^{\f12}\cdot F_{+,h}^{(\al_h)}(\bar{z}_+^{h(\delta,\delta')}).
\eeq

(ii) For the third term involving $p^{(\delta,\delta')}$ on the r.h.s of \eqref{app10}, we only consider the case $|\al_h|\geq 1$. In fact, notice that for $|\al_h|=0$
\beno\begin{aligned}
\Bigl|\int_0^t\int_{\Sigma_{\tau,h}}\na_hp^{(\delta,\delta')}\cdot\langle u_-\rangle^{2(1+\sigma)}\bar{z}_+^{h(\delta,\delta')}dx_hd\tau\Bigr|
\leq \int_0^t\int_{\Sigma_{\tau,h}}|\na_hp^{(\delta,\delta')}|\cdot\langle u_-\rangle^{2(1+\sigma)}|\bar{z}_+^{h(\delta,\delta')}|dx_hd\tau.
\end{aligned}\eeno
Then the estimate in the case of $|\al_h|=0$ can be reduced to that in the case of $|\al_h|\geq1$.

Integrating by parts and using the condition $\div_h\,\bar{z}_+^{h(\delta,\delta')}=0$, we obtain that
\beq\label{app13}\begin{aligned}
&\quad\Bigl|\int_0^t\int_{\Sigma_{\tau,h}}\na_h\p_h^{\al_h}p^{(\delta,\delta')}\cdot\langle u_-\rangle^{2(1+\sigma)}\p_h^{\al_h}\bar{z}_+^{h(\delta,\delta')}dx_hd\tau\Bigr|\\
&=\Bigl|\int_0^t\int_{\Sigma_{\tau,h}}\p_h^{\al_h}p^{(\delta,\delta')}\cdot\p_1\langle u_-\rangle^{2(1+\sigma)}\p_h^{\al_h}\bar{z}_+^{1(\delta,\delta')}dx_hd\tau\Bigr|\\
&\lesssim\|\langle u_+\rangle^{\f12(1+\sigma)}\langle u_-\rangle^{1+\sigma}\p_h^{\al_h}p^{(\delta,\delta')}\|_{L^2_tL^2_h}\cdot
\bigl(F_{+,h}^{(\al_h)}(\bar{z}_+^{h(\delta,\delta')})\bigr)^{\f12}.
\end{aligned}\eeq
We only need to bound $\|\langle u_+\rangle^{\f12(1+\sigma)}\langle u_-\rangle^{1+\sigma}\p_h^{\al_h}p^{(\delta,\delta')}\|_{L^2_tL^2_h}$.

By the expression of $p^{(\delta,\delta')}$ in \eqref{app8}, we split $p^{(\delta,\delta')}$ into four parts.
We first deal with the term involving $p_{11}^{(\delta,\delta')}$. Let $\theta(r)$ be the cut-off function defined in Step 2 of the proof to Proposition \ref{al h prop}. Using $\div_h\,\bar{z}^h_{+(\delta)}=0$, integrating by parts, we have
\beno\begin{aligned}
\p_h^{\al_h}p_{11}^{(\delta,\delta')}(\tau,x_h)&=\f{1}{2\pi}\int_{\R^2}\p_\alpha\log |x_h-y_h|\cdot\theta(|x_h-y_h|)\p_h^{\al_h}\bigl(\bar{z}_-^{\beta(\delta,\delta')}\p_\beta\bar{z}^\alpha_{+(\delta)}\bigr)
(\tau,y_h)dy_h\\
&\quad+\f{1}{2\pi}\int_{\R^2}\p_\alpha\log |x_h-y_h|\cdot\bigl(1-\theta(|x_h-y_h|)\bigr)\p_h^{\al_h}\bigl(\bar{z}_-^{\beta(\delta,\delta')}\p_\beta\bar{z}^\alpha_{+(\delta)}\bigr)
(\tau,y_h)dy_h\\
&\eqdefa I_1(\tau,x_h)+I_2(\tau,x_h).
\end{aligned}\eeno

For $I_1(\tau,x_h)$, using \eqref{x-y<2}, we have
\beno\begin{aligned}
&\quad\|\langle u_+\rangle^{\f12(1+\sigma)}\langle u_-\rangle^{1+\sigma}I_1\|_{L^2_tL^2_h}\\
&
\lesssim\|\int_{|x_h-y_h|\leq2}\f{1}{|x_h-y_h|}\Bigl(\langle u_+\rangle^{\f12(1+\sigma)}\langle u_-\rangle^{1+\sigma}\bigl|\p_h^{\al_h}\bigl(\bar{z}_-^{\beta(\delta,\delta')}\p_\beta\bar{z}^\alpha_{+(\delta)}
\bigr)\bigr|\Bigr)(\tau,y_h)d{y_h}\|_{L^2_tL^2_h}\\
&\stackrel{\text{Young}}{\lesssim}\|\f{1}{|x_h|}\|_{L^1(|x_h|\leq2)}\|\langle u_+\rangle^{\f12(1+\sigma)}\langle u_-\rangle^{1+\sigma}\p_h^{\al_h}\bigl(\bar{z}_-^{\beta(\delta,\delta')}\p_\beta\bar{z}^\alpha_{+(\delta)}
\bigr)\|_{L^2_tL^2_h}\\
&\lesssim\sum_{\beta_h\leq\al_h}\|\langle u_+\rangle^{1+\sigma}\p_h^{\al_h-\beta_h}\bar{z}_-^{\beta(\delta,\delta')}\|_{L^\infty_tL^2_h}
\|\f{\langle u_-\rangle^{1+\sigma}}{\langle u_+\rangle^{\f12(1+\sigma)}}\p_h^{\beta_h}\p_\beta\bar{z}^\alpha_{+(\delta)}\|_{L^2_tL^\infty_h}.
\end{aligned}\eeno
Then we obtain
\beq\label{app14}
\|\langle u_+\rangle^{\f12(1+\sigma)}\langle u_-\rangle^{1+\sigma}I_1\|_{L^2_tL^2_h}
\lesssim\sum_{k\leq|\al_h|}\bigl(E_{-,h}^{(k)}(\bar{z}_-^{h(\delta,\delta')})\bigr)^{\f12}
\cdot\sum_{k\leq|\al_h|+3}\bigl(F_{+,h}^{(k)}(\bar{z}^h_{+(\delta)})\bigr)^{\f12}.
\eeq

For $I_2(\tau,x_h)$, using the condition $\div_h\,\bar{z}_-^{h(\delta,\delta')}=0$ and integrating by parts, we have
\beno\begin{aligned}
I_2(\tau,x_h)=&\f{1}{2\pi}\int_{\R^2}\p_h^{\gamma_h}\p_\beta\p_\alpha\log |x_h-y_h|\cdot\bigl(1-\theta(|x_h-y_h|)\bigr)\p_h^{\al_h-\gamma_h}\bigl(\bar{z}_-^{\beta(\delta,\delta')}\bar{z}^\alpha_{+(\delta)}
(\tau,y_h)\bigr)dy_h\\
&-\f{1}{2\pi}\int_{\R^2}\bigl(\p_h^{\gamma_h}\p_\alpha\log |x_h-y_h|\cdot\p_\beta\theta(|x_h-y_h|)+\p_\beta\p_\alpha\log |x_h-y_h|\cdot\p_h^{\gamma_h}\theta(|x_h-y_h|)\\
&\quad+\p_\alpha\log |x_h-y_h|\cdot\p_h^{\gamma_h}\p_\beta\theta(|x_h-y_h|)\bigr)\p_h^{\al_h-\gamma_h}
\bigl(\bar{z}_-^{\beta(\delta,\delta')}\bar{z}^\alpha_{+(\delta)}\bigr)
(\tau,y_h)dy_h\\
\eqdefa&I_{21}(\tau,x_h)+I_{22}(\tau,x_h),
\end{aligned}\eeno
where $\gamma_h\leq\al_h$ and $|\gamma_h|=1$.

For $I_{22}(\tau,x_h)$, we have
\beno
|I_{22}|\lesssim\int_{1\leq|x_h-y_h|\leq2}\f{1}{|x_h-y_h|}\cdot
|\p_h^{\al_h-\gamma_h}\bigl(\bar{z}_-^{\beta(\delta,\delta')}\bar{z}^\alpha_{+(\delta)}\bigr)
(\tau,y_h)|dy_h.
\eeno
By the similarly derivation as  that for $I_1$, we have
\beq\label{app15}
\|\langle u_+\rangle^{\f12(1+\sigma)}\langle u_-\rangle^{1+\sigma}I_{22}\|_{L^2_tL^2_h}
\lesssim\sum_{k\leq|\al_h|-1}\bigl(E_{-,h}^{(k)}(\bar{z}_-^{h(\delta,\delta')})\bigr)^{\f12}
\cdot\sum_{k\leq|\al_h|+1}\bigl(F_{+,h}^{(k)}(\bar{z}^h_{+(\delta)})\bigr)^{\f12}.
\eeq

For $I_{21}(\tau,x_h)$, using \eqref{x-y>1}, we have
\beno\begin{aligned}
&\quad\|\langle u_+\rangle^{\f12(1+\sigma)}\langle u_-\rangle^{1+\sigma}I_{21}\|_{L^2_tL^2_h}\\
&
\lesssim\|\int_{|x_h-y_h|\geq 1}\f{1}{|x_h-y_h|^{\f32(1-\sigma)}}\Bigl(\langle u_+\rangle^{\f12(1+\sigma)}\langle u_-\rangle^{1+\sigma}\bigl|\p_h^{\al_h-\gamma_h}\bigl(\bar{z}_-^{\beta(\delta,\delta')}\bar{z}^\alpha_{+(\delta)}
\bigr)\bigr|\Bigr)(\tau,y_h)d{y_h}\|_{L^2_tL^2_h}\\
&\stackrel{\text{Young}}{\lesssim}\|\f{1}{|x_h|^{\f32(1-\sigma)}}\|_{L^2(|x_h|\geq1)}\|\langle u_+\rangle^{\f12(1+\sigma)}\langle u_-\rangle^{1+\sigma}\p_h^{\al_h-\gamma_h}\bigl(\bar{z}_-^{\beta(\delta,\delta')}\bar{z}^\alpha_{+(\delta)}
\bigr)\|_{L^2_tL^1_h}.
\end{aligned}\eeno
Since $\sigma\in(0,\f13)$, we have $\|\f{1}{|x_h|^{\f32(1-\sigma)}}\|_{L^2(|x_h|\geq1)}<\infty$. Then we obtain
\beq\label{app16}
\|\langle u_+\rangle^{\f12(1+\sigma)}\langle u_-\rangle^{1+\sigma}I_{21}\|_{L^2_tL^2_h}
\lesssim\sum_{k\leq|\al_h|-1}\bigl(E_{-,h}^{(k)}(\bar{z}_-^{h(\delta,\delta')})\bigr)^{\f12}
\cdot\sum_{k\leq|\al_h|-1}\bigl(F_{+,h}^{(k)}(\bar{z}^h_{+(\delta)})\bigr)^{\f12}.
\eeq

Thanks to \eqref{app14}, \eqref{app15} and \eqref{app16}, we have
\beno
\|\langle u_+\rangle^{\f12(1+\sigma)}\langle u_-\rangle^{1+\sigma}\p_h^{\al_h}p_{11}^{(\delta,\delta')}\|_{L^2_tL^2_h}
\lesssim\sum_{k\leq|\al_h|}\bigl(E_{-,h}^{(k)}(\bar{z}_-^{h(\delta,\delta')})\bigr)^{\f12}
\cdot\sum_{k\leq|\al_h|+3}\bigl(F_{+,h}^{(k)}(\bar{z}^h_{+(\delta)})\bigr)^{\f12}
\eeno
Since $\bar{z}^h_{+(\delta)}=\f12\int_{-1}^1z^h_{+(\delta)}(\cdot,x_3)dx_3$, we have
\beno
\bigl(F_{+,h}^{(k)}(\bar{z}^h_{+(\delta)})\bigr)^{\f12}\leq\bigl(F_+^{(k)}(z^h_{+(\delta)})\bigr)^{\f12}.
\eeno
Then we have
\beq\label{app17}
\|\langle u_+\rangle^{\f12(1+\sigma)}\langle u_-\rangle^{1+\sigma}\p_h^{\al_h}p_{11}^{(\delta,\delta')}\|_{L^2_tL^2_h}
\lesssim\sum_{k\leq|\al_h|}\bigl(E_{-,h}^{(k)}(\bar{z}_-^{h(\delta,\delta')})\bigr)^{\f12}
\cdot\sum_{k\leq|\al_h|+3}\bigl(F_+^{(k)}(z^h_{+(\delta)})\bigr)^{\f12}.
\eeq
Similarly, we have
\beq\label{app18}
\|\langle u_+\rangle^{\f12(1+\sigma)}\langle u_-\rangle^{1+\sigma}\p_h^{\al_h}p_{12}^{(\delta,\delta')}\|_{L^2_tL^2_h}
\lesssim\sum_{k\leq|\al_h|+3}\bigl(E_{-}^{(k)}(z^h_{-(\delta')})\bigr)^{\f12}
\cdot\sum_{k\leq|\al_h|}\bigl(F_{+,h}^{(k)}(\bar{z}_+^{h(\delta,\delta')})\bigr)^{\f12}.
\eeq

For term involving $p_{2(\delta)}$ and $p_{2(\delta')}$, since $\div\,w_{-(\delta)}=0$ and $w_{-(\delta)}^3|_{x_3=\pm1}=0$, we have
\beq\label{app21}
\widetilde{R}_{(\delta)}=\na_h\cdot\f{1}{2}\int_{-1}^1\sum_{i=1}^3\p_i(w_{-(\delta)}^iz_{+(\delta)}^h)(\cdot,x_3)dx_3=\na_h\cdot\sum_{\al=1}^2\p_\al\bigl(M_1(w_{-(\delta)}^\al z_{+(\delta)}^h)\bigr).
\eeq
Following the similar argument applied to \eqref{app17} and using Sobolev inequality,  we have
\beq\label{app20}\begin{aligned}
&\quad\|\langle u_+\rangle^{\f12(1+\sigma)}\langle u_-\rangle^{1+\sigma}\p_h^{\al_h}p_{2(\delta)}\|_{L^2_tL^2_h}\\
&
\lesssim\sum_{k\leq|\al_h|}\sup_{x_3\in(-1,1)}\bigl(E_{-,h}^{(k)}(w^h_{-(\delta)}(\cdot,x_3))\bigr)^{\f12}
\cdot\sum_{k\leq|\al_h|+4}\bigl(F_{+}^{(k)}(z^h_{+(\delta)})\bigr)^{\f12}.
\end{aligned}\eeq
Similar estimate holds for $\|\langle u_+\rangle^{\f12(1+\sigma)}\langle u_-\rangle^{1+\sigma}\p_h^{\al_h}p_{2(\delta')}\|_{L^2_tL^2_h}.$

Due to \eqref{app17}, \eqref{app18} and \eqref{app20}, we obtain the estimate of $\|\langle u_+\rangle^{\f12(1+\sigma)}\langle u_-\rangle^{1+\sigma}\p_h^{\al_h}p^{(\delta,\delta')}\|_{L^2_tL^2_h}$ for $|\al_h|\geq1$. Similar estimate holds for $|\al_h|=0$.  Then using \eqref{app13},  for $|\al_h|\geq0$, we have
\beq\label{app19}\begin{aligned}
&\quad\Bigl|\int_0^t\int_{\Sigma_{\tau,h}}\na_h\p_h^{\al_h}p^{(\delta,\delta')}\cdot\langle u_-\rangle^{2(1+\sigma)}\p_h^{\al_h}\bar{z}_+^{h(\delta,\delta')}dx_hd\tau\Bigr|\\
&\lesssim\Bigl(\sum_{k\leq\max\{|\al_h|,1\}}\bigl(E_{-,h}^{(k)}(\bar{z}_-^{h(\delta,\delta')})\bigr)^{\f12}
\cdot\sum_{k\leq\max\{|\al_h|,1\}+4}\bigl(F_{+}^{(k)}(z^h_{+(\delta)})\bigr)^{\f12}\\
&\quad
+\sum_{k\leq\max\{|\al_h|,1\}+4}\bigl(E_{-}^{(k)}(z^h_{-(\delta')})\bigr)^{\f12}
\cdot\sum_{k\leq\max\{|\al_h|,1\}}\bigl(F_{+,h}^{(k)}(\bar{z}_+^{h(\delta,\delta')})\bigr)^{\f12}\\
&\quad+\sum_{\tilde{\delta}=\delta,\delta'}\bigl[\sum_{k\leq\max\{|\al_h|,1\}}\sup_{x_3\in(-1,1)}\bigl(E_{-,h}^{(k)}(w^h_{-(\tilde\delta)}(\cdot,x_3))\bigr)^{\f12}
\cdot\sum_{k\leq\max\{|\al_h|,1\}+4}\bigl(F_{+}^{(k)}(z^h_{+(\tilde\delta)})\bigr)^{\f12}\bigr]\Bigr)
\cdot\bigl(F_{+,h}^{(\al_h)}(\bar{z}_+^{h(\delta,\delta')})\bigr)^{\f12}.
\end{aligned}\eeq

(iii) For the fourth term on the r.h.s of \eqref{app10}, integrating by parts,  for $|\al_h|\geq1$, we have
\beno\begin{aligned}
&\int\int_{W_{t,h}^{[u_+,\infty]}}\na_h\p_h^{\al_h}p^{(\delta,\delta')}\cdot\langle u_-\rangle^{2(1+\sigma)}\p_h^{\al_h}\bar{z}_+^{h(\delta,\delta')} dx_hd\tau\\
&=-\int\int_{W_{t,h}^{[u_+,\infty]}}\p_h^{\al_h}p^{(\delta,\delta')}\cdot\p_1\langle u_-\rangle^{2(1+\sigma)}\p_h^{\al_h}\bar{z}_+^{1(\delta,\delta')} dx_hd\tau
-\int_{C_{u_+}}\p_h^{\al_h}p^{(\delta,\delta')}\cdot\langle u_-\rangle^{2(1+\sigma)}\p_h^{\al_h}\bar{z}_+^{1(\delta,\delta')} d\sigma_+.
\end{aligned}\eeno
Following the similar argument applied in Step 2.2 of the proof of Proposition \ref{al h prop}, for $|\al_h|\geq1$, we have
\beno\begin{aligned}
&\int_{\R}\f{1}{\langle u_+\rangle^{1+\sigma}}\Bigl|\int\int_{W_{t,h}^{[u_+,\infty]}}\na_h\p_h^{\al_h}p^{(\delta,\delta')}\cdot\langle u_-\rangle^{2(1+\sigma)}\p_h^{\al_h}\bar{z}_+^{h(\delta,\delta')} dx_hd\tau\Bigr| d{u_+}\\
&\lesssim\int\int_{W_{t,h}}|\p_h^{\al_h}p^{(\delta,\delta')}|\cdot\langle u_-\rangle^{2(1+\sigma)}|\p_h^{\al_h}\bar{z}_+^{1(\delta,\delta')}|dx_hd\tau.
\end{aligned}\eeno
Similar estimate holds for $|\al_h|=0$. Then  for $|\al_h|\geq0$, we have
\beq\label{app22}
\int_{\R}\f{1}{\langle u_+\rangle^{1+\sigma}}\Bigl|\int\int_{W_{t,h}^{[u_+,\infty]}}\na_h\p_h^{\al_h}p^{(\delta,\delta')}\cdot\langle u_-\rangle^{2(1+\sigma)}\p_h^{\al_h}\bar{z}_+^{h(\delta,\delta')} dx_hd\tau\Bigr| d{u_+}
\lesssim\text{r.h.s of \eqref{app19}}.
\eeq

(iv) For the last term on the r.h.s of \eqref{app10}, we have
\beq\label{app23}
\int_0^t\int_{\Sigma_{\tau,h}}|\r_+^{(\al_h)}|\cdot\langle u_-\rangle^{2(1+\sigma)}|\p_h^{\al_h}\bar{z}_+^{h(\delta,\delta')}|dx_hd\tau
\lesssim\|\langle u_+\rangle^{\f12(1+\sigma)}\langle u_-\rangle^{1+\sigma}\r_+^{(\al_h)}\|_{L^2_tL^2_h}\cdot\bigl(F_{+,h}^{(\al_h)}(\bar{z}_+^{h(\delta,\delta')})\bigr)^{\f12}.
\eeq
We only need to estimate $\|\langle u_+\rangle^{\f12(1+\sigma)}\langle u_-\rangle^{1+\sigma}\r_+^{(\al_h)}\|_{L^2_tL^2_h}$.

For $\r_{+,1}^{(\al_h)}$ and $\r_{+,2}^{(\al_h)}$, we  have
\beq\label{app24}
\|\langle u_+\rangle^{\f12(1+\sigma)}\langle u_-\rangle^{1+\sigma}\r_{+,1}^{(\al_h)}\|_{L^2_tL^2_h}
\lesssim\sum_{k\leq|\al_h|+2}\bigl(E_{-,h}^{(k)}(\bar{z}^h_{-(\delta)})\bigr)^{\f12}
\cdot\sum_{k\leq|\al_h|}\bigl(F_{+,h}^{(k)}(\bar{z}_+^{h(\delta,\delta')})\bigr)^{\f12},
\eeq
\beq\label{app25}
\|\langle u_+\rangle^{\f12(1+\sigma)}\langle u_-\rangle^{1+\sigma}\r_{+,2}^{(\al_h)}\|_{L^2_tL^2_h}
\lesssim\sum_{k\leq|\al_h|}\bigl(E_{-,h}^{(k)}(\bar{z}_-^{h(\delta,\delta')})\bigr)^{\f12}\cdot\sum_{k\leq|\al_h|+3}\bigl(F_{+,h}^{(k)}(\bar{z}_{+(\delta')}^{h})\bigr)^{\f12}.
\eeq

For $\r_{+,3}^{(\al_h)}$, using conditions that $\div\,w_{-(\delta)}=0$, $w_{-(\delta)}^3|_{x_3=\pm1}=0$ and $z_{+(\delta)}^h=\bar{z}_{+(\delta)}^h+w_{+(\delta)}^h$, and integrating by parts, we have
\beno\begin{aligned}
R_{+(\delta)}&=M_1\bigl(w_{-(\delta)}^h\cdot\na_hz_{+(\delta)}^h\bigr)+\f12\int_{-1}^1w_{-(\delta)}^3\p_3w_{+(\delta)}^hdx_3\\
&=M_1\bigl(w_{-(\delta)}^h\cdot\na_hz_{+(\delta)}^h+\na_h\cdot w_{-(\delta)}^hw_{+(\delta)}^h\bigr).
\end{aligned}\eeno
Using Sobolev inequality, we also have
\beq\label{app26}\begin{aligned}
&\|\langle u_+\rangle^{\f12(1+\sigma)}\langle u_-\rangle^{1+\sigma}\r_{+,3}^{(\al_h)}\|_{L^2_tL^2_h}
\lesssim\sum_{\tilde{\delta}=\delta,\delta'}\Bigl(\sum_{k\leq|\al_h|}\sup_{x_3\in(-1,1)}\bigl(E_{-,h}^{(k)}(w_{-(\tilde\delta)}^h(\cdot,x_3))\bigr)^{\f12}
\cdot\sum_{k\leq|\al_h|+3}\bigl(F_{+}^{(k)}(z_{+(\tilde\delta)}^{h})\bigr)^{\f12}\\
&\quad+\sum_{k\leq|\al_h|+3}\sup_{x_3\in(-1,1)}\bigl(E_{-,h}^{(k)}(w_{-(\tilde\delta)}^{h}(\cdot,x_3))\bigr)^{\f12}
\cdot\sum_{k\leq|\al_h|}\sup_{x_3\in(-1,1)}\bigl(F_{+,h}^{(k)}(w_{+(\tilde\delta)}^h(\cdot,x_3))\bigr)^{\f12}\Bigr).
\end{aligned}\eeq

Noticing that  $\bar{z}_{\pm(\delta)}^h=\f12\int_{-1}^1z_{\pm(\delta)}^h(\cdot,x_3)dx_3$ and $w_{\pm(\delta)}^h=z_{\pm(\delta)}^h-\bar{z}_{\pm(\delta)}^h$. Then it   holds
\beq\label{app27}\begin{aligned}
&\bigl(E_{\pm,h}^{(k)}(\bar{z}_{\pm(\delta)}^{h})\bigr)^{\f12}\lesssim\bigl(E_{\pm}^{(k)}(z_{\pm(\delta)}^{h})\bigr)^{\f12},\quad
\bigl(F_{\pm,h}^{(k)}(w_{\pm(\delta)}^{h}(\cdot,x_3))\bigr)^{\f12}\lesssim\bigl(F_{\pm}^{(k)}(z_{\pm(\delta)}^{h})\bigr)^{\f12}.
\end{aligned}\eeq
Similar estimates hold for $\bigl(F_{\pm,h}^{(k)}(\bar{z}_{\pm(\delta)}^h)\bigr)^{\f12}$ and $\bigl(E_{\pm,h}^{(k)}(w_{\pm(\delta)}^h(\cdot,x_3))\bigr)^{\f12}$.

Thanks to \eqref{app23}, \eqref{app24}, \eqref{app25}, \eqref{app26} and \eqref{app27}, we have
\beq\label{app28}\begin{aligned}
&\int_0^t\int_{\Sigma_{\tau,h}}|\r_+^{(\al_h)}|\cdot\langle u_-\rangle^{2(1+\sigma)}|\p_h^{\al_h}\bar{z}_+^{h(\delta,\delta')}|dx_hd\tau\\
&\lesssim\Bigl(\sum_{k\leq|\al_h|+3}\bigl(E_{-}^{(k)}(z^h_{-(\delta)})\bigr)^{\f12}\cdot\sum_{k\leq|\al_h|}\bigl(F_{+,h}^{(k)}(\bar{z}_+^{h(\delta,\delta')})\bigr)^{\f12}
+\sum_{k\leq|\al_h|}\bigl(E_{-,h}^{(k)}(\bar{z}_-^{h(\delta,\delta')})\bigr)^{\f12}\cdot\sum_{k\leq|\al_h|+3}\bigl(F_{+}^{(k)}(z_{+(\delta')}^{h})\bigr)^{\f12}\\
&\quad+\sum_{\tilde{\delta}=\delta,\delta'}\Bigl[\sum_{k\leq|\al_h|}\sup_{x_3\in(-1,1)}\bigl(E_{-,h}^{(k)}(w_{-(\tilde\delta)}^h(\cdot,x_3))\bigr)^{\f12}\cdot\sum_{k\leq|\al_h|+3}\bigl(F_{+}^{(k)}(z_{+(\tilde\delta)}^{h})\bigr)^{\f12}\\
&\quad+\sum_{k\leq|\al_h|}\sup_{x_3\in(-1,1)}\bigl(F_{+,h}^{(k)}(w_{+(\tilde\delta)}^h(\cdot,x_3))\bigr)^{\f12}\cdot\sum_{k\leq|\al_h|+3}\bigl(E_{-}^{(k)}(z_{-(\tilde\delta)}^{h})\bigr)^{\f12}\Bigr]\Bigr)
\cdot\bigl(F_{+,h}^{(\al_h)}(\bar{z}_+^{h(\delta,\delta')})\bigr)^{\f12}.
\end{aligned}\eeq

{\it Step 3.3. Convergence of the sequence $\{\bar{z}_{\pm(\delta)}^h\}_{0<\delta\leq1}$.} Thanks to \eqref{app10}, \eqref{app11}, \eqref{app19}, \eqref{app22} and \eqref{app28}, using \eqref{app27} and \eqref{app5}, we obtain that
\beq\label{app29}\begin{aligned}
&\quad\sum_{k\leq N}\Bigl(E_{+,h}^{(k)}(\bar{z}_+^{h(\delta,\delta')})+F_{+,h}^{(k)}(\bar{z}_+^{h(\delta,\delta')})\Bigr)\\
&\leq C_1\sum_{k\leq N}E_{+,h}^{(k)}(\bar{z}_{+,0}^{h(\delta,\delta')})+C_1\varepsilon_1\sum_{k\leq N}\Bigl(E_{-,h}^{(k)}(\bar{z}_-^{h(\delta,\delta')})+F_{+,h}^{(k)}(\bar{z}_+^{h(\delta,\delta')})\Bigr)\\
&\quad+C_1\varepsilon_1\sum_{\tilde{\delta}=\delta,\delta'}\sum_{k\leq N}\sup_{x_3\in(-1,1)}\bigl[\bigl(E_{-,h}^{(k)}(w_{-(\tilde\delta)}^h(\cdot,x_3))\bigr)^{\f12}+\bigl(F_{+,h}^{(k)}(w_{+(\tilde\delta)}^h(\cdot,x_3))\bigr)^{\f12}\bigr]\cdot\sum_{k\leq N}\bigl(F_{+,h}^{(k)}(\bar{z}_+^{h(\delta,\delta')})\bigr)^{\f12}.
\end{aligned}\eeq
Similar estimate holds for $\bar{z}_-^{h(\delta,\delta')}$.

Taking $\varepsilon_1$ sufficiently small, using H\"older inequality, we deduce from \eqref{app29} that
\beq\label{app30}\begin{aligned}
&\quad\sum_{+,-}\sum_{k\leq N}\Bigl(E_{\pm,h}^{(k)}(\bar{z}_\pm^{h(\delta,\delta')})+F_{\pm,h}^{(k)}(\bar{z}_\pm^{h(\delta,\delta')})\Bigr)
\leq 2C_1\sum_{+,-}\sum_{k\leq N}E_{\pm,h}^{(k)}(\bar{z}_{\pm,0}^{h(\delta,\delta')})\\
&\qquad+C_2\varepsilon_1^2\sum_{+,-}\sum_{\tilde{\delta}=\delta,\delta'}\sum_{k\leq N}\sup_{x_3\in(-1,1)}\Bigl(E_{\pm,h}^{(k)}(w_{\pm(\tilde\delta)}^h(\cdot,x_3))+F_{\pm,h}^{(k)}(w_{\pm(\tilde\delta)}^h(\cdot,x_3))\Bigr).
\end{aligned}\eeq
Thanks to \eqref{lim4} and \eqref{app30}, we have
\beq\label{lim5}\begin{aligned}
&\sum_{+,-}\sum_{k\leq N}\Bigl(E_{\pm,h}^{(k)}(\bar{z}_\pm^{h(\delta,\delta')})+F_{\pm,h}^{(k)}(\bar{z}_\pm^{h(\delta,\delta')})\Bigr)
\leq 2C_1\sum_{+,-}\sum_{k\leq N}E_{\pm,h}^{(k)}(\bar{z}_{\pm,0}^{h(\delta,\delta')})\\
&\quad+C_3\varepsilon_1^2\sum_{+,-}\sum_{\tilde{\delta}=\delta,\delta'}\sum_{k\leq N}\Bigl(\sup_{x_3\in(-1,1)}E_{\pm,h}^{(k)}\bigl(z_{\pm(\tilde\delta),0}^h(\cdot,x_3)-z_{\pm(0),0}^h(\cdot)\bigr)+{\tilde\delta}^2\varepsilon_1^4\Bigr).
\end{aligned}\eeq
By virtue of \eqref{lim ansatz}, \eqref{lim3} and \eqref{lim5}, we obtain that
 $\{\bar{z}_{\pm(\delta)}^h(t,x_h)\}_{0<\delta\leq1}$ is a Cauchy sequence in $H^N(\R^2)$ for $t>0$.

\medskip

{\bf Step 4. Limit  of the sequence $\{\bar{z}_{\pm(\delta)}^h\}_{0<\delta\leq1}$.} Since $\{\bar{z}_{\pm(\delta)}^h(t,x_h)\}_{0<\delta\leq1}$ is a  Cauchy sequence in $H^N(\R^2)$, there exists a unique $z_{\pm(0)}^h(t,x_h)$ such that
\beq\label{lim6}
\lim_{\delta\rightarrow 0}\bar{z}_{\pm(\delta)}^h(t,x_h)=z_{\pm(0)}^h(t,x_h),\quad\text{in}\quad H^N(\R^2).
\eeq

Taking $\delta'\rightarrow 0$ in \eqref{lim5}, using \eqref{lim1} and \eqref{lim3}, we obtain that
\beq\label{lim7}\begin{aligned}
&\sum_{+,-}\sum_{k\leq N}\Bigl(E_{\pm,h}^{(k)}(\bar{z}_{\pm(\delta)}^h-z_{\pm(0)}^h)+F_{\pm,h}^{(k)}(\bar{z}_{\pm(\delta)}^h-z_{\pm(0)}^h)\Bigr)\\
&
\leq C_4\sum_{+,-}\sum_{k\leq N}\sup_{x_3\in(-1,1)}E_{\pm,h}^{(k)}\bigl(z_{\pm(\delta),0}^h(\cdot,x_3)-z_{\pm(0),0}^h(\cdot)\bigr)+C_4\delta^2\varepsilon_1^6.
\end{aligned}\eeq
By virtue of \eqref{lim ansatz} and \eqref{lim7} and by using Sobolev inequality, we arrive at \eqref{lim2}.

\medskip

{\bf Step 5. Derivation of the limit system.} For system \eqref{scaling for app system} involving $\bar{z}^h_{\pm(\delta)}$, we shall take $\delta\rightarrow 0$ and derive the limit system.
From \eqref{scaling for app system}, we first have
\beno
\p_t\bar{z}_{\pm(\delta)}^h=\pm\p_1\bar{z}_{\pm(\delta)}^h-\bar{z}_{\mp(\delta)}^h\cdot \nabla_h\bar{z}_{\pm(\delta)}^h-\nabla_h\bigl(M_1p_{(\delta)}\bigr)-M_1\bigl(w_{-(\delta)}\cdot\na z_{+(\delta)}^h\bigr)
\eeno

Thanks to \eqref{app6} and the derivation of $\|\langle u_+\rangle^{\f12(1+\sigma)}\langle u_-\rangle^{1+\sigma}\p_h^{\al_h}p^{(\delta,\delta')}\|_{L^2_tL^2_h}$ (see \eqref{app17}, \eqref{app18} and \eqref{app20}),  by
\eqref{lim1}, \eqref{lim2} and \eqref{app5},  we obtain that
\beno
\{\na_h(M_1p_{(\delta)})\}_{0<\delta\leq1}\quad\text{is a  Cauchy sequence in }\ \ H^{N-1}(\R^2).
\eeno
Then we have
\beno
\{\p_t\bar{z}_{\pm(\delta)}^h\}_{0<\delta\leq1}\quad\text{is a  Cauchy sequence in }\ \ H^{N-1}(\R^2).
\eeno
Moreover, by virtue of \eqref{app5}, we have
\beq\label{app34}
\sum_{k\leq N-1}\bigl(E_{\pm,h}^{(k)}(\p_t\bar{z}_{\pm(\delta)}^h)+F_{\pm,h}^{(k)}(\p_t\bar{z}_{\pm(\delta)}^h)\bigr)\lesssim\varepsilon_1^2.
\eeq
Then by the uniqueness of limit, we have
\beq\label{app35}
\lim_{\delta\rightarrow 0}\p_t\bar{z}_{\pm(\delta)}^h=\p_tz_{\pm(0)}^h,\quad\text{in}\ \ H^{N-1}(\R^2).
\eeq

Due to \eqref{lim1}, \eqref{lim2} and \eqref{app35}, taking $\delta\rightarrow 0$, we deduce from \eqref{scaling for app system} that in the sense of $H^{N-1}(\R^2)$, $\bigl(z_{+(0)}^h(t,x_h),z_{-(0)}^h(t,x_h)\bigr)$ satisfy
\beq\label{2D MHD}
\begin{aligned}
\p_tz_{+(0)}^h-\p_1z_{+(0)}^h+z_{-(0)}^h\cdot \nabla_hz_{+(0)}^h& = -\nabla_h p_{(0)},  \ \ \text{in}\ \ \R^2\times\R^+\\
\p_tz_{-(0)}^h+\p_1z_{-(0)}^h+z_{+(0)}^h\cdot \nabla_hz_{-(0)}^h& = -\nabla_h p_{(0)}, \\
\na_h\cdot z_{+(0)}^h =0,\quad \na_h\cdot z_{-(0)}^h &=0,
\end{aligned}
\eeq
where
\beno
p_{(0)}(t,x_h)=-\f{1}{2\pi}\int_{\R^2}\log |x_h-y_h|\cdot\bigl(\p_\al z_{\mp(0)}^\beta\p_\beta z_{\pm(0)}^\al\bigr)(t,y_h)dy_h.
\eeno
We remark that \eqref{2D MHD} is exactly the two dimensional ideal MHD system.

Now we only need to prove that $z_{\pm(0)}^h$ is continuous  with respect to $t$ and verifies the initial data
\beq\label{app36}
z_{\pm(0)}^h|_{t=0}=z_{\pm(0),0}^h.
\eeq
Thanks to \eqref{app5}, \eqref{lim2}, \eqref{app34} and \eqref{app35}, we have
\beno
\sum_{k\leq N}\bigl(E_{\pm,h}^{(k)}(z_{\pm(0)}^h)+F_{\pm,h}^{(k)}(z_{\pm(0)}^h)\bigr)+\sum_{k\leq N-1}\bigl(E_{\pm,h}^{(k)}(\p_tz_{\pm(0)}^h)+F_{\pm,h}^{(k)}(\p_tz_{\pm(0)}^h)\bigr)\lesssim\varepsilon_1^2.
\eeno
We obtain that
\beno
z_{\pm(0)}^h\in C([0,\infty); H^{N-1}(\R^2)),
\eeno
and then   \eqref{app36} holds.

It remains to derive the {\it a priori} estimate \eqref{energy estimate for 2D MHD}  for the limit system \eqref{2D MHD}-\eqref{app36}. Setting $\bar{z}_{\pm(\delta')}^h=0$ in Step 2, by virtue of \eqref{app30}, \eqref{lim1} and the assumption of the initial data, we obtain that
\beno
\sum_{+,-}\sum_{k\leq N}\bigl(E_{\pm,h}^{(k)}(z_{\pm(0)}^h)+F_{\pm,h}^{(k)}(z_{\pm(0)}^h)\bigr)\leq C\sum_{+,-}\sum_{k\leq N}E_{\pm,h}^{(k)}(z_{\pm(0),0}^h).
\eeno
This is exactly \eqref{energy estimate for 2D MHD}.

\medskip

{\bf Step 6. Asymptotics from 3D MHD to 2D MHD.} Thanks to \eqref{lim4} and \eqref{lim7}, we have
\beq\label{lim8}\begin{aligned}
&\sum_{+,-}\sum_{k\leq N}\sup_{x_3\in(-1,1)}\Bigl(E_{\pm,h}^{(k)}(z_{\pm(\delta)}^h(\cdot,x_3)-z_{\pm(0)}^h(\cdot))
+F_{\pm,h}^{(k)}(z_{\pm(\delta)}^h(\cdot,x_3)-z_{\pm(0)}^h(\cdot))\Bigr)\\
&\quad+\sum_{+,-}\sum_{k\leq N-1}\sup_{x_3\in(-1,1)}\bigl(E_{\pm,h}^{(k)}(z_{\pm(\delta)}^3(\cdot,x_3))+F_{\pm,h}^{(k)}(z_{\pm(\delta)}^3(\cdot,x_3))\bigr)\Bigr)\\
&
\leq C_5\sum_{+,-}\sum_{k\leq N}\sup_{x_3\in(-1,1)}E_{\pm,h}^{(k)}\bigl(z_{\pm(\delta),0}^h(\cdot,x_3)-z_{\pm(0),0}^h(\cdot)\bigr)+C_5\delta^2\varepsilon_1^4.
\end{aligned}\eeq
By virtue of \eqref{lim ansatz}, \eqref{lim8} and   Sobolev inequality, we get \eqref{limitation 1}.
Then Theorem \ref{limitation} is eventually proved.

%
%
%
%
%
%

\vspace{0.3cm}
\noindent {\bf Acknowledgments.}
The author  is partially supported by NSF of China under grant 11201455, 11671383 and by innovation grant from National Center for Mathematics  and Interdisciplinary  Science.


\begin{thebibliography}{99}

\bibitem{Alfven}  H. Alfv\'{e}n, \textit{Existence of electromagnetic-Hydrodynamic waves},  Nature, vol. 150 (1942),  405-406.

\bibitem{Bardos}  C. Bardos, C. Sulem and P.-L. Sulem, \textit{Longtime dynamics of a conductive fluid in the presence of a strong magnetic field}, Trans. Amer. Math. Soc. 305 (1988), no. 1, 175-191.

\bibitem{c-l} Y.Cai , Z. Lei, Global Well-posedness of the Incompressible Magnetohydrodynamics, arXiv:1605.00439.

\bibitem{Ch-K} D. Christodoulou and S. Klainerman, \textit{The global nonlinear stability of Minkowski space}, Princeton Mathematical Series 41, 1993.
\bibitem{Davidson}  P. A. Davidson, \textit{An introduction to Magnetohydrodynamics}, Cambridge Texts in Applied Mathematics, 2001.
\bibitem{He-Xu-Yu} L.-B. He, L. Xu and P. Yu, \textit{On global dynamics of three dimensional magnetohydrodynamics: nonlinear stability of Alfv\'{e}n waves}, submit 2015.

\bibitem{If} D. Iftimie, \textit{The 3D Navier-Stokes equations seen as a perturbation of the 2D Navier-Stokes equations}, Bull. Soc.
Math. France 127 (1999) 473¨C517.

\bibitem{If-Ra-Sell}  D. Iftimie, G. Raugel and G. R. Sell, \textit{Navier-Stokes Equations in thin 3D domains with Navier boundary conditions}, Indiana Univ. Math. J. 56 (2007), 1083-1156.

\bibitem{Lin-Xu-Zhang}	F. Lin, L. Xu and P. Zhang, Global small solutions of 2-D incompressible MHD system. J. Differential Equations 259 (2015), no. 10, 5440-5485.


\bibitem{Lindblad-Rodnianski} H. Lindblad and I. Rodnianski, \textit{The global stability of Minkowski spac-time in harmonic guage}, Ann. of Math.(2), 171(2010), no. 3, 1401-1477.

\bibitem{Ma-Ra-Ra} J. E. Marsden, T. S. Ratiu and G. Raugel, \textit{The Euler equations on thin domains}, International Conference on Differential Equations  Vol. 1, 2 (Berlin, 1999),  1198-1203, World Sci. Publ., River Edge, NJ, 2000.

\bibitem{Mu-Ro-Ha}   Z. E. Musielak, S. Routh and R. Hammer, \textit{Cutoff-free propagation of torsoonal Alfv\'en waves along thin magnetic flux tube}, The Astrophysical Journal,  2007 April 10, 659:650-654.

\bibitem{Ra} G. Raugel, \textit{Dynamics of partial differential equations on thin domains}, CIME Course, Montecatini
Terme, Lecture Notes in Mathematics 1609, Springer Verlag, (1995), 208-315.

\bibitem{Ra-Sell} G. Raugel and G. R. Sell, \textit{Navier-Stokes equations on thin 3D domains. I. Global attractors and global
regularity of solutions}, J. Amer. Math. Soc. 6 (1993), 503¨C568.

\bibitem{Xu-Zhang} L. Xu and P. Zhang, \textit{Global Small Solutions to Three-Dimensional Incompressible Magnetohydrodynamical System}, SIAM J. Math. Anal. 47 (2015), no. 1, 26-65.

\bibitem{w-z} D.Wei, Z.Zhang, \textit{Global well-posedness of the MHD equations in a homogeneous magnetic field}, to appear in Analysis \&
PDE.

\end{thebibliography}
\end{document}